\documentclass[11pt]{amsart}

\usepackage[utf8x]{inputenc}

\usepackage{amsmath, amsthm, amsfonts, amssymb}
\usepackage{enumerate}
\usepackage{float}
\usepackage[all,cmtip]{xy}

\usepackage{chngcntr}

\usepackage{hyperref}
\usepackage[all]{xy}

\newtheorem{maintheorem}{Theorem} 
\newtheorem{theorem}{Theorem}
\newtheorem{lemma}[theorem]{Lemma}
\newtheorem{coro}[theorem]{Corollary}
\newtheorem{prop}[theorem]{Proposition}

\newtheorem*{assumptions*}{Assumptions}

\newtheorem*{rem*}{Remark}

\theoremstyle{remark}
\newtheorem{remark}[theorem]{Remark}
\newtheorem*{remark*}{Remark}

\theoremstyle{definition}
\newtheorem{definition}{Definition}

\newcommand{\F}{{\mathbf F}}

\newcommand{\R}{{\mathbf R}}

\newcommand{\W}{{\mathbf W}}

\newcommand{\Z}{{\mathbf Z}}

\newcommand{\NN}{{\mathbb N}}

\newcommand{\QQ}{{\mathbb Q}}
\newcommand{\RR}{{\mathbb R}}

\newcommand{\TT}{{\mathbb T}}

\newcommand{\ZZ}{{\mathbb Z}}

\newcommand{\SL}{{\rm SL}}
\newcommand{\GL}{{\rm GL}}

\newcommand{\be}[1]{\begin{equation} \label{#1} }
\newcommand{\ee}{\end{equation}}
\newcommand{\beq}{\begin{equation}}

\def \TT{{\mathbb T}}
\def \RR{{\mathbb R}}
\def \ZZ{{\mathbb Z}}
\def \Diff{{\rm Diff}}

\def \al{{\alpha}}

\def \W{{\mathcal W}}

\def \cD{{\mathcal{D}}}
\def \cN{{\mathcal{N}}}
\def \W{\mathcal{W}}

\def \Ho{{\mathrm{Homeo}}}

\def \bal{\bar{\alpha}}

\def \hx0{\hat{x_0}}

\def \id{\mathrm{id}}

\def \P{\mathcal{P}}
\def \Z{\mathcal{Z}}

\def \F{\mathcal{F}}

\usepackage{tikz}
\usetikzlibrary{arrows,shapes,trees}
\usepackage{graphicx}
\usepackage{amssymb}
\usepackage{amsfonts}
\usepackage{amsthm}
\usepackage{amsmath}
\usepackage{mathrsfs}

\theoremstyle{plain}
\newtheorem{quest}{Question}

\author{Danijela Damjanovi\'c}

\address[Damjanovi\'c]{Department of mathematics, Kungliga Tekniska högskolan, Lindstedtsvägen 25, SE-100 44 Stockholm, Sweden.} 
\email{ddam@kth.se}

\author{Amie Wilkinson}
\address[Wilkinson]{Department of mathematics, the University of Chicago, Chicago, IL, USA, 60637}
\email{wilkinso@math.uchicago.edu}

\author{Disheng Xu}

\address[Xu]{Department of mathematics, the University of Chicago, Chicago, IL, USA, 60637}
\email{dishengxu@math.uchicago.edu}

\begin{document}

\title[Pathology and asymmetry]{Pathology and asymmetry: centralizer rigidity for partially hyperbolic diffeomorphisms}
\maketitle

\begin{abstract}  We discover a rigidity phenomenon within the volume-preserving partially hyperbolic diffeomorphisms with $1$-dimensional center.
In particular,  for smooth, ergodic perturbations of certain algebraic systems -- including the discretized geodesic flows over hyperbolic manifolds   and certain   toral automorphisms with simple spectrum and exactly one eigenvalue on the unit circle, the smooth centralizer is  either virtually $\ZZ^\ell$ or contains a smooth flow.   

At the heart of this work are two very different rigidity phenomena.
The first  was discovered in \cite{AVW, AVW2}: for a class of volume-preserving partially hyperbolic systems including those studied here, the disintegration of volume along the center foliation is either equivalent to Lebesgue or atomic. The second phenomenon is the rigidity associated to several commuting partially hyperbolic diffeomorphisms  with very different hyperbolic behavior transverse to a common center foliation \cite{DX0}. 

We employ a variety of  techniques, among them a novel geometric approach to building new partially hyperbolic elements in hyperbolic Weyl chambers using Pesin theory and leafwise conjugacy, measure rigidity via thermodynamic formalism for circle extensions of Anosov diffeomorphisms, partially hyperbolic Liv\v sic theory, and nonstationary normal forms.
\end{abstract}

\bigskip

\centerline{\em To the memory of Anatole Katok.}
 
\tableofcontents

\section{Introduction } 

The {\em centralizer} of a diffeomorphism $f\colon M\to M$ is the set of diffeomorphisms $g$ that commute with $f$ under composition: $f\circ g = g\circ f$.  Put another way, the centralizer of $f$ is  the group of symmetries of $f$,  where ``symmetries" is meant the classical sense:  coordinate changes that leave the  dynamics of the system unchanged.  The centralizer of $f$ always contains the integer powers of $f$ and typically not more, at least conjecturally \cite{Smale1, Smale2}.  By contrast,  a diffeomorphism belonging to a smooth flow has large centralizer, containing a $1$-dimensional Lie group.

To date, the study of smooth  centralizers has mainly focused in two directions: showing that the typical map commutes only with its powers; and classifying the manifolds and/or dynamics that can support abelian centralizers of sufficiently high rank.   In this paper we aim at describing  the centralizers of \emph{all diffeomorphisms} in a small neighborhood of a given map, for specific classes of maps. This relates to one of the classical questions in perturbation theory: if a diffeomorphism belongs to a smooth flow, which perturbations also belong to a smooth flow? We answer this question fully for algebraic geodesic flows in negative curvature in conservative setting.  

More generally, we start with certain diffeomorphisms with {\em exceptionally}  large centralizer -- containing a $1$-dimensional Lie group -- and consider what happens when these diffeomorphisms are perturbed.  We find that for such perturbed systems, if the centralizer gets large enough, as measured by the  rank of its abelianization, 
 then in fact it must be exceptionally large.

To fix notation, let ${\mathcal G}$ be a group: our central example will be the space $\Diff^r(M)$ of $C^r$ diffeomorphisms of a closed manifold $M$ under composition.  For $f\in {\mathcal G}$, denote by  $\Z_{\mathcal G}(f)$  the centralizer of $f$ in $\mathcal G$:
\[\Z_{\mathcal G}(f) := \{ g\in {\mathcal G}\,:\, g f = f g \}.\]
We say that $f\in {\mathcal G}$ has {\em trivial centralizer in ${\mathcal G}$} if  the centralizer of $f$ consists of the iterates of $f$:
\[\Z_{\mathcal G}(f)  = <f>\,:= \{f^n\,:\, n\in \ZZ\} \cong \ZZ,
\]
and {\em virtually trivial centralizer}  if $\Z_{\mathcal G}(f)$ contains $<f>$ as a finite index subgroup.\footnote{If general, one says that a property holds {\em virtually} for a group $G$ if $G$ contains a finite index subgroup $H$ with that property.} 

For $f\in \Diff^r(M)$ and $M$ fixed, we will  use the shorthand notation   $\Z_r(f) : = \Z_{ \Diff^r(M)}(f)$.    If $f\in \Diff_{\mathrm{vol}}^r(M)$ is a volume-preserving element of $ \Diff^r(M)$, then we denote
   $\Z_{r,\mathrm{vol}}(f) : = \Z_{ \Diff_{\mathrm{vol}}^r(M)}(f)$.   It is not hard to see (see Lemma~\ref{lemma: g pr vol}) that if  $f\in \Diff_{\mathrm{vol}}^r(M)$ is ergodic with respect to volume, then $\Z_r(f) = \Z_{r,\mathrm{vol}}(f)$.

\subsection*{Discretized geodesic flows}

The context in which our main results are easiest to state and prove is that of perturbations of discretized geodesic flows in negative curvature.  Let $X$ be a closed, negatively curved, locally symmetric   manifold, for example, a compact hyperbolic manifold.    Denote by  $T^1X$  the unit tangent bundle of $X$ and by $\psi_t$ the geodesic flow $\psi_t:T^1X\to T^1X$ over $X$. The flow $\psi_t$ preserves the canonical Liouville probability measure on $T^1X$, which we denote by $\mathrm{vol}=\mathrm{vol}_{T^1X}$.  Any element $\psi_t$ of this flow commutes with any other element, and thus
\[\Z_{\infty}(\psi_t) \supseteq \{\psi_s : s\in \RR\} \cong \RR.
\]
Our first result concerns volume-preserving perturbations of the {\em discretized flow:} the time-$t_0$ map $\psi_{t_0}$, for a fixed $t_0\neq 0$.  Such a perturbation $f\in \Diff^\infty_{\mathrm{vol}}(T^1X)$ will not necessarily embed in a flow: for example, any perturbation with a hyperbolic periodic point cannot embed in a flow, and such perturbations are plentiful.  The upshot of this result is that if such a perturbation does {\em not} embed in a flow, then it has virtually trivial centralizer.

\begin{maintheorem}\label{intro: cent dic geod fl} Let $X$ be a closed, negatively curved, locally symmetric manifold, and  let $\psi_t\colon T^1X\to T^1X$ be the associated geodesic flow. Fix $t_0\neq 0$, and suppose $f\in \Diff^\infty_{\mathrm{vol}}(T^1X)$ is a $C^{1}-$small perturbation of $\psi_{t_0}$. Then either $f$ has virtually trivial centralizer in ${\Diff^\infty(T^1X)}$ or $f$  embeds into a smooth, volume preserving flow (and thus $\Z_{\infty}(f) \supseteq \RR$).  Moreover, in the latter case, the centralizer $\Z_{\infty}(f)$ is virtually $\RR$.
 \end{maintheorem}
The conclusions of Theorem~\ref{intro: cent dic geod fl} hold in considerably greater generality; see Theorem \ref{main: cent dic geod fl} and Remark~\ref{rem: manifolds that work}.  In particular, $X$ can be any closed Riemannian manifold with pointwise  $1/4$-pinched negative curvature (such as a surface), or more generally any closed, negatively curved  manifold whose geodesic flow satisfies   either a $2$-bunched or narrow band spectrum condition.

Thus for perturbations of these flows, up to finite index subgroups, the centralizer is either $\ZZ$ or $\RR$.  We do not know whether the same result holds for perturbations of discretized Anosov flows in general.
\begin{quest} Do the same conclusions of Theorem~\ref{intro: cent dic geod fl} hold for the volume-preserving perturbations of the time-$t_0$ map of an arbitrary volume-preserving Anosov flow?
\end{quest}
 
A partial answer to this question has recently been found in dimension 3 by Barthelm\'e and Gogolev \cite{BarGo}.

We remark that {\em virtually trivial} cannot be replaced by {\em trivial} in the conclusion of Theorem~\ref{intro: cent dic geod fl}.  Indeed for any $t_0\in\RR$, Burslem shows in \cite[Theorem 1.3]{B} that the  time-$t_0/2$ map $\psi_{t_0/2}$ can be $C^\infty$ approximated by
$f\in \Diff^\infty_{\mathrm{vol}}(T^1X)$ with trivial centralizer.  Then map $f^2$ has virtually trivial, but not trivial, centralizer and $C^\infty$-approximates $\psi_{t_0}$.

\subsection*{Toral automorphisms}

Linear automorphisms of tori present a rich family of algebraic systems with notable rigidity properties.  Any orientation-preserving automorphism of the torus $\TT^d = \RR^d/\ZZ^d$ lifts to a linear automorphism of $\RR^d$ preserving $\ZZ^d$, which can be represented by a matrix $C\in \SL(d,\ZZ)$.  For such a matrix $C$ we write $T_C\colon \TT^d\to \TT^d$ to denote the associated toral automorphism.  Since $C$ has determinant $1$, the map $T_C$  preserves the Lebesgue-Haar measure on $\TT^d$, which we again denote by $\mathrm{vol}(=\mathrm{vol}_{\TT^d})$.

In the {\em hyperbolic} case where $C$  has no eigenvalues on the unit circle, the automorphism $T_C$ has a strong topological rigidity property known as structural stability: any perturbation of $T_C$ in $\Diff^1(\TT^d)$ is topologically conjugate to $T_C$.  The centralizer of a perturbation $f\in \Diff^1(\TT^d)$ {\em within $\Ho^+(\TT^d)$} is thus isomorphic to the centralizer of $T_C$ in $\Ho^+(\TT^d)$.  It is well-known (see  Lemma \ref{lemma: rank cent}) that when $C$ is irreducible ---  meaning that its characteristic polynomial is irreducible over $\ZZ$
 -- both $\Z_{\Ho^+(\TT^d)}(T_C)$ and $\Z_{\SL(d,\ZZ)}(C)$ are virtually finitely generated free abelian groups whose rank is determined by the number of distinct eigenvalues of $C$.   Of course for a perturbation $f\in \Diff^r(\TT^d)$ of $T_C$, the centralizer $\Z_{r}(f)$ can be considerably smaller than $\Z_{\Ho^+(\TT^d)}(f)$: in fact, Palis and Yoccoz showed that,   among the smooth Anosov diffeomorphisms, there exists an open and dense subset of $f\in \Diff^\infty(\TT^d)$ such that    the centralizer $\Z_{\infty}(f)$ is trivial \cite{PY1, PY2}.

From a dynamical point of view, perturbations of the non-hyperbolic automorphisms are considerably  more interesting.  When $C$ has no eigenvalues that are roots of unity, then $T_C$ is mixing with respect to $\mathrm{vol}$, and in several cases of interest, {\em stably mixing:} any sufficiently smooth, volume-preserving perturbation of $T_C$ is mixing if $d\leq 5$ \cite{RH05}.

We consider a case in which both structural stability and ergodicity are violated in a fairly dramatic fashion,  where the generating matrix $C\in \SL(d,\ZZ)$ has $1$ as an eigenvalue, with multiplicity $1$.\footnote{The case where $C\in \SL(d,\ZZ)$ has exactly one eigenvalue of modulus $1$ can be treated by similar methods.}  By conjugating by a toral automorphism, we may assume without loss of generality that  $C=\begin{pmatrix}
A&\\
& 1
\end{pmatrix}. 
$  For such $A$, the map $T_C= T_A\times \id_{\TT}$ admits non-conjugate {\em affine} perturbations of the form
$f=T_A\times R_\theta$,  where $R_\theta(z) = z+\theta$  is a rotation by $\theta\in\TT$ in last factor in $\TT^d = \TT^{d-1}\times \TT$, and so $T_C$ is not structurally stable, even within the restricted class of affine transformations.  By the same token, these affine perturbations also have large centralizer, commuting with any affine map of the form
$T_B\times R_\theta$, with $B\in \Z_{\SL(d-1,\ZZ)}(A)$, and $\theta\in \TT$.

In the case that $A$ is irreducible, the Dirichlet unit theorem gives that the group $\Z_{\SL(d-1,\ZZ)}(A)$ is virtually $\ZZ^{\ell_0}$, where $\ell_0 = \ell_0(A) :=r+c-1$,  $r$ is the number of real eigenvalues
of $A$ and $c$ is the number of pairs of complex eigenvalues of $A$  (see Lemma~\ref{lemma: rank cent}).   We obtain the following classification result for the centralizer of perturbations of $T_C$.

\begin{maintheorem}\label{intro: dich} Let $f\in \Diff^\infty_{\mathrm{vol}}(\TT^d)$ be a $C^1-$small, ergodic perturbation of $T_A\times \id_{\TT}$, where $A\in \SL(d-1,\ZZ)$ is hyperbolic and irreducible. 
Let $\ell_0 =\ell_0(A)$.  Then one of the following holds: 
\begin{enumerate}
\item $\Z_{\infty}(f)$ is virtually $\ZZ^\ell$ for some $\ell\in [1, \ell_0]$.   Furthermore, $\ell<\ell_0$ if $\ell_0>1$.
\item $\Z_{\infty}(f)$ is virtually $\ZZ \times \TT$. 
\item  $\Z_{\infty}(f)$ is virtually $\ZZ^{\ell_0}\times \TT$,  $\ell_0>1$ and $f$ is $C^\infty$ conjugate to $T_A\times R_\theta$,  $\theta\notin \QQ/\ZZ$.
\end{enumerate}
\end{maintheorem}

\begin{remark} Theorem~\ref{intro: dich} has a stronger formulation for perturbations of   isometric extensions of an irreducible toral automorphism, stated in  Theorem~\ref{main: dich} in the next section.
For similar problems on nilmanifolds, cf. our upcoming paper \cite{DWX}. 
\end{remark}

\begin{remark}Consider the simplest non-ergodic example  of $f=T_A\times\id$ itself, for which  $\Z_{\infty}(f)$ is virtually $\ZZ^{\ell_0}\times \Diff^\infty(\TT)$. This example illustrates the \emph{a priori} possibility that the centralizer might not be virtually abelian,  and thus part of  the work in  Theorem~\ref{intro: dich}  is to establish that for an {\em ergodic} perturbation, the centralizer is virtually abelian.  In particular, this shows that   the ergodicity assumption in  Theorem \ref{intro: dich} is necessary.  On the other hand, the ergodicity assumption is satisfied generically:   it is proved by Burns and Wilkinson in \cite{BW} and  F. Rodriguez Hertz, M. A. Rodriguez Hertz and Ures in \cite{RR} that ergodicity (indeed, mixing) holds open and densely among  the  \emph{partially hyperbolic diffeomorphisms with $1$-dimensional center} in $\Diff^\infty_{\mathrm{vol}}(\TT^d)$ (for precise definitions and more details, see Section \ref{sec: prl PH}). In particular, for any neighborhood $U$ of $T_C$, there is a $C^1-$open set  $U_0\subset U$ such that every $f\in U_0$ is ergodic.

  We conjecture that for any   volume preserving (possibly non-ergodic) $C^1$-small perturbation $f$ of $T_A\times \id$, the group $\Z_{\infty}(f)$  is either virtually trivial or contains a nontrivial Lie group.
\end{remark}

\begin{remark}
 We expect that the conclusions in Theorem \ref{intro: dich} extend to the case when $T_A$ is  \emph{reducible} hyperbolic toral automorphism as well. Moreover, we conjecture that for general hyperbolic $T_A$, the conclusion (1) should read: for every  $g\in \Z_{\infty}(f)$, and on any $<f,g>$-invariant subtorus of $\TT^d$, the action of $<f, g>$ is virtually a $\ZZ$-action.
 \end{remark}

Both  Theorems~\ref{intro: cent dic geod fl} and \ref{intro: dich} are consequences of more general results, which we state in Section~\ref{s:statements}.

\subsection*{The secret sauce}

While it does not appear in the statements, there is a hidden concept behind Theorems~\ref{intro: cent dic geod fl} and \ref{intro: dich}: pathological foliations.  Both the discretized geodesic flows and the toral automorphisms we discuss above preserve smooth, $1$-dimensional foliations, in the first case, the foliation by orbits of the flow, and in the second, the foliation by circles tangent to the last factor in $\TT^{d-1}\times \TT$.  

Transverse to the leaves of these foliations, the dynamics is hyperbolic, and so the theory of normally hyperbolic foliations developed in \cite{HPS} applies.  In particular, the perturbations of these examples considered in Theorems~\ref{intro: cent dic geod fl} and \ref{intro: dich} also preserve $1$-dimensional foliations with smooth leaves, homeomorphic as foliations to the unperturbed smooth foliations  (see Section \ref{sec: prl PH} for a detailed discussion).  The measure-theoretic properties of these {\em center} foliations are well-studied and play a key role in our proofs.

By a standard procedure, the volume $\mathrm{vol}$ can be locally disintegrated along the leaves of a foliation $\F$ to obtain  in each foliation chart a measurable family of measures, supported on the local leaves (or {\em plaques}) $\F_{loc}$ of the foliation.  Each  plaque $\F_{loc}(x)$ of a foliation, being a  $C^1$ embedded disk, also carries a natural measure class $\mathrm{vol}_{\F_{loc}(x)}$ associated to  leafwise volume, or length in the case of $1$-dimensional leaves. If the foliation is $C^1$ (i.e. has $C^1$ foliation charts), then the disintegration of $\mathrm{vol}$ and leafwise volume are equivalent, meaning they have the same sets of measure zero.  

When, as is typically the case in our perturbed examples, the foliation is not $C^1$, anything goes, at least {\em a priori}. The two extremal cases are:
\begin{itemize}
\item {\em Lebesgue disintegration,} where the disintegrated and leafwise volume are equivalent.  A foliation $\F$ of $M$ has Lebesgue disintegration if for every set $Z\subset M$:
\[\mathrm{vol}(Z) = 0 \iff \mathrm{vol}_{\F_{loc}(x)}(Z)=0,\;\hbox{ for } \mathrm{vol}\hbox{-a.e. }x\in M.  \]
\item {\em atomic disintegration,} where the disintegrated volume is atomic.    A foliation $\F$ of $M$ has atomic disintegration if there exists a set $Y\subset M$ and $k\geq 1$ such that
\[\mathrm{vol}(M\setminus Y) = 0\; \hbox{ and } \#\{Y\cap \F_{loc}(x)\} \leq k,\;\hbox{ for } \mathrm{vol}\hbox{-a.e. }x\in M.  \]
\end{itemize}
If a foliation fails to have Lebesgue disintegration with respect to volume, we call it {\em pathological}, a concept first considered by Shub and Wilkinson in \cite{SW}. This concept plays an important role in our paper.  In brief, pathological disintegration is associated with small centralizer and Lebesgue disintegration with large centralizer (at least in the group of homeomorphisms).

\subsection*{Higher rank abelian actions}  
 Another key role in our proofs is played by {\it higher rank abelian} group actions with some hyperbolicity.   A smooth {\em Anosov action} is a
homomorphism $\alpha\colon G\to  \Diff^\infty(M)$, where $G$ is a finitely generated abelian group, and $\alpha(a)$ is an Anosov diffeomorphism,  for some $a\in G$ (see Section ~\ref{sec: PH} for definitions of Anosov and partially hyperbolic diffeomorphisms). For example, if $f$ is Anosov, and the smooth centralizer $\Z_{\infty}(f)$
is finitely generated and abelian, then the action of $\Z_{\infty}(f)$ on $M$ is Anosov.  This is the case, for example, when $M$ is the  torus, and $f$ is an irreducible hyperbolic automorphism.

An Anosov action has \emph{higher rank}  if it contains an Anosov $\mathbb Z^2$ subaction that does not have a topological factor (possibly on a different manifold) which is virtually a $\mathbb Z$-action.   Anosov higher-rank actions often display a range of rigidity properties (cf. \cite{KNT}, \cite{KS1}), most strikingly \emph{global rigidity}. Katok  and Spatzier conjectured that any higher rank Anosov $\mathbb Z^k$ action on a compact manifold is \emph{essentially algebraic}, i.e. smoothly conjugate to an affine action on a nilmanifold, up to a finite cover of $M$ and up to a finite index subgroup in $\mathbb Z^k$. (For more on the Katok-Spatzier conjecture, see for example  \cite{DX1} and references therein). The conjecture was proved for Anosov actions {\it on nilmanifolds} by F. Rodriguez Hertz and Wang \cite{HW} (for the statement on $\TT^d$, see Theorem \ref{main: gl rig anosov} in Section \ref{sec: HR act}).

In particular,  if  $T_A$ is an {irreducible, hyperbolic} automorphism of the torus $\mathbb T^d$, where the centralizer  of $T_A$ is  virtually $\ZZ^\ell$, for some $\ell> 1$, then the result of Rodriguez Hertz and Wang  implies the following dichotomy for the centralizer of \emph{every} sufficiently small perturbation $f$ of $T_A$, when $\ell\geq 2$ 
(Corollary \ref{coro: gl rig toral} in Section \ref{sec: HR act}):
{\em either  $\Z_{\infty}(f)$ is virtually trivial, or
$\Z_{\infty}(f)$ is essentially algebraic, and its rank is the same as that of $\Z_{\infty}(T_A)$.}
This has been the only existing situation where the smooth centralizer is completely locally classified. 
Our results in Theorems \ref{intro: cent dic geod fl}  and  \ref{intro: dich} give classification of the centralizer for \emph{ergodic, conservative} perturbations of certain \emph{partially hyperbolic} systems.


One of the main achievements in \cite{HW} is showing existence of  {\em many} independent hyperbolic elements in an action given a {\em single} hyperbolic element. This is also one of the main obstacles to proving the Katok-Spatzier conjecture in full generality. 
In \cite{HW} it is shown that any higher-rank  Anosov action on a nilmanifold has  Anosov elements in {\em every} Weyl chamber; together with \cite{FKS}, this proves the Katok--Spatzier conjecture on nilmanifolds, and gives the dichotomy of the centralizer as mentioned above.   The proof in \cite{HW} makes use of the Franks--Manning conjugacy on nilmanifolds and fine analytic properties of the dynamics of Anosov diffeomorphisms, in particular exponential rates of mixing.


The actions considered here (as in the setting of Theorem \ref{intro: dich}), have a hyperbolic part and a $1$-dimensional nonhyperbolic, central part.  The hyperbolic part is, on a topological level, a maximal Anosov action -- considerably simpler than the actions considered in \cite{FKS, HW}.  On the other hand, the methods in these works are not available to us:  the central part of our actions obstruct conjugacy to a linear system, and the dynamics of the systems are potentially not even mixing.  What is available instead is a {\em leaf conjugacy} to a linear system, that is, a topological conjugacy modulo the center dynamics.  Starting from the leaf conjugacy,
 and using maximality of the action, we build up the partial hyperbolicity of other elements in the action.  
 Existence of many partially hyperbolic elements in the large rank centralizer   in the conservative setting forces Lebesgue disintegration of the volume in the center direction.

Our arguments are geometric rather than analytic in nature and employ a range of techniques, including the theory of normally hyperbolic foliations, rigidity of $1$-dimensional solvable group actions, Weyl chamber analysis, Pesin theory, normal forms, and Liv\v sic theory.   One important idea, also employed in \cite{BMW}, is to use Pesin theory
and uniform estimates to upgrade a uniformly expanded topological foliation $\W^\#$ to a foliation with smooth leaves.  To carry out such an argument requires precise control over the H\"older exponent of leaf conjugacies, something  established relatively recently  in \cite{PSW}. 



\subsection{Acknowledgements}   We thank Sylvain Crovisier, Benson Farb, Federico Rodriguez Hertz,  Curtis McMullen, Yakov Pesin, Rafael Potrie and Zhiren Wang for useful discussions and Andy Hammerlindl and Dennis Sullivan for corrections to an earlier manuscript.   We are grateful to Boris Kalinin for explaining to us the details of his recent results in normal form theory, which are used in this paper. Damjanovi\'c was supported by Swedish Research   Council   grant VR2015-04644.  Wilkinson was supported by NSF Grant DMS$-1402852$.   This research was partially conducted during the period Xu served as a Marie-Curie research fellow in Imperial College London.

\subsection{Structure of this paper}

In Section~\ref{s:statements} we state our main results in the more general context of partially hyperbolic diffeomorphisms with $1$ dimensional center foliations and discuss prior results.  Section~\ref{section: preliminaries} contains background information and some new techniques used in the proofs of our main results.  In Section~\ref{section: geodesic proofs}, we prove the main results about discretized geodesic flows (Theorems \ref{main: cent dic geod fl} and \ref{main: geod fl}). 
Theorem~\ref{main: thm pr}  provides the disintegration dichotomy which is the driver behind one of our main results, Theorem~\ref{main: dich}.  The proof of  Theorem~\ref{main: thm pr} occupies Section~\ref{sec: proof Cartan}.  Finally, in Section~\ref{sec: pf Thm dich}, we prove Theorem~\ref{main: dich}.  The Appendix contains the statement of a result from another work that we use in this paper.

\section{Statements of the main results and discussion}\label{s:statements}
\subsection{The general formulations}In this section we state the following more general versions of the rigidity results for centralizers, which immediately imply Theorems \ref{intro: cent dic geod fl} and  \ref{intro: dich}.
 
\begin{maintheorem}\label{main: cent dic geod fl} Let $X$ be a closed, negatively curved manifold, and let $\psi_t\colon T^1X\to T^1X$ be the geodesic flow.  Suppose that 
 $\psi_1$ satisfies either the $2$-bunched or narrow band spectrum condition given in Section~\ref{sec: normal form}, Definition~\ref{def:narrow band and bunched partially hyperbolic}.

 Then there exists $r_0 = r_0(X)\geq 1$ such that for all $r>r_0$, and any $t_0 \neq 0$, if  $f\in \Diff^\infty_{\mathrm{vol}}(T^1X)$ is sufficiently $C^{1}$ close to $\psi_{t_0}$, then   either $\Z_{r}(f)=\Z_{r,{\mathrm{vol}}}(f)$ is virtually trivial, or $\Z_{s}(f)=\Z_{s,{\mathrm{vol}}}(f)$ is virtually $\RR$ for any  $s\geq 1$.   In the latter case $f$ embeds into a $C^\infty$, volume preserving flow. 
 \end{maintheorem}

\begin{remark}\label{rem: manifolds that work} The hypotheses of Theorem~\ref{main: cent dic geod fl} are satisfied by a large class of negatively curved manifolds $X$.  In particular:
\begin{enumerate}
\item The $2$-bunched  condition is satisfied if $X$ has {\em pointwise (strictly) $1/4$-pinched curvature:}  the minimum and maximum sectional curvatures $K_{\hbox{\tiny{$\mathrm{min}$}}}(x) \leq K_{\hbox{\tiny{$\mathrm{max}$}}}(x) < 0$ at $x\in X$ satisfy 
\begin{equation}\label{eqn: quarter pinched} \zeta(X):=\inf_{x\in X} \frac{K_{\hbox{\tiny{$\mathrm{max}$}}}(x)}{K_{\hbox{\tiny{$\mathrm{min}$}}}(x)}  > 1/4.\end{equation} (See \cite[Theorem 3.2.17]{Kling} and the discussion following Definition~\ref{def:narrow band and bunched partially hyperbolic}).  This holds for example, if $X$ is a surface.
In this case $r_0(X) = \sqrt{\zeta(X)^{-1}} \in [1,2)$.
\item The  narrow band spectrum condition is satisfied by all locally symmetric $X$. If $X$ is a real hyperbolic manifold, then $r_0(X)=1$, and  if $X$ is locally symmetric but not real hyperbolic, then $r_0(X)=2$ (Lemma~\ref{lem: Loc Symm Narrow Band}).
\end{enumerate}
\end{remark}

Let $g\colon \TT^{d-1}\to \TT^{d-1}$ be a diffeomorphism.  An {\em isometric (circle) extension} of $g$ is a map $f = g_\rho\colon \TT^{d-1}\times \TT \to\TT^{d-1}\times \TT $ of the form
\[g_\rho(x,y)=(g(x),  y + \rho(x)),\]
where $\rho\colon \TT^{d-1}\to \TT$ is a  continuous map   taking values in the circle $\TT=\TT^1$.   If $\rho$ is homotopic   to a constant then it can be lifted to (and hence viewed as) a map taking values in $\RR$.   The map $g_\rho$ is a $C^r$ diffeomorphism if and only if $g$ and $\rho$ are $C^r$ and preserves volume if and only if $g$ does. 

The simplest examples of isometric extensions are products $g\times R_\theta$, where $R_\theta(y) := y+\theta$ is a rotation.  In this case $\rho\equiv\theta$ is a constant function.   It is easy to check that there exists $\beta\colon \TT^{d-1}\to \TT$ such that $\id_\beta\circ g_\rho =  \left(g\times R_\theta\right)\circ \id_\beta$ if and only if $\rho$ satisfies the cohomological equation
\[\rho = -\beta\circ g + \beta +\theta.
\]
In this case we say that $\rho$ is cohomologous to a constant $\theta$.   If $g\in \Diff^2_{\mathrm{vol}}(\TT^{d-1})$ is Anosov, then  $g_\rho$ is ergodic if and only if  $\rho$ is not cohomologous to a rational constant,  and  $g_\rho$ is stably ergodic if and only if $\rho$ is not cohomologous to a constant  \cite{BW99}.
 
If $T_{A}$ is an irreducible hyperbolic automorphism and $(T_{A})_\rho$ is ergodic, then for all $s\geq 1$, the centralizer of $(T_{A})_\rho$ in $\Diff^s(\TT^d)$ contains    $\ZZ\times \TT$. In addition, it contains $\ZZ^{\ell_0(A)}\times \TT$ if $\rho$ is $C^\infty$ cohomologous to a constant, where $\ell_0(A)>0$ is defined in the introduction. 
 \footnote{  In fact for $s$ large enough, the centralizer is either   virtually $\ZZ\times \TT$ or virtually $\ZZ^{\ell_0(A)}\times \TT$, in the latter case $\rho$ is $C^\infty$ cohomologous to a constant.  This follows from Theorem~\ref{main: dich}  but also has a more elementary proof using cocycle rigidity of the centralizer of $T_{A}$.} 
  Our first result addresses perturbations of these maps.

\begin{maintheorem}\label{main: dich} Suppose $A\in \SL(d-1,\ZZ)$ is an irreducible hyperbolic matrix.  Let $\ell_0:= \ell_0(A)$.  Then there exists $r_0\geq 1$ such that for any $r>r_0$ and any $C^\infty$ function $\rho_0\colon\TT^{d-1} \to \RR$,   if $f\in \Diff^\infty_{\mathrm{vol}}(\TT^d)$ is a $C^1-$small, ergodic perturbation of the isometric extension  $f_0:= \left(T_{A}\right)_{\rho_0}$, then one of the following holds for the centralizer  $\Z_{s}(f)=\Z_{s,\mathrm{vol}}(f)$:
\begin{enumerate}
\item (Small centralizer) $\Z_{s}(f)$ is virtually $\ZZ^\ell$ for some $\ell\in [1, \ell_0]$ and any $s\geq r$.   Furthermore, $\ell<\ell_0$ if $\ell_0>1$.
\item (Isometric extension) $\Z_{s}(f)$ is virtually $\ZZ\times \TT$ for all $s\geq r$. In this case $f$ is smoothly conjugate to a smooth isometric extension  $g_\rho$ of an Anosov diffeomorphism $g\in \Diff^\infty_{\mathrm{vol}}(\TT^{d-1})$. Moreover, either $g$ is not $C^\infty$ conjugate to $A$, or $\rho_0$ is not $C^\infty$ cohomologous to a constant.
\item (Rigidity) $\Z_{s}(f)$ is virtually $\ZZ^{\ell_0}\times \TT$ for all $s\geq 1$. In this case,  $f$ is $C^\infty$ conjugate to the product $T_{A}\times R_\theta$ with $\theta\notin \QQ/\ZZ$.  

\end{enumerate}
\end{maintheorem}

\begin{remark}   The value $r_0$ in Theorem~\ref{main: dich} is explicit: $r_0=\max(\frac{\lambda^s}{\mu^s}, \frac{\lambda^u}{\mu^u})$,   where $\lambda^u,\mu^u$ (resp. $\lambda^s, \mu^s$) are the top and bottom unstable (resp. stable) Lyapunov exponents of $A$.
\end{remark}
\begin{remark} In the interests of space, Theorem \ref{main: dich} treats only isometric extensions homotopic to $T_{A}\times \id_{\TT}$.  For the general case where $\rho_0\colon\TT^{d-1}\to\TT$ is not null-homotopic, similar results hold, up to finite factors.  In particular,  for an ergodic perturbation $f$ of an arbitrary isometric extension  $\left(T_{A}\right)_{\rho_0}$, conclusions (1) and (2) are the same, and in conclusion (3), $f$ is smoothly conjugate to an ergodic affine map isotopic to  $\left(T_{A}\right)_{\rho_0}$.
\end{remark}

\begin{remark}
In the case $\ell_0>1$ the conclusion (1) gives that  the rank of the centralizer of a perturbation is \emph{strictly} less than $\ell_0$. We conjecture that conclusion (1) should be much stronger: the centralizer should be virtually trivial. The main obstacle in obtaining virtually trivial centralizer in this case is that several techniques we use apply currently only to  the \emph{maximal} actions \footnote{  or more generally, for totally non-symplectic (TNS) action (cf. Proposition \ref{prop: cyc rig WH}).} defined in Section~\ref{sec: HR act}, as opposed to general higher rank actions.  

In sufficiently low dimension, Theorems~\ref{intro: dich} and \ref{main: dich} confirm this conjecture and 
give a dichotomy between virtually trivial centralizer and large centralizer.
In particular, if we assume in addition to the hypotheses of Theorem~\ref{main: dich} that
$A\in \SL(d-1,\ZZ)$ satisfies one of the following conditions:
\begin{itemize}
\item $d=3$ or $4$;
\item $d=5$ and $A$ has at least one pair of complex roots;
\item $d=6$ and $A$ has two pairs of complex roots; or
\item $d=7$ and $A$ has three pairs of complex roots;
\end{itemize} 
then $\ell_0(A) = 1$ or $2$,  and the dichotomy in   Theorem~\ref{main: dich}   reduces to the following:
if $f\in \Diff^\infty_{\mathrm{vol}}(\TT^{d})$ is a $C^1-$small, ergodic perturbation of $f_0$, then
  $\Z_{s}(f)$ is   either  virtually trivial,  virtually $\ZZ\times \TT$ for all $s\geq r$, or virtually $\ZZ^2\times \TT$  for all $s\geq 1$.
\end{remark}

Before stating the rest of the main results in this paper, we define partial hyperbolicity and some related concepts.

\subsection{Partially hyperbolic diffeomorphisms and center foliations}\label{sec: PH}

Let $M$ be a complete Riemannian manifold, and let $h\in \Diff(M)$.  A {\em dominated splitting} for $h$ is a
direct sum decomposition of the tangent bundle
\[TM=E^1\oplus E^2\oplus\cdots \oplus E^k\]
such that 
\begin{itemize}
\item the bundles $E^i$ are {\em $Dh$-invariant}: for every $i\in \{1,\ldots, k\}$ and $x\in M$, we have $D_xh(E^i(x)) = E^i(h(x))$; and
\item  $Dh\vert_{E^i}$ {\em  dominates} $Dh\vert_{E^{i+1}}$: there exists $N\geq 1$ such that for any $x\in M$ and any unit vectors $u\in E^{i+1}$ , and $v\in  E^{i}$:
\[ \|D_xh^N (u)\| \leq  \frac{1}{2} \,\|D_x h^N(v)\|.\]

\end{itemize}
The property of a splitting being dominated is independent of choice of equivalent metric (and independent of choice of metric in the case where $M$ is compact).  A dominated splitting is always continuous.  If $M$ is compact and $h'$ is $C^1$ close to $h$ with a dominated splitting, then $h'$ also has a dominated splitting, which varies continuously with $h'$ in the $C^1$ topology. 

A $C^{1}$ diffeomorphism $f: M \rightarrow M$ of a complete Riemannian manifold $M$ is \emph{partially hyperbolic} if there is a dominated splitting $TM = E^{u} \oplus E^{c} \oplus E^{s}$ and $N \geq 1$ such that for any $x \in M$, and any choice of unit vectors $v^{s} \in E^{s}(x)$ and $v^{u} \in E^{u}(x)$, we have
\[
\max\{\|D_xf^{N}(v^{s})\|,  \|D_xf^{-N}(v^{u})\|\} < 1/2.
\]
We always assume the bundles $E^{s}$ and $E^{u}$ are nontrivial. If $E^c$ is trivial then $f$ is  \emph{Anosov}.  

A flow $\varphi\colon M\times \RR\to M$ is {\em Anosov} if for some $t_0\neq 0$, the time-$t_0$ map $\varphi_{t_0}$ is partially hyperbolic, with the center bundle $E^c = \RR\dot\varphi$ tangent to the orbits of the flow.  If $\varphi$ is Anosov, then  the time-$t$ map $\varphi_{t}$
is partially hyperbolic for every $t\neq 0$.  An  example  of an Anosov flow is the geodesic flow over a closed, negatively curved manifold, such as those considered in Theorem~\ref{main: cent dic geod fl}.   

Isometric circle extensions of Anosov diffeomorphisms, such as the diffeomorphisms considered in Theorem~\ref{main: dich},  are also partially hyperbolic, with $E^c$ tangent to the vertical foliation by circles $\{\{x\}\times \TT: x\in \TT^{d-1}\}$ (see, e.g.~\cite{BW99}).
 
 If $M$ is a closed manifold, then partial hyperbolicity is open property in the $C^1$ topology on $\Diff^1(M)$.   Thus the $C^1$-small perturbations considered in Theorems~\ref{main: cent dic geod fl} and  \ref{main: dich} are also partially hyperbolic.

If $f$ is partially hyperbolic and $C^r$, $1\leq r\leq \infty$, then the bundles $E^{s}$ and $E^{u}$ are tangent to foliations $\W^{s}$ and $\W^{u}$, known respectively as the stable and the unstable foliations of $f$. These foliations have $C^{r}$ leaves but are typically only H\"older continuous. For a more detailed discussion of  foliation regularity, see Section \ref{sec: reg fol}.

We say a $Df-$invariant distribution $E\subset TM$ is \emph{integrable}  if there exists an $f-$invariant foliation $\W = \{\W(x)\}_{x\in M}$ with $C^1$ leaves everywhere tangent to the bundle $E$, and {\em uniquely integrable} if every $C^1$ curve tangent to $E$ lies in a single leaf of $\W$.

The bundles $E^u$ and $E^s$ are thus integrable, and are in fact uniquely integrable.
The center bundle $E^c$ is not always
integrable (see \cite{RRU}), but in many examples of interest,  such as the time-one
map of an Anosov flow and its perturbations, or  perturbations of an isometric extension of Anosov map,  the theory of normally hyperbolic foliations developed in \cite{HPS} implies that $E^c$ is integrable, as are the bundles $E^{cs} = E^c\oplus E^s$ and $E^{cu} = E^c \oplus E^u$. In particular, for those $f$ considered in this paper, $E^c$ is integrable, and in fact uniquely integrable,  tangent to a {\em center foliation }$\W^c$,  (see Theorem~\ref{main: HPS}).  Our main results can be recast in terms of the measure theoretic properties of center foliations, as follows.

\subsection{Lebesgue disintegration and large centralizer}

As mentioned in the introduction, some of the key ingredients in proofs of Theorems \ref{main: cent dic geod fl} and \ref{main: dich} are the following {\it dichotomy results} which link the disintegration of volume along the center foliation with the structure of the centralizer. 

For volume-preserving perturbations of the discretized geodesic flow on any negatively curved manifold, we have

\begin{maintheorem}\label{main: geod fl} Let $X$ be a closed, negatively curved Riemannian manifold,  and let $\psi_t\colon T^1X\to T^1X$ be the geodesic flow.  Fix $t_0\neq 0$, and suppose $f\in \Diff^{2}_{\mathrm{vol}}(T^1X)$ is a $C^{1}-$small perturbation of $\psi_{t_0}$. Then either the volume $\mathrm{vol}$ has Lebesgue disintegration along $\W^c_f$,  or $f$ has virtually trivial centralizer in $\Diff(T^1X)$.
\end{maintheorem}

For perturbations of an isometric extension of a hyperbolic toral automorphism, we have 
 
\begin{maintheorem}\label{main: thm pr} Let $f_0:\TT^d\to \TT^d$ and  $\ell_0$ be as in Theorem \ref{main: dich}, and   
let $f\in \Diff^2_{\mathrm{vol}}(\TT^d)$ be a $C^1-$small, ergodic perturbation of $f_0$. Then either the volume has Lebesgue disintegration along $\W^c_f$, or $\Z_2(f)$ is virtually $\ZZ^\ell$ for some $\ell\leq \ell_0$.    Moreover   $\ell<\ell_0$ if $\ell_0>1$.\end{maintheorem}
 

\subsection{Prior results}  

As mentioned in the introduction,  it is expected that the typical diffeomorphism has small centralizer.  Indeed, Smale asked \cite{Smale1, Smale2} whether the set of $C^r$ diffeomorphisms with trivial centralizer is generic in $\Diff^r(M)$.  Several works have been devoted to this question in various contexts, going back to Kopell's solution \cite{Ko} to the question in the smooth case on the circle: those diffeomorphisms with  trivial centralizer contains a $C^\infty$ open and dense  in $\Diff^\infty(\TT)$.   The question has also been answered in full generality by Bonatti--Crovisier--Wilkinson in the $C^1$ topology: trivial centralizer is generic (but not open)  in $\Diff^1(M)$ and $\Diff^1_{\mathrm{vol}}(M)$, for any closed manifold $M$ \cite{BCW1, BCW2, BCGW}.  See \cite{BCW1} for a discussion of the history of this problem.

In the restricted context of partially hyperbolic systems, stronger results are known in the smooth category: Palis--Yoccoz showed that the set of $C^\infty$ diffeomorphisms with trivial centralizer  contains an open and dense subset of the set of Axiom A diffeomorphisms in $\Diff^\infty(M)$ possessing at least one periodic sink or source \cite{PY1, PY2}.  The conditions have subsequently been relaxed  \cite{F, RV}.  In another direction, Burslem showed that for a class of $C^\infty$ partially hyperbolic systems, (including non-volume-preserving perturbations of the systems considered in this paper), there is a residual subset whose centralizer is trivial.

When it comes to (partially) hyperbolic diffeomorphisms whose centralizers contain large rank abelian subgroups of (partially) hyperbolic diffeomorphisms, the general philosophy has been that a rich variety of (partially) hyperbolic dynamics in an abelian group action should be a rare occurrence. Classes of algebraic examples of such abelian actions have been listed in \cite{KS}  by Katok and Spatzier, who also proved in  \cite {KS1} that such \emph{Anosov}  abelian actions are \emph{locally rigid}:  small perturbations of such an action are all smoothly conjugate to unperturbed action.  Further  local rigidity results  for classes of partially hyperbolic abelian actions are found in \cite{DK3},  \cite{VW}.  Moreover, for Anosov diffeomorphisms,  if the centralizer contains a $\mathbb Z^2$ subgroup  that does not factor onto a virtually $\mathbb Z$-action,  Katok and Spatzier conjectured that $f$ is then smoothly conjugate to a hyperbolic (infra)nilmanifold automorphism, and in particular it has a full rank centralizer smoothly conjugate to a group of automorphisms. We refer to \cite{DX1},  \cite{SV} and references therein for the history and most recent results in the direction of this \emph{global rigidity} conjecture.

 In the case of volume preserving  partially hyperbolic diffeomorphisms with \emph{compact center foliation}, it is found in \cite{DX0} that a large rank centralizer with sufficiently many  partial hyperbolic elements also leads to global rigidity. In particular,  it was first discovered  in \cite{DX0} that the bad disintegration of volume along the leaves of the center foliation should be the main obstacle to rigidity for higher rank partially hyperbolic actions. The forthcoming paper \cite{DWX} exploits this further by obtaining in some cases stronger global rigidity results (see Appendix A).
For the case of commuting isometric extensions over hyperbolic toral automorphisms, local rigidity results have been obtained earlier under Diophantine conditions, in \cite{DF}.

Work of Avila, Viana and Wilkinson \cite{AVW, AVW2} establishes a dichotomy for a class of partially hyperbolic diffeomorphisms with $1$-dimensional center foliation: either the disintegration of Lebesgue is atomic on the center foliation or volume has   Lebesgue disintegration on  the center.   \footnote{Under an accessibility assumption. See Section~\ref{ss: accessibility}.}  Moreover, for these maps, if  volume has Lebesgue disintegration on the center, then there is a {\em continuous} volume-preserving flow commuting with the map.  These results apply directly to the systems considered here, and we take them as a starting point.  Otherwise, our methods are almost entirely disjoint from those in \cite{AVW, AVW2}.

\section{Preliminaries}\label{section: preliminaries}


\subsection{Regularity of maps and foliations}\label{sec: reg fol} 
For $r\in (0,1)$, we say that map between metric spaces 
is $C^r$ if it is  H\"older continuous of exponent  $r$.
For $r \geq 1 $ we say that a map between smooth manifolds is $C^r$ if it is $C^{[r]}$ and the $[r]$th-order derivatives are $C^{r-[r]}$. For $r\geq 0$, a map is $C^{r+}$ if it is $C^{r+\varepsilon}$ for some 
$\varepsilon> 0$.

Let $M$ be a manifold of dimension $d\geq 2$. A {\em $k-$dimensional topological foliation}
 $\F$ of $M$ is a decomposition of $M$ into path-connected subsets
\[M = \bigcup_{x\in M} \F(x)
\]
called {\em leaves}, where $x\in \F(x)$, and two  leaves $\F(x)$ and $\F(y)$ are either disjoint or equal,  and a covering of  $M$ by coordinate neighborhoods $\{U_\al\}$ with local coordinates $(x^1_\al,
\dots, x^d_\al)$ with the following property.  For $x\in U_\al$, denote by $\F_{U_\al}(x)$ the connected component of $\F(x)\cap U_\al$ containing $x$. Then in coordinates on $U_\al$ the local leaf $\F_{U_\al}(x)$ is  given by a set of equations of the form $x^{k+1}_\al=\dots=x^d_\al= cst$. If the local coordinates $(x_{\al}^1,\dots, x_\al^d)$ can be chosen uniformly $C^r$ along the local leaves (i.e., to have uniformly $C^r$ overlaps on the sets  $x^{k+1}_\al=\dots=x^d_\al= cst$) then we say that $\F$ {\em has $C^r-$leaves}. If the $(x_\al^1,\dots, x_\al^d)$ can be chosen $C^r$ on $U_\al$ then  $\F$ is called a {\em $C^r$ foliation.}

Note that the leaves of a foliation with $C^r$ leaves are $C^r$, injectively immersed submanifolds of $M$.


The next lemma follows from an application of $C^r-$section theorem in \cite{HPS}; for a precise proof cf. Corollary 5.6 in \cite{DX} or \cite{PSW}.
\begin{lemma}\label{lemma: impr cr}
Let $f$ be a $C^{r+1}$ diffeomorphism of a closed Riemannian manifold $M$. Let $\W$ be an $f-$invariant foliation with uniformly $C^{r}-$leaves. For $x\in M$, let $\al_x:=\|Df |^{-1}_{T\W(x)}\|$. Let $E^1$ and $E^2$ be  continuous, $f-$invariant distributions on $M$ such that the distribution $E = E^1\oplus E^2$ is uniformly $C^r$ along $\W$ leaves and $E^1\oplus E^2$ is a dominated splitting in the sense that for any $x \in M$,
$$k_x:=\frac{\max_{v\in {E}^2(x), \|v\|=1}\|Df(v)\|}{\min_{v\in {E}^1(x), \|v\|=1}\|Df(v)\|}<1.$$
If $\sup_{x\in M}k_x\al_x^r<1$, then $E^1$ is uniformly $C^r$ along the leaves of $\W$. In particular if $\al_x\leq 1$ for all $x\in M$ then $E^1$ is uniformly $C^r$ along the leaves of $\W$.
\end{lemma}

Suppose $\F$ is a foliation of a closed manifold $M$ with $C^1$ leaves, and $\mu$ is a Borel probability measure on $M$. Let $B$ be a foliation box, and let $\mu_B$ be normalized Lebesgue measure on $B$.  There is a unique family of conditional measures $\mu_x$ defined for $\mu_B-$almost every $x$ in $B$ with the following properties (see \cite{Ro}).  First, for almost every $x$, the measure $\mu_x$ is supported on  the plaque
$\F_B(x)$; second,  for every $\mu_B$-integrable function $\psi\colon B\to \RR$, we have
\[\int_B\psi(x)\,d\mu_{B}(x) = \int_B\int_{\F_{B}(x)} \psi(y) d\mu_x(y)\,d\mu_{B}(x).
\]

 We say $\mu$ has {\it Lebesgue disintegration} along $\F$ if for any foliation box $B$ and $\mu_B-$almost every $x$, the conditional measure of $\mu_B$ on $\F_B(x)$ is equivalent to the Riemannian measure on $\F_B(x)$. The measure $\mu$ has {\it atomic disintegration} (along $\F$) if there exists $k\geq 1$ such that for any foliation box $B$ the conditional of $\mu_B$ measure on $\F_B(x)$ is atomic,  with at most $k$ atoms, for $\mu_B-$almost every $x$.

\begin{lemma}\label{lemma:fix center} 
 Let  $\F$ be an orientable topological foliation of a closed manifold $M$ such that all leaves are circles. Suppose that there exists a full volume set  $S\subset M$ and $k\in \NN$ such that $S$ meets almost every leaf of $\F$ in exactly $k$ points. Let 
$\mathcal G_{\hbox{\tiny{$\mathrm{fix}$}}}(\F)$ be  the set of $g\in \Diff_{\mathrm{vol}}(M)$ such that  $g$ preserves orientation on $\F$, and $g(\F(x)) = \F(x)$, for all $x\in M$.  Then $\mathcal G_{\hbox{\tiny{$\mathrm{fix}$}}}(\F)$ is a finite cyclic group.
\end{lemma}
\begin{proof}Since the action of $\mathcal G_{\hbox{\tiny{$\mathrm{fix}$}}}(\F)$ fixes all the leaves of $\F$ and preserves the volume, on almost every leaf $\F(x)$, any element $g$ of $\mathcal G_{\hbox{\tiny{$\mathrm{fix}$}}}(\F)$ maps atoms to atoms, which means that $g$ induces a permutation on $S\cap \F(x)$. Moreover since $g$ preserves the orientation of each circle leaf of $\F$ it induces a cyclic permutation (with respect to the circle ordering) of the atoms on almost every leaf.  

Thus for every $x\in S$,  the restriction of $g\in\mathcal G_{\hbox{\tiny{$\mathrm{fix}$}}}(\F)$  to $\F(x)$ has rotation number ${k'(g,x)}/{k}~(\mathrm{mod }\,1)$, for some $k\in \ZZ^+$ and $k'=k'(g,x)\in \ZZ/k\ZZ$, where $k$ is the number of atoms. Since the rotation number is a continuous function on diffeomorphisms, and $S$ is dense, $k'(g,x)$ is independent of $x$. Therefore on \textbf{every} center leaf, $g$ has rotation number ${k'(g)}/{k}~~(\mathrm{mod }\,1)$. Moreover for any other $h\in \mathcal G_{\hbox{\tiny{$\mathrm{fix}$}}}(\F)$ such that $k'(g)=k'(h)$,  and every $x\in S$, $h$ induces the same permutation on $S\cap \F(x)$ as $g$, which implies that $g=h$, by the density of $S$. Therefore $k'$ induces an injective homomorphism from $\mathcal G_{\hbox{\tiny{$\mathrm{fix}$}}}(\F)$ to $\ZZ/k\ZZ$.  
\end{proof}

\subsection{Lyapunov exponents and the Oseledec splitting}\label{Oseledec}  Suppose $M$ is a smooth manifold  and $f\in \Diff^1(M)$ is a diffeomorphism preserving a probability measure $\mu$ (for instance, volume). 
In analogy with the Birkhoff ergodic theorem, one can inquire about
the asymptotic behavior of the composition of tangent maps of $f$
$$D_pf^n=D_{f^{n−1}(p)}f\circ\cdots \circ D_pf : T_pM\to T_{f^n(p)}M,$$
for $\mu$-a.e. $p\in M$.
An answer is given by the Oseledets Multiplicative Ergodic theorem, which we describe here in the setting of continuous cocycles.

Suppose $X$ is a compact metric space and $E\to X$ is a (continuous) vector bundle. Let $T: X\to X$ be homeomorphism.  A continuous linear {\em cocycle} over
$T$ is a bundle map $F\colon E\to E$ covering $T$.  
On the fibers, $F$ is given by  linear maps
$F_x\colon E_x \to E_{T x}$
that vary continuously with $x$.
For simplicity we assume that each $F_x$ is invertible, so that $F$ is a bundle isomorphism. 

Suppose that $T$ preserves an ergodic probability measure $\mu$ on $X$, and $E$ is equipped with a continuous Finsler structure $\{\|\cdot \|_x : x\in X\}$.
Then the Oseledec theorem gives real numbers $\lambda_1>\cdots> \lambda_k$   called {\em Lyapunov exponents} and a measurable, $F-$invariant splitting
$E=E^{\lambda_1}\oplus \cdots \oplus E^{\lambda_k}$,
such that for $v\in E_x\setminus\{0\}$, 
\[v\in E^{\lambda_i}_x\iff \lim_{n\to \pm \infty}\frac{1}{n}\log\|F^n(v)\|_{T^n(x)}=\lambda_i.\]
 The splitting  $E=\oplus E^{\lambda_i}$ is called
the  Oseledets splitting for the cocycle.

The following  well-known result allows one to deduce {\em uniform} growth of cocycles from knowledge about exponents for {\em every} invariant measure. The proof is a corollary of a classical result on subadditive sequences (cf. \cite{S98} or chapter 4 in \cite{K11}.)

\begin{lemma}\label{lemma: est cocyc}Let $f : X\to X$ be a continuous map of a compact metric
space, and let $F\colon E\to E$ be a continuous linear cocycle over $f$, where $p:E\to X$ is a continuous vector bundle over $X$. 
\begin{enumerate}
\item If for any $f-$invariant ergodic measure $\nu$, the top Lyapunov exponent $\lambda^{\hbox{\tiny{$\mathrm{max}$}}} (F,\nu)$ is  $\leq\lambda$, then for any $\epsilon>0$, there exists $n\in \ZZ^+$ such that $$\|F^n(x)\|\leq e^{n(\lambda+\epsilon)},~~ \forall x\in X.$$
\item If for any $f-$invariant ergodic measure $\nu$, the bottom Lyapunov exponent $\lambda^{\hbox{\tiny{$\mathrm{min}$}}}(F,\nu)$ is $\geq \lambda'$, then for any $\epsilon>0$, there exists $n\in \ZZ^+$ such that $$\|F^n(x)^{-1}\|^{-1}\geq e^{n(\lambda'-\epsilon)},~~ \forall x\in X.$$
\end{enumerate}
\end{lemma}

\subsection{Some useful properties of commuting maps}\label{commuting}

A basic principle in the study of abelian actions is the following: if $f$ and $g$ are commuting maps, and $\Upsilon$ is an $f$-invariant object, then $g_*(\Upsilon)$ is also $f$-invariant. For example, if $f(p) = p$, then $f(g(p)) = g(f(p)) = g(p)$. Thus $g(\mathrm{Fix}(f)) \subset \mathrm{Fix}(f)$; in other words, the set of $f$-periodic points of period $k$ is a $g$-invariant set.  Similar results hold for invariant sets of commuting homeomorphisms, such as the limit set and non-wandering set.

In the measurable context, if $\mu$ is an $f$-invariant measure, then $g_\ast\mu$ is also $f$-invariant, and so $g_\ast$ preserves the set of $f$-invariant measures.  When further assumptions are added, such as those in the present context, we get the following useful lemma.

\begin{lemma}\label{lemma: g pr vol} Let $M$ be a closed manifold, and suppose that $f,g\in \Diff(M)$ satisfy $fg=gf$.   If $f$ is  volume preserving and  topologically transitive (for example, if $f$ is ergodic with respect to volume), then $g$ is volume preserving as well.
\end{lemma}
\begin{proof}The commutativity implies that $\mathrm{vol}_M$ and  ${g}_\ast (\mathrm{vol}_M)$ are both $f-$invariant measures. Since $g$ is $C^1$, the induced Radon-Nikodym derivative $\frac{d ({g}_\ast \mathrm{vol}_M)}{d (\mathrm{vol}_M)}$ is an $f-$invariant continuous function. Transitivity of $f$ implies that this derivative is constant and equal to the degree of $g$, which is $1$.  Thus ${g}_\ast(\mathrm{vol}_M)=\mathrm{vol}_M$. 
\end{proof}

If $f$ and $g$ are commuting diffeomorphisms, then their derivatives commute as well.   It follows that if $f(p) = p$, then the derivative of $f$ at $p$ is conjugate to its derivative at $g(p)$,
and so $D_pf$ and $D_{g(p)}f$ have the same eigenvalues.  More generally:
\begin{lemma}\label{lemma: commLyap} Let $M$ be a closed manifold, and suppose that $f,g\in \Diff(M)$ satisfy $fg=gf$.  If $\mu$ is an ergodic invariant measure for $f$, then the Lyapunov exponents of $\mu$ are the same as the Lyapunov exponents of $g_\ast\mu$.
\end{lemma}

 Applying the same principle to the invariant subbundles in a dominated splitting, we obtain the following lemma, whose proof is straightforward.

\begin{lemma}\label{lemma: PHcent} Let $M$ be a closed manifold, and suppose that $f,g\in \Diff(M)$ satisfy $fg=gf$.   If $f$ preserves a dominated splitting
$TM = E^1\oplus\cdots\oplus E^\ell$,
then so does $g$.  Moreover if, for some $i$,  $E^i$ is uniquely integrable, with integral foliation $\W^i$, then $g(\W^i) = \W^i$.
\end{lemma}

Sufficiently high regularity of a map plus some hyperbolicity can force high regularity of its centralizer.  A basic motivating example is a linear map on $\RR$.  If $f(x) = 2x$ and $fg = gf$, then $g(0) = 0$, and the commutativity of $f$ and $g$ implies that for all $x\neq 0$ and $n$:
\[\frac{g(x)}{x} = \frac{f^n g f^{-n}(x)}{x} = \frac{2^n}{x} g\left(\frac{x}{2^n}\right) = \frac{g\left(\frac{x}{2^n}\right) - g(0))}{\frac{x}{2^n}}.
\]
If $g$ is differentiable at $0$, then the right hand side converges as $n\to \infty$ to $g'(0)$.  Thus $g(x) = g'(0) x$ is linear. 

 A  more sophisticated illustration of this principle  in the setting of linear Anosov diffeomorphisms is the following result, due to Adler and Palais:

\begin{lemma}\label{lemma: cent linear Ansv}[\cite{AP}] Suppose $A\in \SL(k,\ZZ) $ does not have a root of unity as an eigenvalue, and let $T_A$ be  the induced  automorphism of $\TT^k$. Suppose $h:\TT^k\to \TT^k$ is a homeomorphism such that $T_A h=hT_A$.  Then $h$ is affine and $h(0)\in \QQ^k/\ZZ^k$.
\end{lemma}

For such toral automorphisms $T_A$, we thus have
\[\Z_{\mathrm{Homeo}(\TT^k)}(T_A) \subseteq \{ x\mapsto T_L x + b  :  L\in \Z_{\GL(k,\ZZ)}(A),\;\, b\in\QQ^k/\ZZ^k \}.
\]
When $A$ is irreducible, the linear part of the right hand side can be computed using the following lemma, which a corollary of the Dirichlet unit theorem (cf. Proposition 3.7 in \cite{KKS}).

\begin{lemma}\label{lemma: rank cent}  Let $A\in \GL(k,\ZZ)$ be a matrix with characteristic polynomial irreducible over $\ZZ$. Denote by $\Z_{\GL(k,\ZZ)}(A)$ and $\Z_{\SL(k,\ZZ)}(A)$  the centralizer  of $A$ in $\GL(k,\ZZ)$ and $\SL(k,\ZZ)$, respectively.
Then $\Z_{\GL(k,\ZZ)}(A)$ and $\Z_{\SL(k,\ZZ)}(A)$ are  abelian,  and both are virtually $\ZZ
^{r+c−1}$ where $r$ is the number of real eigenvalues
and $c$ is the number of pairs of complex eigenvalues, $r + 2c = k$.
\end{lemma}

\subsection{More on partial hyperbolicity}\label{sec: prl PH} In this section we discuss fundamental concepts in the study of  partially hyperbolic diffeomorphisms: {\em normal hyperbolicity,  leaf conjugacy, center bunching, and accessibility}. We also discuss some  results of Avila--Viana--Wilkinson \cite{AVW, AVW2}    that we use in this paper. 

\subsubsection{Normal hyperbolicity}
Suppose $M$ is closed manifold, and let $f_1, f_2\in \Diff(M)$. Assume that $\F_1,\F_2$ are foliations of $M$ with $C^1$ leaves and that $f_1$ and $f_2$ respectively preserve $\F_1$ and $\F_2$.
\begin{definition} A {\em leaf conjugacy} from  $(f_1, \F_1)$ to $(f_2, \F_2)$ is a homeomorphism $h: M\to M$ sending $\F_1$ leaves diffeomorphically onto $\F_2$ leaves, equivariantly in the sense that \[h(f_1(\F_1(p))) = f_2(\F_2(h(p))), ~~\forall p\in M.\]\end{definition}

\begin{definition}Suppose $f \in \Diff(M)$ and $\F$ is an $f-$invariant foliation of $M$ with $C^1$ leaves. $\F$ is {\em normally hyperbolic} if there exists a $Df-$invariant dominated splitting $TM=E^u\oplus E^c\oplus  E^s$, with at least two of the bundles nontrivial,   such that $Df$ uniformly expands $E^u$, uniformly contracts $E^s$,  and such that $T\F=E^c$.  
\end{definition}
Note that a diffeomorphism with a normally hyperbolic foliation is partially hyperbolic, with $E^c = T\F$, but, as remarked above, the converse does not hold in general: the center bundle of a partially hyperbolic diffeomorphism is not necessarily tangent to a foliation, let alone an invariant foliation.

\begin{definition}A partially hyperbolic diffeomorphism $f$ is {\em dynamically coherent} if there exist
$f-$invariant center stable and center unstable foliations $\W^{cs}$ and $\W^{cu}$, tangent to the bundles $E^{cs}$ and $E^{cu}$, respectively; intersecting their leaves gives an invariant center foliation $\W^c$.
\end{definition}

\subsubsection{Fibered partially hyperbolic systems}\label{sec: comp cent}

In many of the cases of interest here, the center foliation $\W^c$ of a partially hyperbolic diffeomorphism $f$ has compact leaves that form a fibration.  We distinguish between several cases of such fibered systems.

\begin{definition}Let $f$ be a partially hyperbolic diffeomorphism of a closed manifold $M$. Assume that there exists an $f-$invariant center foliation $\W^c_f$ with compact leaves.
\begin{itemize}
\item If $\W^c_f$ is a topological fibration of $M$, i.e. the quotient space $M/\W^c_f$ is a topological manifold\footnote{Or, equivalently, if $\W^c_f$ has trivial holonomy; see \cite{Bo}}, then $f$ is called a \emph{fibered partially hyperbolic system}, and the map $\bar{f}:M/\W^c\to M/\W^c$ canonically induced by $f$ is called the \emph{base map}.
\item A fibered partially hyperbolic system $f$ is  \emph{smoothly fibered} (or \emph{$C^{r}-$fibered}, for  $r\geq 1$)  if $\W^c_f$ is a $C^\infty$ (respectively $C^r$) foliation, and $f$ is $C^\infty$ (resp. $C^r$).
\item A fibered partially hyperbolic system $f$ is  \emph{isometrically fibered} if there is a continuous Riemannian metric on $E^c$ such that $Df|_{E^c_f}$ is an isometry.
\item An isometrically fibered partially hyperbolic system $f$ is an {\em  isometric extension} (or smoothly isometrically fibered) if $f$ is smoothly fibered.
\end{itemize}
\end{definition}

\subsubsection{Leafwise structural stability}
A central result in \cite{HPS} concerns perturbations of normally hyperbolic systems. It provides techniques to study integrability of the central distribution and robustness of the central foliation for partially hyperbolic systems.  

To study the precise smoothness of the leaves of a normally hyperbolic foliation, we refine the definition of normal hyperbolicity.  For $r\geq 1$ we say that $(f,\F)$ is {\em $r$-normally hyperbolic} if there exists $k\geq 1$ such that
\[\sup_p \|D_p f^k|_{E^s}\|\cdot \|(D_p f^k|_{T\F})^{-1}\|^r<1, \hbox{ and }\sup_p \|(D_p f^k|_{E^u})^{-1}\|\cdot \|D_p f^k|_{T\F}\|^r<1.\]
Note that $1$-normally hyperbolic $=$ normally hyperbolic, and $r$-normal hyperbolicity is a $C^1$-open condition.

The proof of the following theorem can be found in  \cite[Theorems 7.5 and 7.6]{HPS} (see also Remark 4 on p. 117), \cite[Theorem 1.26]{Bo}, and \cite{CPnotes},  \cite[Theorem 7.1]{HPS}, and \cite[Theorems A and B]{PSW}.
See the discussion in \cite[Section 3]{PSW}.

\begin{maintheorem}[Foliation Stability and H\"older continuity of the leaf conjugacy] \label{main: HPS} Let $M$ be a closed manifold, and let $(f,\F)$ be an $r$-normally hyperbolic foliation of $M$,  for some $r\geq 1$, with $Df$-invariant splitting $E^u\oplus \left( T\F = E^c\right)\oplus E^s$.    Then the leaves of $\F$ are uniformly $C^r$. The bundles $E^u$ and $E^s$ are uniquely integrable and the leaves of their integral foliations $\W^u$ and $\W^s$ are as smooth as $f$.

 Suppose in addition that one of the following holds:
\begin{enumerate}
\item[(a)] $f$ is a fibered system, with $1$-dimensional center fibration $\F$, or
\item[(b)] the restriction   $\left.Df\right|_{T\F}$ is an isometry.
\end{enumerate}

Then
\begin{enumerate}
\item $f$ is dynamically coherent, and the foliations  $\W^{cu}$, $\W^{cs}$ and $\F = \W^{cu}\cap \W^{cs}$ are $r$-normally hyperbolic and uniquely integrable.
\item Every diffeomorphism $g$ that $C^1$-approximates $f$ is dynamically coherent and the foliations 
 $ \W^{cs}_g$, $\W^{cu}_g$ and $\F_g = \W^{cu}_g\cap \W^{cs}_g$ are $r$-normally hyperbolic. near $\F$. Moreover, $(f, \F)$ is leaf conjugate to $(g, \F_g)$ by a  homeomorphism $h^c\colon M \to M$ close to  the identity.   
\item In case (a) the conjugacy in (2) is H\"older continuous.
\end{enumerate}

\end{maintheorem}

  By combining Theorem~\ref{main: HPS} with Lemma~\ref{lemma: PHcent} gives the next proposition as an immediate corollary.

\begin{prop}\label{prop: gwc=wc}    Let $f\colon M\to M$ satisfy one of the following conditions.

\begin{enumerate}
\item[(a)] $M=\TT^d$, and $f$ is a $C^1-$small perturbation of an isometric extension of an Anosov diffeomorphism of $\TT^{d-1}$;
\item[(b)] $M=T^1X$, where $X$ is a closed, negatively curved manifold, and 
$f$  is a $C^1-$small perturbation of the discretized geodesic flow $\psi_{t_0}$, for some $t_0\neq 0$ (or more generally any Anosov flow). 
 
\end{enumerate}
 Then $f$ is dynamically coherent, and for any   $g\in \Z_{1}(f)$   we have  $g\W^{\ast}_f=\W^{\ast}_f$, for $\ast\in \{c,s,u,cs,cu\}$.
\end{prop}
 
Finally, we have a lemma that we will use in Section~\ref{section: geodesic proofs}.
 \begin{lemma}\label{lem: g fixes periodic}  Let $\psi_t\colon M\to M$  be an Anosov flow with the property that the lift $\tilde\psi_t$ of $\psi_t$ to the universal cover $\widetilde M$ has no closed orbits, and let $f$ be a $C^1$-small perturbation of $\psi_{t_0}$, for some $t_0\neq 0$.  

Then for any  $g\in \Z_{1}(f)$ , and for any closed leaf $\W^c_f(x)$, there exists $k\geq 1$ such that 
$g^k(\W^c_f(x)) = \W^c_f(x)$.  
\end{lemma}
\begin{proof}  Consider the lifts $\tilde{\psi}_t,\tilde{f}$ of $\psi_t,f$ respectively to $\widetilde M$, where $\tilde{f}$ is uniformly $C^1-$close to $\tilde{\psi}_t$, and $\tilde f$ preserves the lift $\widetilde\W^c_f$ of $W^c_f$.  On each $\widetilde\W^c_f-$leaf, the action of $\tilde f$ is uniformly close to a translation by $t_0$ on $\RR$ and thus is topologically conjugate to a translation. Thus there exist $0<\tau_{min} \leq \tau_{max}$ such that for every $x\in \widetilde M$ and every $N\geq 1$, we have $d^c(x, \tilde f^N(x)) \in [N\tau_{min}, N\tau_{max}]$, where $d^c$ is the distance measured along $\widetilde\W^c_f$ leaves.

Let $g\in \Z_{\Diff(M)}(f)$, and fix an arbitrary lift $\tilde g\colon \widetilde M\to \widetilde M$.
Fix an arbitrary $\tilde x_0\in \widetilde M$, and let $\gamma\colon [0,T]\to \widetilde M$ be a unit-speed, $C^1$ path tangent to $\widetilde\W^c_f(\tilde x_0)$ with $\gamma(0) = \tilde x_0$ and  $\gamma(T) = \tilde f^N (\tilde x_0)$, for some $N\geq 1$.  Note that $T\in [N\tau_{min},N \tau_{max}]$.

 Proposition~\ref{prop: gwc=wc} implies that  $\tilde g$ preserves the foliation $\widetilde\W^c_f$,  and so for any $m\geq 0$,  $\tilde g^m(\gamma)$ is a path tangent to $\widetilde\W^c_f$ from $\tilde{g}^m(\tilde x_0)$ to $\tilde{f}^N(\tilde{g}^m(\tilde x_0))$.  It follows that the length of  $\tilde g^m(\gamma)$ also  lies in the interval $[N\tau_{min}, N\tau_{max}]$.  Now suppose that $\W^c_f(x)$ is a closed center leaf in $M$ of length $R\in [N\tau_{min}, N\tau_{max}]$. Then there are lifts  $z_1, z_2$  of $x$ to $\widetilde M$ connected by a  unit-speed path in $\widetilde \W^c_f$ of length $R$.  This path is contained in a  unit-speed path connecting $z_1$ to $\tilde f^{N+1}(z_1)$, whose length lies in $[(N+1)\tau_{min}, (N+1)\tau_{max}]$.  Thus, for all $m\geq 1$, the distance between $g^m(z_1)$ and $g^m(z_2)$ is bounded by  $(N+1)\tau_{max}$. 

Since $g$ is a diffeomorphism preserving $\W^c_f$, it permutes the closed leaves.  Thus $g^m(\W^c_f(x))$ is a closed leaf whose length is at most $CR$, where $C$ does not depend on $R$ or $m$.  Since $f$ is a perturbation of $\psi_{t_0}$, its periodic center leaves of bounded length are isolated, and there are only finitely many of length $\leq CR$.  If follows that every closed center leaf of $\W^c_f$ is $g$-periodic.
\end{proof}

 \subsubsection{Bunching conditions}\label{sss: centerbunched}
For $r \geq 1$, we say that a partially hyperbolic diffeomorphism $f$ of a Riemannian manifold $M$ is   {\em center $r-$bunched}   if there exists $k\geq  1$  such that:
\[\sup_p \left \{ \|D_p f^k|_{E^s}\|\cdot \|(D_p f^k|_{E^c})^{-1}\|^r, \|(D_p f^k|_{E^u})^{-1}\|\cdot \|D_p f^k|_{E^c}\|^r\right \}<1,\]
\begin{eqnarray*}
&&\sup_p \|D_p f^k|_{E^s}\|\cdot \|(D_p f^k|_{E^c})^{-1}\|\cdot \|D_p f^k|_{E^c}\|^r<1, \hbox{ and}\\
&&\sup_p \|(D_p f^k|_{E^u})^{-1}\|\cdot \|D_p f^k|_{E^c}\|\cdot \|(D_p f^k|_{E^c})^{-1}\|^r<1.
\end{eqnarray*}
When $f$ is $C^r$ and dynamically coherent, the first of these three inequalities is $r$-normal hyperbolicity and implies that the leaves of  $\W^c, \W^{cs},\W^{cu}$ are $C^r$. If $f$ is $C^{r+1}$  and dynamically coherent they also imply the stable and unstable \emph{holonomy} and $E^s, E^u$ are $C^r$ along $\W^c$, cf. \cite{PSW, W}. We say that $f$ is   \emph{center bunched} if it is center $1$-bunched.   If $E^c$ is $1$-dimensional, then $f$ is automatically center bunched. All systems we consider here have $1$-dimensional center and thus are center bunched.

Unfortunately, the term ``bunching" is also used in a completely different way, to describe stable (and unstable) expansion rates for contracted (and expanded) subbundles.
\begin{definition}\label{def: stable bunched}
Let $f\in \Diff(M)$, and suppose that $E\subseteq TM$ is a continuous $Df$-invariant, 
subbundle.
For $r>0$, we say that $Df\vert_{E}$ is {\em $r$-bunched} if there exists $k\geq  1$  such that:
\[\sup_{p\in M}\; \max \{ \| D_p f^k|_{E} \|,\;\|(D_p f^k|_{E})^{-1} \|\cdot \| D_p f^k|_{E} \|^{r}\}<1.\]
The smaller $r$ is, the harder it is to satisfy $r$-bunching (as opposed to {\em center} $r$-bunching, which is easier to satisfy for small $r$) .  If $Df|_E$ is conformal, then it is $r$-bunched, for all $r>1$.
If $f$ is partially hyperbolic, we say that the stable (resp unstable) spectrum of $f$ is $r$-bunched if 
 $Df\vert_{E^s_f}$ (resp. $Df^{-1}\vert_{E^u_f}$) is $r$-bunched. 
\end{definition}

\subsubsection{Accessibility}\label{ss: accessibility}
  The foliations $\W^u_f$ and $\W^s_f$ of a partially hyperbolic diffeomorphism $f\colon M\to M$ induce an equivalence relation on $M$: we say that $x,y\in M$ are in the same {\em accessibility class} if they can be joined by an $su-$path, that is, a piecewise $C^1$ path such that every piece is contained in a single leaf of $\W^s_f$ or a single leaf of $\W^u_f$.  Then $f$ is \emph{accessible} if $M$ consists of a single accessibility class.
At the opposite extreme of accessibility is joint integrability:  $E^u_f$ and $E^s_f$ are \emph{jointly integrable} if there exists an $f-$invariant foliation $\W^{H}$ with $C^1$ leaves everywhere tangent to the bundle $E^u\oplus E^s$.  In this case, unique integrability of $E^u, E^s$ implies that accessibility classes are the leaves of the foliation $\W^{H}$.

 Pugh and Shub conjectured that if $f\in \Diff^2_{\mathrm{vol}}(M)$ is partially hyperbolic and accessible, then $f$ is ergodic.  This was proved for center bunched $f$ by Burns--Wilkinson \cite{BW}.  In particular, acessibility implies ergodicity for systems with $1$-dimensional center bundle, and {\em stable accessibility} --- i.e., accessibility that persists under $C^1$-small perturbations --- implies stable ergodicity.

Pugh and Shub also conjectured that stable accessibility is a dense property among $C^r$ partially hyperbolic diffeomorphisms, volume-preserving or not.  Dolgopyat--Wilkinson \cite{DW} proved $C^1$ density of stable accessibility among all $C^r$ partially hyperbolic diffeomorphisms, and Hertz-Hertz-Ures \cite{RR} proved $C^r$ density (for any $r$) among the systems with $1$-dimensional center foliation.


The next proposition will be used in the proofs of in Theorems \ref{main: dich} and \ref{main: thm pr}.

\begin{prop}\label{lemma: dich atom rot} Let $f_0, A$ 
be as in Theorem \ref{main: dich}, and let $f\in \Diff^2_{\mathrm{vol}}(\TT^d)$ be  a $C^1-$small, ergodic perturbation of $f_0$. Then   
\begin{enumerate}\item  $f$ is a fibered partially hyperbolic system. There is an equivariant fibration $\pi:\TT^d\to \TT^{d-1}$ such that $\pi \circ f = T_{A}\circ \pi$.  The fibers of $\pi$ are the leaves of the center foliation $\W^c_f$ by circles, where $\W^c_f$ is given by Theorem \ref{main: HPS}.
\item One of the following holds:
\begin{enumerate}
\item there exists a full volume set  $S\subset \TT^d$ and $k\in \NN$ such that $S$ meets every leaf
of $\W^c_f$
in exactly $k$ points, i.e.  volume has atomic disintegration along $\W^c_f$;  \label{case:atomic}

\item $f$ is accessible, $\W^c$ is absolutely continuous, and the disintegration of $\mathrm{vol}_{\TT^d}$ has a continuous density function on the leaves of $\W^c$; \label{case: lebesgue access}

\item  $f$ is topologically conjugate to $T_{A}\times R_\theta$,
for some $\theta\notin \QQ/\ZZ$ by a homeomorphism that is $C^1$ along the leaves of $\W^c_f$.  \label{case: rigid}

\end{enumerate}
\end{enumerate}
\end{prop}

\begin{proof}[Proof of Proposition~\ref{lemma: dich atom rot}] (1) follows from Theorem \ref{main: HPS}.

The proof of (2) involves  an analysis of the accessibility classes of $f$.  
The first possibility is that $f$ has an open accessibility class $U\neq \emptyset$.  
Since $f$ is an ergodic, fibered partially hyperbolic system, with one dimensional fibers,  \cite[Theorem C (2)]{AVW2} implies that either 
  $f$ is accessible and $\mathrm{vol}$ has absolutely continuous disintegration, or
 $\mathrm{vol}$ has atomic disintegration along the leaves of $\W^c_f$.
The conclusions follow immediately.

The second possibility is that there is no open accessibility class; that is, the extreme case of joint integrability holds \cite{RRU}. Assume then that $E^s\oplus E^u$ is integrable, 
 tangent to a foliation $\W^H$.

Recall that $f$ is a $C^1$-small perturbation of  an isometric extension  $f_{0} :=\left(T_{A}\right)_{\rho_{0}}$, where  $\rho_{0} : \mathbb{T}^{d-1} \rightarrow \mathbb{R}$.  If $\rho_0$ is not cohomologous to a constant function, then it is stably accessible  \cite{BW99}.
Since we are assuming there is no open accessibility class, we may assume that $\rho_0$ is cohomologous to a constant function. Liv\v sic's theorem implies that by conjugating by a $C^\infty$ diffeomorphism of $\TT^d$ covering the identity on $\TT^{d-1}$, we may assume that $\rho_0=\theta_0$ is constant.  This implies that $E^u_{f_0}\oplus E^u_{f_0}$ is integrable, the leaves of the integral foliation $\W^H_{f_0}$ are compact, and $f_0$ is conjugate   to the product of $T_{A}$ with  a rotation.  We  show that the same holds for $f$.

\begin{lemma}\label{lemma: top rig integrab} If the distribution $E^u_f\oplus E^s_f$ is integrable then the leaves of 
its integral foliation $\W^H$ are compact.  Each leaf of  $\W^H$  intersects each leaf of $\W^c$ in exactly one point.
\end{lemma}
\begin{proof}[Proof of Lemma~\ref{lemma: top rig integrab}]  We show that  the monodromy representation on  the circle bundle $\TT^d\to \TT^{d-1}$ induced by the foliation $\W^H$, combined with the action of $f$ on an invariant $\W^c_f$ fiber, gives a $C^1$ action of an abelian-by-cyclic group.  These actions have well-known rigidity properties, which we exploit to show that the monodromy part of the representation must have finite image.

To this end, fix $x_0\in \TT^{d}$ such that $f\left(\W^c_f(x_0)\right) = \W^c_f(x_0)$, and consider the map
\[H\colon \pi_1(\TT^{d-1},\pi(x_0)) \cong  \ZZ^{d-1} \to \mathrm{Homeo}^+(\W^c_f(x_0))
\]
defined by $\W^H_f$-holonomy along lifted paths:  for  $y\in \W^c(x_0)$ and $\gamma:[0,1]\to \TT^{d-1}$ in the class $[\gamma]$, consider the unique continuous lift $\gamma^{y}:[0,1]\to \TT^{d}$ such that $\gamma^{y}(0) =  y$, $\gamma^{y}[0,1]\subset \W^H(y)$, and $\pi\circ\gamma^{y}=\gamma$. We then define
\[H([\gamma])(y) := \gamma^{y}(1).\]
Then $H$ is a homomorphism, which we call the {\em monodromy representation.}

We remark that for $f = f_0$, where the leaves of $\W^H$ are compact, the map $H$ is trivial.

\begin{lemma}\label{lemma: H f rel}For any  $[\gamma]\in \pi_1(\TT^{d-1})$:
\[H({T_{A}}_\ast [\gamma])= f \circ H([\gamma])\circ f^{-1}.\]
\end{lemma}
\begin{proof}(cf. \cite{NT}) Fix $y\in \W^c(x_0)$ and $\gamma:[0,1]\to \TT^{d-1}$ in the class $[\gamma]$.  Consider the lift $\gamma^{f^{-1}(y)}$ of $\gamma$ starting at $f^{-1}(y)$.  Note that the path $f\circ \gamma^{f^{-1}}(y)$ is a lift of $T_{A}\circ\gamma$ starting at $y$ and tangent to $\W^H$ (by $f$-invariance of $\W^H$).  But $(T_{A}\circ\gamma)^{y}$ is the unique such lift.  It follows that 
$(T_{A}\circ\gamma)^{y} = f\circ \gamma^{f^{-1}(y)}$; evaluating both paths at their endpoint gives the desired conclusion. This completes the proof of Lemma~\ref{lemma: H f rel}.
\end{proof}

\begin{lemma}\label{lemma: WH C2} If $E^u_f\oplus E^s_f$ is integrable, then its integral foliation $\W^H$ is a $C^1$ foliation.
\end{lemma}
\begin{proof} Since $E^c$ is $1$-dimensional, $f$ is center bunched.  Then \cite[Theorem B]{PSW97} implies that the leaves of $\W^s_f$ and $\W^u_f$ uniformly $C^1$ subfoliate the leaves of $\W^{cs}_f$ and $\W^{cu}_f$, respectively. This implies that the stable and unstable holonomy maps between  $\W^c_f$ leaves is $C^1$.  The holonomy maps along $\W^H$ between $\W^c$ leaves can be written as a composition of stable and unstable holonomies, and thus are uniformly $C^1$.  

A foliation with unifomly $C^1$ leaves and uniformly $C^1$ holonomy maps is $C^1$ (see \cite{PSW97}), and thus $\W^H$ is a $C^1$ foliation.
\end{proof}

Lemma~\ref{lemma: WH C2} implies that the monodromy representation $H$ above has $C^1$  image:
\[H(\pi_1(\TT^{d-1},\pi(x_0))) \subset \Diff^1(\W^c_f(x_0)).\]
Note that the induced  action of $T_{A}$ on $\pi_1(\TT^{d-1},\pi(x_0))$ is just matrix multiplication by $A$ under the natural  identification $\pi_1(\TT^{d-1},\pi(x_0)) \cong  \ZZ^{d-1}$.  Consider the abelian-by-cyclic group $\Gamma_{A}=  \ZZ \ltimes_{A} \ZZ^{d-1}$ defined by
\[
\Gamma_{A}
 :=\left\langle a, e_{1}, \ldots, e_{d-1} :  e_{i} e_{j}=e_{j} e_{i}, \quad a e_{i} a^{-1}= e_{1}^{\alpha_{1, i}}\cdots e_{d-1}^{\alpha_{d-1, i}}\right\rangle,
\]
where $A = (\alpha_{i,j})$.
Lemma~\ref{lemma: H f rel} implies that we have a representation  $\eta\colon \Gamma_{A}\to \Diff^1(\W^c_f(x_0))$ defined by
\[\eta(a) := \left. f\right|_{\W^c_f(x_0) };\;\, \eta(e_i) := H([e_i]).\]
Such representations are quite rigid.  In particular, we have
\begin{theorem}\label{theorem: Navas}\cite[Theorems 1.3, 1.7 and 1.10]{BMNR} For any representation $\eta\colon  \Gamma_{A}\to \Diff^1(\TT)$, either the image $\eta(\Gamma_{A})$ is abelian
or  there exists an integer $m\geq 1$, a real eigenvalue $\lambda$ of $A$,  and a point $x\in \TT$ such that $\eta(a^m)(x)=x$, and 
$\eta(a^m)'(x) = \lambda^m$.
\end{theorem}

Applying Theorem~\ref{theorem: Navas} to the situation at hand,  we obtain that either $\eta(\Gamma_{A})$ is abelian,  or there exists $m\geq 1$ such that  $\eta(a^m) = \left.f^m\right|_{\W^c_f(x_0)}$ has fixed point with derivative $\lambda^m$, for some real eigenvalue $\lambda$ of $A$.  But since $A$ is hyperbolic, the  eigenvalues of $A$ are bounded in absolute value away from $1$. Since  $\left.f_0\right|_{\W^c_{f_{0}}(0)}$  is a rotation by $\theta_0$, whose derivative is everywhere $1$, if $f$ is sufficiently $C^1$-close to $f_0$, this is impossible.

Hence $\eta(\Gamma_{A})$ is abelian, which implies that
\begin{equation}\label{eqn: ABC abelian image}
\hbox{$\eta(a e_{i} a^{-1}) = \eta(e_i) = \eta(e_{1})^{\alpha_{1, i}}\cdots \eta(e_{d-1})^{\alpha_{d-1, i}}$, for $i=1,\ldots, d-1$.}
\end{equation}
Since $1$ is not an eigenvalue of $A$, it follows that $A-\id$ is invertible over $\QQ$, and so equations (\ref{eqn: ABC abelian image}) imply that the group generated by $\eta(e_{1}), \ldots, \eta(e_{d})$
is finite, of order $k \leq |\det(A-\id)|$.

Thus the image of $H$ is isomorphic to group of order $k$, and the leaves of $\W^H$ are  compact, meeting each leaf of $\W^c_f$ in exactly $k$ points.  We claim that $k=1$. As observed above, for $f_0$, the image of $H$ is trivial, and the leaves of $\W^H_{f_0}$  are horizontal.  Since $f$ is  close to $f_0$, the leaves of $\W^H_f$ are nearly horizontal; in particular if  $d_{C^1}(f_0,f)$ is sufficiently small, either $k=1$ or all the orbits of $H$ on $\W^c_f(x_0)$ have arbitrarily small diameter (since $k$ is bounded by $|\det(A-\id)|$). Thus by the following theorem of Newman \cite{New}, we have $k=1$. 
\begin{maintheorem} Let $N$ be a connected topological manifold endowed with a metric. Then
there is $\epsilon > 0$ such that any non-trivial action of a finite group on $N$ has an orbit of
diameter larger than $\epsilon$.
\end{maintheorem}
This completes the proof of  Lemma~\ref{lemma: top rig integrab}.  \end{proof}

To finish the proof of Proposition~\ref{lemma: dich atom rot}, we construct a $C^1$ projection $\mathrm{Pr}^c\colon \TT^d\to \W^c_{f}(x_0)$ sending $x$ to the unique point of intersection of $\W^H_f(x)$ and $\W^c_{f}(x_0)$. Let $\sigma \colon \W^c_{f}(x_0) \to\TT$ be a $C^1$ diffeomorphism  and define $\zeta \colon \TT^d\to \TT^d$
by $\zeta(x):= (\pi(x), \sigma \circ \mathrm{Pr}^c (x))$.  Then $\zeta$ conjugates $f$ to $T_{A} \times g$, where $g\colon \TT\to\TT$ is a diffeomorphism preserving a smooth ergodic measure.  By a further $C^1$  change of coordinates, we may assume that $g$ is an irrational rotation $R_\theta$.  We are thus in case~\ref{case: rigid}.  This completes the proof of Proposition~\ref{lemma: dich atom rot}.
\end{proof}

\subsubsection{Estimate of the H\"older exponents of  leaf conjugacies in the presence dominated splittings}\label{sec: multi distri PSW}
As in Theorems \ref{main: dich} and \ref{main: thm pr} let $f_0\in \Diff^\infty_{\mathrm{vol}}(\TT^d)$ be an isometric extension of an automorphism $T_{A_f}$  on $\TT^{d-1}$, where $A_f\in \SL(d-1,\ZZ)$  is hyperbolic (we do not need irreducibility here).  We denote by $P:\TT^d\to \TT^{d-1}$ the projection along 
the $\W^c_{f_0}$ leaves, which is just projection onto the first factor in $\TT^{d-1}\times\TT$. Under the identification $T\TT^{d-1}\cong\TT^{d-1}\times \RR^{d-1}$, the action of $DT_{A_f}$ is just $T_{A_f}\times A_f$, and 
$T\TT^{d-1}= \TT^{d-1}\times (\oplus V^i)$ is the $T_{A_f}-$invariant dominated splitting, where 
$\RR^{d-1} = \oplus V^i$ is the decomposition into  Lyapunov subspaces  of $A_f$.  

There is a $Df_0$-invariant dominated splitting $TM=\oplus E^i_{f_0}$ projecting to the dominated splitting for $T_{A_f}$, so that $D_{p} P(E^i_{f_0}) = \{P(p)\}\times V^i$, for each $i$.  Moreover the Lyapunov exponent of $Df_0\vert_{E^i_{f_0}}$ is equal to Lyapunov exponent of $A_f\vert_{V^i}$.

As in Theorem \ref{main: thm pr},  we now assume that $f\in \Diff^2_{\mathrm{vol}}(\TT^d)$ is a $C^1$-small perturbation of $f_0$. Then $Df$ also preserves a dominated splitting $TM=\oplus E_f^i$. By Theorem \ref{main: HPS}, $f$ is a fibered partially hyperbolic system,  and $(f_0;\W^c_{f_0})$ is leaf
conjugate to $(f;\W^c_f)$ by a \textit{bi-H\"older} continuous homeomorphism $h^c : \TT^d\to \TT^d$. 
The leaf conjugacy $h^c$ is canonical in the sense that
\begin{equation}\label{eqn: h pi P}
\pi\circ h^c=P,
\end{equation}
where $\pi$ is defined in Proposition~\ref{lemma: dich atom rot}. For the estimate of the bi-H\"older exponents of $h^c$, cf. \cite{PSW}.   In this context we can give a concrete description of how $h^c$ is constructed.  Fixing a smooth normal bundle $\cN$ to $E^c$, the map $h^c = h^c_\cN$ is defined by 
\[h^c(x) = \pi^{-1}(P(x))\cap \cD_\epsilon(x),
\]
where $\{\cD_\epsilon(x) : x\in M\}$ is the smooth family of embedded disks defined by
\[\cD_\epsilon(x) = \exp_x\left( \{tv: t\in [0,\epsilon), v\in \cN(x)\}\right).
\]
If $f$ is sufficiently $C^1$ close to $f_0$ and $\epsilon>0$ is sufficiently small, then $h^c$ is a well-defined homeomorphism that is smooth along the leaves of $\W^c_{f_0}$ (as smooth as the leaves of $\W^c_{f_0}$).

It is easy to see that $E^i_{f_0}, E^i_{f_0}\oplus E^c_{f_0}$ are integrable; we denote by $\W^i_{f_0}, \W^{ic}_{f_0}$ their integral manifolds respectively. In general, $E^i_{f}$ and $E^i_{f}\oplus E^c_{f}$ might not be integrable.

\begin{lemma}\label{lemma: multi distri PSW} Suppose  $E^i_f, E^i_f\oplus E^c_f$ are integrable and their integral manifolds are denoted by $\W^i_f, \W^{ic}_f$ respectively. If, for some normal bundle $\cN$, the map $h^c = h^c_\cN$ sends $\W^{ic}_{f_0}$ to $\W^{ic}_f$, for each $i$,  then
 the H\"older exponents of $h^c$ and $(h^c)^{-1}$ approach $1$ as $d_{C^1}(f,f_0)\to 0$.
\end{lemma}
\begin{proof}

Fix $i$ and consider the foliation  $\W^{ic}_f$.  Its leaves  are jointly foliated by $\W^i_f$ and the uniformly compact foliation $\W^c_f$.  By taking $f^{-1}$ if necessary, we may assume that the leaves of $\W^{i}_f$ are uniformly contracted by the dynamics.  Let $\lambda_i<0$  be the corresponding Lyapunov exponent for $A_f\vert_{L_i}$.  Since $f_0$ is an isometric extension of a linear map, for any $\epsilon>0$ we may choose a continuous adapted metric on $\TT^{d}$ such that for all $f$ sufficiently $C^1$-close to $f_0$, for all $p\in \TT^d$, and all $v\in E^i_{f}(p)$:
\[e^{\lambda_i - \epsilon}\|v\|  \leq  \|D_p f(v)\| \leq e^{\lambda_i + \epsilon}\|v\|.
\]
Let $\mu_i =  e^{\lambda_i - \epsilon}$ and $\nu_i = e^{\lambda_i + \epsilon}$.
If $f$ is sufficiently $C^1$-close to $f_0$, then for any $w\in M$ and $w'\in W^i_f(w,loc)$, if $f^{-j}(w') \in \W^i_f( f^{-j}(w'),loc)$ for $j=0,\ldots n$, then 
\[   \nu_i^{-n} \,d(w,w') \leq   d_{\W^i_f} (f^{-n}(w), f^{-n}(w'))  \leq   \mu_i^{-n}\, d(w,w').\]
(This is easily proved by induction on $n$).

 Consider the restriction of $h^c$ to $\bigsqcup \W^{ic}_{f_0}$, whose image is $\bigsqcup \W^{ic}_{f}$, sending  $\W_{f_0}^c$ leaves smoothly to $\W_{f}^c$ leaves.  Now $h^c$ does not necessarily send $\W^{i}_{f_0}$ leaves to $\W^{i}_{f}$, but we can estimate the H\"older exponent of $h^c$ restricted to $\W^i_{f_0}$ leaves via a standard argument, which we now describe.

Fix $\eta>0$ such that for all $w, w'\in \bigsqcup \W^{ic}_{f_0}$,   with $d(w,w') < \eta$, then for any $z\in \W^c_f(w)$, there is a unique point $z'$ in  $\W^i_f(z,loc)\cap \W^c_f(w')$ and the distance between $z$ and $z'$ is uniformly comparable to the distance between $z$ and $z'$ as measured along $\W^i_{f}(z,loc)$.  This is possible because the foliation $\W^c_f$ has uniformly compact leaves. Next fix a small constant $\delta>0$ such that  $d(w,w')<\delta$ implies $d(h^c(w), h^c(w'))<\eta$.

Now let $x\in M$ and $x'\in \W^i_{f_0}(x)$. Let $y= h^{c}(x)$ and $y' = h^c(x')$. We want to estimate $d(y,y')$ in terms of $d(x,x')$.   Let $z=\W^i_f(y)\cap \W^c_f(y')$.  By the construction of $h^c$ using the smooth normal bundle $\cN$, we have that 
$d(y',z) = O(d(y,z))$, so it suffices to estimate $d(y,z)$ in terms of $d(x,x')$.

We may assume that $d(x,x') < \delta$.  Fix $n\geq 0$ such that  $d(x,x')\in [\delta\mu_i^{n+1}, \delta\mu_i^{n})$.
Since $x'\in \W^i_{f_0}(x,loc)$, we have $d(f_0^{-n}(x), f_0^{-n}(x')) <  \mu_i^{-n} d(x,x') < \delta$.
By our choice of $\delta$, we have that  $d(h^{c}(f^{-n}(x)), h^c(f^{-n}(x'))) < \eta$.
Since $h^c$ is a leaf conjugacy, $h^{c}(f_0^{-n}(x))\in \W_f^c(f^{-n} (y))$ and $h^{c}(f_0^{-n}(x'))\in \W_f^c(f^{-n} (y')) = \W_f^c(f^{-n} (z))$.   Since $f^{-n} (y) \in \W^i(f^{-n} (z),loc)$, our choice of $\eta$ implies that $d(f^{-n} y, f^{-n} z)$ is comparable to the distance measured along 
$\W^i_f$, which is at least  $\nu_i^{-n} d(y,z)$.  Thus $d(y,z) = O(\nu_i^{-n}) = O(\mu_i^{-n\beta}) = O(d(x,x')^\beta)$, where
\[\beta = \frac{\log \mu_i }{\log{\nu_i}} =  \frac{\lambda_i + \epsilon} {\lambda_i - \epsilon}.
\]  
Since we may choose $\epsilon>0$  arbitrarily small by setting $d_{C^1}(f_0,f)$ small enough, this shows that we may choose $\beta$ arbitrarily close to $1$.

This shows that $h^{c}$ is uniformly $\beta$-H\"older continuous along $\W^{ic}_{f_0}$-leaves, for all $i$.  It is thus $\beta$-H\"older continuous.

A similar argument  (reversing the roles of $f_0$ and $f_0^{-1}$) shows that $(h^{c})^{-1}$ is $\beta$-H\"older continuous.
\end{proof}

\subsection{Some Pesin theory}

We will also use the following well-known corollaries of Pesin theory. Let $f$ be a $C^{r}(r>1)$ diffeomorphism of a closed $d-$manifold $M$, let $\nu$ be an $f-$invariant ergodic probability measure, and let $\lambda^{\hbox{\tiny{$\mathrm{max}$}}} =\lambda_1 \geq \cdots\geq \lambda_d\geq \lambda^{\hbox{\tiny{$\mathrm{min}$}}}$ be the Lyapunov exponents of $Df$ with respect to $\nu$.

For $x\in M, \,\delta> 0$, and $\lambda< 0$, we define the {\em local stable set}\[\W^s(x,  \lambda, \delta) := \{y\in M : d(f^n(x), f^n(y)) \leq  \delta \exp(\lambda n), \;\forall n \geq  0\}.\]
The set of  \emph{regular} points  for $(f,\nu)$ in $M$ (also called the Lyapunov--Perron regular points, cf. \cite{BP}) have full $\nu$-measure in $M$ and the following important property.
\begin{prop}[Stable manifold theorem]\label{prop: stb mfd}Fix $\lambda<0$ such that $\lambda_{k+1}<\lambda<\lambda_k$ holds for some $k$. Then for any regular point $x$, the local stable set $\W^s(x, \lambda, \delta)$ is a $C^r$  embedded disk in  $M$ for small enough $\delta$. The dimension of the disk is  $d-k$.
\end{prop} 
We call the set $\W^s(x, \lambda, \delta)$ defined by the Proposition the \emph{local Pesin stable manifold}, and we denote it by $\W^s(x,\lambda, loc)$ (cf. \cite{P77} for a concrete estimate on $\delta$). 
Suppose $x$ is a regular point  and $\W^s(x,\lambda, loc)$ is defined as above. The {\em global Pesin manifold} (of $\W^s(x, \lambda, loc)$) is defined by 
\[\W^s(x,\lambda) =\cup_{n=0}^\infty f^{-n}(\W^s(f^n(x), \lambda, loc)).\]


We also  obtain the following criterion for a diffeomorphism to contract an invariant bundle.
\begin{lemma}\label{lemma: h dim reg chart} Let $f\colon M\to M$ be a $C^{1+}$ diffeomorphism, and let $\W$ be an $f$-invariant foliation with $C^{1+}$ leaves.  Suppose there exist  $\kappa_1 < \kappa_2 <  0$ such that for all $x\in M$, there exist $\delta>0$ and $N\in \NN$ such that  for all $y\in \W(x)$:
\[d(x,y)<\delta \implies  e^{\kappa_1 n} \leq d(f^n(x), f^n(y)) \leq  e^{\kappa_2 n},
\]
for all $n\geq N$.

Then for all  $\epsilon>0$, there exists $N'\in \NN$ such that for all $v\in T\W$, and all $n\geq N'$, we have
\[ e^{ (\kappa_1 -\epsilon) n}\|v\| \leq \|Df^n v\| \leq  e^{ (\kappa_2+\epsilon) n}\|v\|.\]
\end{lemma}
\begin{proof} 
Assume the hypotheses; we show how to establish the upper inequality  $\|Df^n v\| \leq  e^{(\kappa_2+\epsilon) n}\|v\|$ (the lower inequality is similarly proved).

By  Lemma~\ref{lemma: est cocyc}, it suffices to show that for every $f$-invariant, ergodic measure $\nu$, the top Lyapunov exponent of the cocycle $\left. Df\right|_{T\W}$ with respect to $\nu$ is at most $\kappa_2$.  To this end, fix $\nu$, and let $\beta_k< \beta_{k-1} <\cdots < \beta_1<0$ be the Lyapunov exponents of the cocycle  $\left. Df\right|_{T\W}$ with respect to $\nu$.  Let 
\[T\W = E_k \oplus \cdots \oplus E_1\]
be the corresponding Oseledec decomposition.  

Using the graph transform argument in  \cite[Theorem 3.16]{PSh} in restriction to the leaves of $\W$, one can construct for each $i$ a measurable family of $C^{1+}$ disks $\cD_i(x)$, defined over a full $\nu$-measure set $S_\nu \in M$, with the following properties, for all $x\in S_\nu$:
\begin{itemize}
\item $\cD_i(x)\subset \W(x)$,
\item $T_x\cD_i(x) = E_i(x)$,
\item $f(\cD_i(x)) \subset \cD_i(f(x))$, and
\item for every $\epsilon>0$,  there exists $N=N_x$ such that for all $y\in \cD_i(x)$ and all $n\geq N$, we have
\[ e^{ (\beta_i-\epsilon)n} \leq  d(f^n(x), f^n(y)) \leq e^{ (\beta_i+\epsilon)n} .\]
\end{itemize}

Fix $\epsilon>0$ and  a Pesin regular point $x$ for $\nu$, and consider $\cD_1(x)$.  For $n$ sufficiently large, the action of $f^n$ in restriction to  $\cD_1(x)$ is a uniform contraction by a factor bounded below by $e^{(\beta_1-\epsilon)n}$.  On the other hand, the hypothesis implies that this contraction factor is bounded above by  $e^{\kappa_2 n}$.  It follows that $\beta _1 <  \kappa_2+\epsilon$. Since $\epsilon>0$ was arbitrary, we conclude that $\beta_1\leq \kappa_2$.
\end{proof}

\subsection{Normal forms for uniformly contracting foliations}\label{sec: normal form}

We will use non-stationary normal form theory  to upgrade the regularity of certain homeomorphisms in the centralizer of the partially hyperbolic systems under consideration. 

Let $f$ be a diffeomorphism of a closed manifold $M$, and
let $\W$ be an $f-$invariant  foliation of $M$ with uniformly $C^1$ leaves.  We assume that $f$ uniformly contracts the leaves $\W$.
Let $E = T\W$ be the tangent bundle to $\W$. We denote by $F\colon E\to E$ the bundle automorphism induced by the derivative of $f$:
$F_x =Df |_{Tx\W}\colon E_x\to E_{f x} $.
Then $F$ induces a bounded linear operator $F^\ast$ on the space of continuous sections of $E$ by $F^\ast v(x) =
F(v( f^{−1}x))$. The spectrum of the complexification of $F^\ast$ is called the \emph{Mather spectrum} of $F$. If the non-periodic points of $f$ are dense in $M$, then the Mather spectrum consists of finitely many closed annuli   $A_i$, $i = 1, \ldots , \ell$, centered at $0$ and bounded by circles of radii $e^{\lambda_i}$ and $e^{\mu_i}$, with $\lambda_i = \lambda_i(F)$ and $\mu_i = \mu_i(F)$ satisfying
\begin{equation}\label{eqn: Exponent Ordering}
\lambda_1 \leq \mu_1 < \lambda_2 \leq \mu_2  <\cdots <\lambda_\ell \leq \mu_\ell < 0;
\end{equation}
see \cite{M65, P04}.

The spectral intervals $\{[\lambda_i(F), \mu_i(F)]: i=1,\ldots, \ell\}$ correspond to a splitting of the bundle $E$ into a direct sum
\[E = E^1\oplus \cdots \oplus E^\ell\]
of continuous, $F-$invariant sub-bundles such that Mather spectrum of $F|_{E^i}$ is contained in the annulus $A_i$ (this splitting is thus dominated and invariant under perturbations of $F$). This can be expressed using the Lyapunov metric \cite{GK}: for each $i=1,\dots, \ell$ and $\epsilon>0$, there exists a continuous metric $
\|\cdot\|_{x,\epsilon}$ on $E^i$ such that $$e^{\lambda_i-\epsilon}\cdot \|v\|_{x,\epsilon}\leq \|F_x(v)\|_{f(x)
,\epsilon}\leq e^{\lambda_i+\epsilon}\cdot \|v\|_{x,\epsilon}, \forall v\in E^i_x.$$

\begin{definition}\label{def: narrow band}We say that the bundle automorphism $F$  has \emph{narrow band spectrum} if $\mu_i(F) + \mu_\ell(F)  < \lambda_i(F)$, for $i = 1, \ldots , \ell$.
\end{definition}

For vector spaces $E$ and $\bar E$ we say that a map $P : E\to  \bar E$ is {\em polynomial} if with respect to some bases of $E$ and $\bar E$, each component of $P$ is a polynomial. A polynomial map
$P$ is {\em homogeneous of degree $n$} if $P(av) = a^n P(v)$ for all $v \in E $ and $a \in \RR$. More generally,
for a given splitting $E = E^1 \oplus\dots\oplus E^\ell$ we say that $P\colon E \to \bar E$ has {\em homogeneous type}
$s = (s_1,\ldots,s_\ell)$ if for any real numbers $a_1,\ldots , a_\ell$ and vectors $t_j \in E^j$,  $j = 1,\ldots ,\ell$, we
have
\[P(a_1t_1 +\dots+ a_\ell t_\ell) = a^{s_1}
_1 \cdots a^{s_\ell}_\ell P(t_1 +\cdots+ t_\ell).\]

Suppose $E = E^1 \oplus\cdots\oplus E^\ell$, $\bar E = \bar E^1 \oplus\cdots\oplus \bar E^\ell$ and $P\colon E \to \bar E$ is a
polynomial map. Split $P$ into components $P_i \colon E\to  \bar E^i$ and write
$P = (P_1, \ldots , P_\ell)$.  Let $\lambda=(\lambda_1,\ldots, \lambda_\ell)$ and $\mu = (\mu_1,\ldots,\mu_\ell)$ with  $\lambda_1 \leq \mu_1 <\cdots < \lambda_\ell \leq \mu_\ell<0$. 
 We say that
$P$ is of {\em $(\lambda, \mu)$ sub-resonance type} if for each $i=1,\ldots, \ell$, there exists
$s = s(i) = (s_1,\ldots, s_\ell)$  satisfying the sub-resonance relation
\[\lambda_i \leq \sum_{j=1}^\ell  s_j\mu_j,\]
such that 
$P_i$ has homogeneous type $s$.

Now we can state the main results in this section.

\begin{maintheorem}\label{thm: normal form} \cite{KalininNormalForms, KS06, Sad}  
Let $f$ be a $C^{r}$ diffeomorphism of a closed manifold $M$, and let $\W$ be an $f-$invariant topological foliation of $M$ with uniformly $C^r$ leaves.  Suppose that the leaves of $\W$ are contracted by $f$ and that either: the spectrum of $F=\left.Df\right|_{T\W}$ is $r_{\W}$-bunched, for some $r_{\W} \leq 2$. (See Definition~\ref{def: stable bunched}); or $F$ has narrow band spectrum (see Definition~\ref{def: narrow band}).

Fix  $r > r_{\W}$ (in the bunched case) or  $r > {\lambda_1(F)}/{\mu_\ell(F)}$, setting $\lambda=(\lambda_1(F),\ldots, \lambda_\ell(F))$ and $\mu = (\mu_1(F),\ldots,\mu_\ell(F))$ (in the narrow band case).   Then there exists a family $\{H_x\}_{x\in M}$ of $C^r$ diffeomorphisms
$H_x\colon \W_x \to E_x = T_x\W$ such that
\begin{enumerate}
\item $P_x = H_{fx} \circ f \circ H^{−1}_x\colon E_x\to E_{fx}$ is a linear map (in the bunched case) or a polynomial map of $(\lambda, \mu)$ sub-resonance type (in the narrow band case) for each $x\in M$;
\item $H_x (x) = 0$ and $D_xH_x$ is the identity map for each $x\in M$;
\item $H_x$ depends continuously on $x\in M$ in the $C^r$ topology and  is jointly $C^r$ in $x$
and $y\in \W_x$  along the leaves of $\W$;
\item $H_y\circ H^{−1}_x : E_x \to E_y$ is a linear map (in the bunched case) or a  polynomial map  of $(\lambda,\mu)$ sub-resonance type  (in the narrow band case) for each $x \in M$ and
each $y \in  \W_x$; and
\item if $g$ is a homeomorphism of $M$  that commutes with $f$, preserves $\W$, and is $C^{s}$ along the leaves of $\W$, with 
$s>r_\W$ (in the bunched case) or $s>{\lambda_1(F)}/{\mu_\ell(F)}$ (in the narrow band case), then the maps $H_x$ bring $g$ to a normal form as well, i.e. the map
$Q_x = H_{fx} \circ g \circ H^{−1}_x$ is a  linear map (in the bunched case) or a polynomial  of $(\lambda,\mu)$ sub-resonance type (in the narrow band case) for each $x\in M$.
\end{enumerate}
\end{maintheorem}

\begin{definition}\label{def:narrow band and bunched partially hyperbolic} Let $f$ be a $C^\infty$ partially hyperbolic diffeomorphism of a closed manifold $M$.   
We say that $f$ has {\em $r$-bunched spectrum} if the cocycles $F^s=\left.Df\right|_{E^s_f}$ and $F^u=\left.Df^{-1}\right|_{E^u_f}$ are $r$-bunched. (see Definition~\ref{def: stable bunched}); we call the infimum of such $r$ the {\em critical regularity $r(f)$ of $f$}.

We say that $f$ has  {\em narrow band spectrum} if the cocycles $F^s$ and $F^u$ have narrow band spectrum. In this case, we define the {\em critical regularity $r(f)$} of $f$ by
\begin{equation}\label{eqn: Critical Regularity}
r(f):=\max\left(\frac{\lambda_1^s(f)}{\mu_{\ell_s}^s(f)}, \frac{\lambda_1^u(f)}{\mu_{\ell_u}^u(f)}\right),
\end{equation}
where $\mu_i^*(f) := \mu_i(F^*)$,  $\lambda_i^*(f) :=\lambda_i(F^*)$,
 for $i =1,\ldots, \ell_*$,   for $\ast\in\{s,u\}$.
\end{definition}
We remark that if $f=\psi_{t_0}$, where $\psi_t$ is the geodesic flow over a negatively curved $X$, then
transverse symplecticity of the flow implies that it suffices to check that one of $F^s$ are $F^u$ has $r$-bunched (resp. narrow band) spectrum to verify that $f$ itself has $r$-bunched (resp. narrow band) spectrum in the sense of Definition~\ref{def:narrow band and bunched partially hyperbolic}.

Hasselblatt \cite{Hass} also defines an $\alpha$-bunched condition for Anosov flows, for $\alpha\in(0,2)$.  For a transversely symplectic Anosov flow $\psi_t$, we have that $\psi_t$ is $\alpha$-bunched in the sense of \cite{Hass} if and only if $\psi_1$ has $2/\alpha$-bunched spectrum, in the sense of Definition~\ref{def:narrow band and bunched partially hyperbolic}. The connection between $\alpha$-bunching and pointwise pinching of the curvature in (\ref{eqn: quarter pinched}) is discussed in \cite{Hass}.

\begin{lemma}\label{lem: Loc Symm Narrow Band}  If $\psi_t$ is the geodesic flow over a locally symmetric space $X$, then for any $t_0\neq 0$, the partially hyperbolic map $\psi_{t_0}$ has narrow band spectrum. If $X$ is a real hyperbolic manifold, then $r(\psi_{t_0})=1$, and  if $X$ is locally symmetric but not real hyperbolic, then $r(\psi_{t_0})=2$. 
\end{lemma}

\begin{proof} The geodesic flow on a locally symmetric space has constant expansion and contraction factors on one or two invariant subbundles, depending  on whether $X$ is real hyperbolic or not. 
In particular, the 
Mather spectrum of $D\varphi_{t_0}|_{E^s}$ and $D\varphi_{t_0}^{-1}|_{E^u}$ has either one or two bands, and either
$\lambda_1^s = \mu_1^s =-1=\lambda_1^u = \mu_1^u$,.
in the case where $X$ is real hyperbolic, or
\[\lambda_1^s = \mu_1^s = -2=\lambda_1^u = \mu_1^u; \;  \lambda_2^s = \mu_2^s = -1=\lambda_2^u = \mu_2^u,\]
otherwise. 
Thus  $D\varphi_{t_0}|_{E^s}$  and $D\varphi_{t_0}^{-1}|_{E^u}$ have point Mather spectrum 
(i.e., $\lambda_i^{u/s}=\mu_i^{u/s}$), and the conclusions follow.
\end{proof}

The following lemma follows immediately from the continuity of dominated splittings.
\begin{lemma}\label{lem: NBS Stable} Suppose that $f_0\in \Diff^1(M)$ is partially hyperbolic and has
bunched spectrum, (resp. narrow band spectrum) with critical regularity $r_0 = r(f_0)$. 
Then for any $r> r_0$, if  $f\in \Diff^1(M)$ is sufficiently $C^1$-close to $f_0$,  then $f$ has $r$-bunched spectrum (resp. narrow band spectrum), and $r(f) < r$.
\end{lemma}

Here is our central application of Theorem~\ref{thm: normal form}.

\begin{prop}\label{prop: applic normal form}
Let $f$ be a $C^\infty$ partially hyperbolic diffeomorphism of a closed manifold $M$.  Assume that $f$ has a $1-$dimensional center foliation  $\W^c_f$  with $C^\infty$ leaves.  Suppose that $\varphi=\varphi_t:M\times \RR\to M$ is a flow generated by a continuous vector field $X$  such that $\varphi_t\circ f = f\circ \varphi_t$, for all $t$. Assume that $f$, $\varphi_t$, and $X$ satisfy the following conditions.
\begin{enumerate}
\item    $f$ has  $2$-bunched spectrum,   or $f$ has narrow band spectrum.
\item The vector field $X$ is tangent to $E^c_f$ and uniformly $C^\infty$ along the leaves of 
$\W^c_f$.
\item There exists a dense set $D\subset \RR$ such that for all $t\in D$, $\varphi_t\in \Diff^r(M)$, for some
$r>r(f)$.
\end{enumerate}
Then $\varphi_t$ is a $C^\infty$ flow.
\end{prop}

\begin{proof} Hypothesis (1) implies that for $r>f(f)$, the cocycle $Df\vert_{T\W^s_f}$ satisfies the hypotheses of Theorem \ref{thm: normal form}, and so there exists a non-stationary normalization $\{H_x,x\in M\}$ for $f|_{\W^s_f}$ such that for any   $g\in \Z_{r}(f)$,   the map $H_{gx}\circ g \circ H_x^{-1}$ is a sub-resonance polynomial (with fixed type) as well.

Thus  $\{ H_x\}$ is also a normalization for $\varphi_t$ on $\W^s_f$, for all $t\in D$. Now consider the homeomorphism $\varphi_t$ for an arbitrary fixed $t\in \RR$. Pick $t_k, k=1,2,\dots$ in $D$ such that $\lim_{k\to \infty}t_k=t$. Then the sequence $$H_{\varphi_{t_k}x}\circ \varphi_{t_k}\circ H_x^{-1}: E^s_f(x)\to E^s_f(\varphi_{t_k}(x))$$ uniformly converges to   $H_{\varphi_{t}x}\circ \varphi_{t}\circ H_x^{-1}: E^s_f(x)\to E^s_f(\varphi_{t}(x))$.

But each of $H_{\varphi_{t_k}x}\circ \varphi_{t_k}\circ H_x^{-1}$ is a sub-resonance polynomial (with fixed type), so their $C^0-$limit  is a sub-resonance polynomial as well. Thus $H_{\varphi_{t}x}\circ \varphi_{t}\circ H_x^{-1}$ is uniformly smooth along $E^s_f$, which means $\varphi_{t}$ is uniformly smooth along $\W^s_f$. A similar argument shows that  $\varphi_{t}$ is uniformly smooth along $\W^s_f$.

  Item (1) of Proposition \ref{prop: applic normal form} implies that $\varphi_{t}$ is uniformly smooth along $\W^c$, and the evaluation map $t\mapsto \varphi_t(x), x\in M$ is smooth, uniformly in $x$.  Applying Journ\'e's Lemma as in \cite{AVW},
we obtain that $\{\varphi_t\}$ is a smooth flow and $D=\RR$. 
\end{proof}

\subsection{Partially hyperbolic higher rank abelian actions}\label{sec: HR act}

A detailed ground treatment of  Anosov and partially hyperbolic abelian higher rank actions, including a variety of techniques and examples,  can be found in \cite{KNbook}. For a detailed discussion of smooth ergodic theory of general abelian actions, see \cite{BRHW}.  

An action  $\al:\ZZ^k\to \Diff(M)$ on a closed manifold $M$ is   \emph{partially hyperbolic} if it contains a partially hyperbolic diffeomorphism  $\alpha(a)$, for some $a\in \ZZ^k$, and  \emph{Anosov} if it contains an Anosov diffeomorphism.  Some basic questions and difficulties related to partially hyperbolic actions are described in \cite{DK3},  \cite{DK4}. For  background on  partially hyperbolic abelian actions  with compact center foliation we refer to \cite{DX0} and the references therein.

Oseledec's theorem for a cocycle over an ergodic  transformation has a version for abelian actions  
\cite[Theorem 2.4]{BRHW}. Let $E\to M$ be a continuous vector bundle, and let $\mathcal A$ be a linear $\ZZ^k$-cocycle on $E$  over an ergodic, $\mu$-preserving action $\al$ of $M$, i.e.   ${\mathcal A}\colon \ZZ^k \to \mathrm{Aut}(E)$ is a $\ZZ^k$-action by bundle isomorphisms projecting to the action of $\alpha$ on $M$.
The higher-rank Oseledec theorem implies the existence of finitely many linear functionals $\chi\colon \RR^k\to \RR$, called (\emph{Lyapunov functionals  for $\mathcal A$}), and an $\mathcal A$- invariant measurable splitting $\oplus E_\chi$  of $E$, called the (\emph{Oseledec decomposition for $\mathcal A$}), on a full $\mu$-measure set,  such that for $a\in\ZZ^k$
and $v\in E_\chi(x)$: \[\lim_{a\to \infty}\frac{\log\|\mathcal A(a, x)(v)\|-\chi(a)}{\|a\|}=0.\]

The hyperplanes $\ker_\chi\subset \RR^k$ are called
 \emph{Weyl chamber walls}, and the connected components of $\RR^k-\cup_\chi \ker_\chi$ are the \emph{Weyl chambers} for ${\mathcal A}$ (with respect to $\mu$). Even though elements of the Weyl chambers are vectors in $\RR^k$, we will often say that the diffeomorphism $\al(a)$ is in the Weyl chamber $\mathcal C$ if $a\in   \mathcal C$. 

 Two nonzero Lyapunov functionals $\chi_i$ and $\chi_j$ are \emph{coarsely equivalent} if they are positively proportional: there exists $c>0$ such that $\chi_i=c\cdot \chi_j$. This is an equivalence relation on the set of Lyapunov functionals, and a \emph{coarse Lyapunov functional} is an equivalence class under this relation. Given a fixed ordering of non-zero coarse Lyapunov  functionals $(\chi_1, \dots, \chi_r)$, each Weyl chamber $\mathcal C$ can be labelled by its \emph{signature}:   $(\mathrm{sgn} \chi_1 (a), \dots, \mathrm{sgn}\chi_r(a))$,   where $a$ is any element in $\mathcal C$. The Weyl chambers of $\mathcal A$ in $\RR^k$ together with their assigned signatures we call \emph{the Weyl chamber picture of $\mathcal A$ over $\alpha$}.
Two $\ZZ^k$ cocycles (over possibly two distinct  $\ZZ^k$ actions), have {\it the same Weyl chamber picture} if the  
walls  in $\mathbb R^k$ coincide and the signatures of each Weyl chamber
coincide.  If for two Lyapunov functionals $\chi^1, \chi^2$, we have $\ker\chi^1=\ker\chi^2$ and $\chi^1(a)\chi^2(a)>0$ for some $a$, then $\chi^1,\chi^2$ are positively proportional. This implies the following: 
    
\begin{lemma}\label{lemma: crtrn same weyl}Suppose that the Lyapunov functionals $\{\chi^i \}, \{\chi'^i\}$ of two ergodic cocycles ${\mathcal A}$ and ${\mathcal A}'$  have the same Weyl chambers, and  suppose that for any $i$, there is an element $a\in \ZZ^k$ such that $\chi^i(a)\chi'^i(a)>0$.   Then ${\mathcal A}$ and ${\mathcal A}'$ have the same the Weyl chamber picture.
\end{lemma}

For Anosov actions, the higher-rank Oseledec theorem is applied to the derivative cocycle $D\al$, and 
 the \emph{Weyl chamber picture} depends only on $\al$ and on the invariant measure. In the presence of sufficiently many Anosov elements of the action (for example, one Anosov element in each Weyl chamber), and an ergodic measure of full support,  even the dependence on the measure can be removed. Moreover, in this case 
 the coarse Lyapunov distributions are intersections of stable distributions for finitely many elements of the action, they are well defined everywhere
 and tangent to foliations with smooth leaves.  (For more details see Section 2.2 in \cite{KSp}) 
 The same holds for actions that have many elements normally hyperbolic to a common center foliation \cite{DK4}.

Suppose $\al:\ZZ^k\to \Diff_{\mathrm{vol}}^2(M)$ is a conservative ergodic partially hyperbolic action on a compact manifold $M$ and let $\al(a)$ be a partially hyperbolic element.
By the discussion in Section \ref{commuting}  the sum $E^H:=E^u_a\oplus E^s_a$ of the stable and unstable  distributions of $\al(a)$ is $\alpha-$invariant.
We will apply the higher-rank Oseledec theorem
to the cocycle $D\al\vert_{E^H}$ and to stress this restriction of the derivative cocycle to the smaller bundle,  we call the corresponding  picture the 
  \emph{hyperbolic} Weyl chamber picture for $\alpha$.
  
 
 

 An action    $\al$ is \emph{maximal} if there are exactly $k+1$ coarse Lyapunov   functionals    corresponding to  $k+1$ distinct Lyapunov    hyperspaces,    and if the Lyapunov hyperspaces are in \emph{general position}, i.e.  if no Lyapunov hyperspace contains a non-trivial intersection of two other Lyapunov hyperspaces. 
 Maximality implies a special property of Weyl chambers: there is any combination of signs of Lyapunov functionals among the Weyl chambers, except all positive, and all negative. Prime  examples of maximal  Anosov actions are actions by toral automorphisms. Namely
\begin{lemma}\label{lemma: za max}[\cite{KKS}] Suppose $A\in \SL(k,\ZZ)$ is a hyperbolic irreducible matrix. Then $\Z_{\SL(k,\ZZ)}(A)$ induces a maximal abelian Anosov action on $\TT^k$  if $\ell_0(A)>1$.
\end{lemma}





Results of Franks and Manning \cite{Fr, Ma} imply that every Anosov action $\al\colon \ZZ^k\to \Diff(\TT^d) $ is topologically conjugate to an action $\kappa\colon \ZZ^k\to \mathrm{Aff}(\TT^d)$ by
affine automorphisms of the torus. Such an action $\kappa$ is called a \emph{linearization} of $\al$. The {\em linear part} of  $\kappa$ is the action $\kappa_0\colon  \ZZ^k\to \mathrm{Aut}(\TT^d)$ that sends
$g$ to $T_{A_g}$, where $\kappa(g) = T_{A_g} + v(g)$.  The linear part does not depend on the choice of linearization of $\alpha$.

An affine $\ZZ^{k}$-action $\kappa'$ on $\TT^{d'}$, is called an (algebraic) factor of an affine $\ZZ^k$-action $\kappa$ on $\TT^d$ if there exists a surjective homomorphism $\varphi : \TT^d\to \TT^{d'}$ such that $ \varphi\circ \kappa = \kappa'\circ \varphi$.  An affine action $\kappa$ is said to have a {\em rank one factor} if its linear part $\kappa_0$ has a nontrivial  factor $\kappa'\colon \ZZ^{k} \to \mathrm{Aut}(\TT^{d'})$ such that the image $\kappa'(\ZZ^{k})$ is virtually cyclic. A smooth $\ZZ^k-$action on $\TT^d$ is \emph{higher rank} if  its linearizations have no rank one factor. In particular, when one element of a linear action is an irreducible toral automorphism, the action is called \emph{irreducible} and we have the following easy lemma:  

\begin{lemma}\label{lemma: lnr high ran}[ \cite{KKS}, Section 3.1.] 
Suppose $A,B\in \GL(n,\ZZ)$ satisfies $AB=BA$. Assume that $A$ is irreducible and the group generated by $A$ and $B$ is not virtually $\ZZ$. Then the action generated by $<T_A, T_B>$ on $\TT^n$ is a higher rank action. 
\end{lemma}
 One important feature of higher rank Anosov actions is cocycle rigidity, which has the following  application to isometric extensions: 
\begin{lemma}\label{lemma: R-val cocy rig tori}[
\cite{KS}, Theorem 2.9] Suppose $A,B\in \GL(n,\ZZ)$ commute and generate a higher rank Anosov action $<T_A, T_B>$ on $\TT^n$. Let $\rho_A, \rho_B$ be H\"older functions on $\TT^{n}$. Then the isometric extensions $(T_A)_{\rho_A}, (T_B)_{\rho_B}$ commute iff there exist a H\"older function $\beta$ on $\TT^n$ and $\theta_A, \theta_B\in \RR$ such that $\rho_A=-\beta\circ T_A+\beta+\theta_A$, and $\rho_B=-\beta\circ T_B+\beta+\theta_B$. 
\end{lemma}
We state the global rigidity result  \cite{HW} and its corollaries concerning centralizers.
\begin{maintheorem}\label{main: gl rig anosov}[\cite{HW}] Let $\al:\ZZ^k\to  \Diff^\infty(\TT^d)$ be an Anosov action, and let $\kappa$ be a linearization of $\al$.    If  $\kappa$ is higher rank, then $\al$ is $C^\infty$ conjugate to $\kappa$. 
\end{maintheorem}
As a corollary we have following result about centralizers:
\begin{coro}\label{coro: gl rig toral}Let $A\in \SL(d-1,\ZZ)$ and let $r_0$ be as in Theorem \ref{main: dich}.
Fix $r> r_0$.  Suppose $g\in \Diff^\infty(\TT^{d-1})$ is a $C^1-$small perturbation of $T_{A}$ (or, more generally, has narrow band spectrum). Then either $g$ is $C^\infty$ conjugate to $T_{A}$ or $\Z_{\Diff^s(\TT^d)}(g)$ is virtually trivial for any $s\geq r$. 
\end{coro}
\begin{proof}  Clearly $T_{A}$ has narrow band spectrum. Fix $r'\in (r_0, r)$;  Lemma \ref{lem: NBS Stable} implies that any $g$ sufficiently close to $T_{A}$ has narrow band spectrum, and $r(g) <r'$. Corollary \ref{coro: gl rig toral} then follows from the lemma that follows.\end{proof}

\begin{lemma}\label{lemma: coro HRW} Let $g\colon\TT^n\to \TT^n$ be a $C^\infty$ Anosov diffeomorphism with narrow band spectrum, let $\kappa(g)$ be a linearization of $g$, and let $\kappa_0(g)\in \mathrm{Aut}(\TT^n)$ be its linear part. If $\kappa_0(g)$ is irreducible, then either $g$ is $C^\infty$ conjugate to $\kappa(g)$ (equivalently, to $\kappa_0(g)$) or $\Z_{s}(g)$ is virtually trivial for any $s>r(g)$.
\end{lemma}
\begin{proof} The narrow band spectrum assumption and Theorem \ref{thm: normal form} imply that $g$ preserves some $C^\infty$ normal forms on $\W^{\ast}_g, \ast=s,u$, which are also preserved by  any $h\in \Z_{s}(g)$ for $s>r(g)$.  Since $h$ is smooth along the tranverse foliations $\W^s_g$ and $\W^u_g$, Journ\'e's  lemma implies that $h$ is smooth.  Thus $\Z_s(g)=\Z_\infty(g)\subset \Z_{\mathrm{Homeo}(\TT^n)}(g)$, which has a finite index subgroup $G\cong\ZZ^\ell$, by irreducibility of $\kappa_0(g)$, and Lemmas \ref{lemma: cent linear Ansv} and \ref{lemma: rank cent}.

Suppose that $g$ is not $C^\infty$ conjugate to $\kappa(g)$. Applying Theorem \ref{main: gl rig anosov} to the action of $G$ gives a rank one factor for a linearization of $G$. By the irreducibility of $\kappa_0(g)$ and Lemma \ref{lemma: lnr high ran}, the rank of $G$ must be $1$. Therefore  $\Z_s(g)=\Z_\infty(g)$ is virtually trivial. \end{proof} 

\section{Proofs of Theorems \ref{main: cent dic geod fl} and \ref{main: geod fl}}\label{section: geodesic proofs}

We begin with a general discussion of perturbations of discretized geodesic flows in negative curvature.
Let $X$ be a closed, negatively curved Riemannian manifold of any dimension, and let $\psi_t\colon T^1X\to T^1X$ be the geodesic flow on the unit tangent bundle $T^1X$.  

The centralizer of the flow  $\psi_t$ (and hence any element of the flow) contains the flow itself.   If $X$ admits an isometry $h$, then the derivative $Dh$ preserves the unit tangent bundle $T^1X$ and commutes with the flow.  While the flow fixes its own orbits, the derivative of a nontrivial isometry permutes the orbits nontrivially. 

Suppose $g\colon T^1X\to T^1X$ is an arbitrary continuous map, and let $g_\ast \colon \pi_1(T^1X)\to \pi_1(T^1X)$ be the induced map on the fundamental group.  We claim that $g$ induces a homomorphism $\bar g_\ast \colon \pi_1(X) \to \pi_1(X)$ such that $\bar g_\ast p_\ast = p_\ast g_\ast$, where $p\colon T^1X\to X$ is the canonical projection.  When $\dim(X)\geq 3$, this is immediate, because the fibers of $T^1X$ are simply connected.  When $X$ is a surface, this follows from the fact that $\pi_1(T^1X)$ is a central extension of the simple group $\pi_1(X)$.
 
Note that since $\psi_t$ is isotopic to the identity, it induces a trivial map  on $\pi_1(X)$, whereas the derivative of a nontrivial isometry $h$ induces a nontrivial automorphism $\overline{Dh}_\ast$ of $\pi_1(X)$, namely $h_\ast$ itself.  The latter automorphism $h_\ast$ induces a nontrivial {\em outer} automorphism; that is, it is not induced by a conjugacy on $\pi_1(X)$.  This is because, as we shall see, homeomorphisms of $T^1X$ that leave invariant the orbit foliation of    $\psi_t$    and that induce inner automorphisms of $\pi_1(X)$ must fix the leaves of the orbit foliation.

\begin{prop}\label{prop:innerauto}  Let $X$ be a closed, negatively curved manifold, and suppose that $g\colon T^1X\to T^1X$ is a homeomorphism that leaves invariant the orbit foliation of the geodesic flow $\psi_t$.  The following are equivalent:
\begin{enumerate}
\item there exists  $\hat\gamma\in \pi_1(X)$ such that $\bar g_\ast(\gamma) = \hat\gamma \gamma \hat\gamma^{-1},$ for every $\gamma\in \pi_1(X)$.
\item  $g$ leaves invariant each orbit of $\psi_t$.
\end{enumerate}
\end{prop}
\begin{proof}  (1): Since $g$ preserves the orbits of the geodesic flow, the map $\bar g_\ast$ has a simple discription: given $\gamma\in \pi_1(X)$, represent $\gamma$ by a closed, unit-speed geodesic $c_\gamma$ in $X$ (here we are using free homotopy equivalence): this representation is unique up to reparametrization, because $X$ is negatively curved.  The lift $c_\gamma'$ to $T^1X$ is a closed orbit of $\varphi_t$ and is taken to a closed orbit $ \hat c'$ by $g$; the projection of this orbit to $X$ is a closed geodesic $\hat c = c_{\bar g_\ast(\gamma)}$ representing the class $\bar g_\ast(\gamma)$.

Now suppose that  there exists $\hat\gamma\in \pi_1(X)$ such that for every $\gamma\in \pi_1(X)$,
$\bar g_\ast(\gamma) = \hat\gamma \gamma \hat\gamma^{-1}$.  The group $\Gamma = \pi_1(X)$ acts freely on the universal cover $\widetilde X$ on the left by isometries.  Since $X$ is closed and negatively curved, each $\gamma\in \Gamma$ has a unique axis $\alpha_\gamma$, which is a geodesic in $\widetilde X$, invariant under $\gamma$ and on which $\gamma$ acts by translations.

Denote by $\pi\colon \widetilde X \to X$ the covering projection.   It is easy to see that 
\[\pi^{-1}(c_\gamma) = \bigsqcup_{\eta\in\Gamma} \eta \alpha_\gamma = \bigsqcup_{\eta\in\Gamma}  \alpha_{\eta\gamma\eta^{-1}}.
\]
Denote by $\tilde g$ the action of $g$ on lifted geodesics in $\widetilde X$, which is well-defined up to deck transformations.  Then
\[ \tilde{g}\left(\pi^{-1}(c_\gamma) \right)  = \pi^{-1}(c_{\bar g_\ast{\gamma}})  =    \bigsqcup_{\eta\in\Gamma} \alpha_{\eta\hat\gamma\gamma(\eta\hat\gamma)^{-1}} = \pi^{-1}(c_{\gamma}).\]
Thus $g(c_\gamma'(\RR))  = c_\gamma'(\RR)$, for every closed $\psi_t$-orbit $c_\gamma'(\RR)$.   Since $X$ is closed and negatively curved,  $\psi_t$-periodic orbits are dense in $T^1X$, and so $g$ fixes all $\psi_t$-orbits.

(2) If $g$ fixes all $\psi_t$-orbits, then by the argument for (1),  we obtain that $\bar g_\ast$ preserves the conjugacy classes in $\pi_1(X)$ and thus must act by conjugation.
\end{proof}

Suppose that  $f\in \Diff^r(T^1X), r\geq 1$ is a $C^1$-small perturbation of $\psi_{t_0}$. By Theorem \ref{main: HPS}, $f$ is  dynamically coherent, and $(f,\W^c)$ is leaf conjugate to $(\psi_{t_0}, \W^c_{\psi_{t_0}})$.  Proposition \ref{prop: gwc=wc} implies that for any $g\in \Z_{1}(f)$, $g(\W^*)=\W^*$, for $\ast\in\{u,c,s,cu,cs\}$.

Let $\Z^+_{r}(f)$ be the subgroup of $\Z_{r}(f)$ consisting of the elements that preserve the orientation of $\W^c$. Clearly $\Z^+_{r}(f)$ has finite index in $\Z_{r}(f)$. We denote by $\Z_{r}^c(f)$ the set of $g\in \Z^+_{r}(f)$ fixing the leaves of $\W^c(f)$.  Observe that $\Z^c_{r}(f)$ is a normal subgroup of $\Z^+_{r}(f)$. 

\begin{prop}\label{prop: outerauto}  Let $\psi_{t_0}$ be the discretized geodesic flow over a closed, negatively curved manifold $X$.  There exists $\epsilon>0$ such that for any $r\geq 1$,  if $f\in \Diff^r(T^1X)$, and $d_{C^1}(f,\psi_{t_0})<\epsilon$, then 
$\Z^+_r(f)/\Z_r^c(f)$ is isomorphic to a subgroup of the outer automorphism group $\mathrm{Out}( \pi_1(X))$.
\end{prop}
 
\begin{proof} Consider the map that sends $g\in \Z^+_r(f)$ to $[\bar g_\ast] \in \mathrm{Out}(\pi_1(X))$. It suffices to prove that the kernel of this map is $Z_r^c(f)$.    Suppose then that $g$ lies in the kernel, i.e. that there exists $\hat\gamma\in \pi_1(X)$ such that $\bar g_\ast (\gamma) = \hat\gamma\gamma\hat\gamma^{-1}$, for all $\gamma\in \pi_1(X)$.

Let $h\colon T^1X\to T^1X$ be the leaf conjugacy between   $(\W^c_{f}, \psi_{f})$ and  $(\W^c_{\psi_{t_0}}, \psi_{t_0})$, satisfying
\[h\left( \W^c_{f} (v) \right) = \W^c_{\psi_{t_0}}(h(v)),\]
for all $v\in T^1X$, and let $g_1 = h\circ g \circ h^{-1}$, which is a homeomorphism preserving the orbit foliation of $\psi_t$.  Since $h$ is homotopic to the identity, the induced maps $\bar g_\ast$ and $\bar {g_1}_\ast$ are the same (i.e., conjugacy by $\hat\gamma$).   Proposition~\ref{prop:innerauto} implies that $g_1$ fixes the $\psi_t$ orbits, and so $g$ fixes the leaves of $\W^c$, i.e. $g\in \Z_r^c(f)$.  Similarly, if $g\in \Z_r^c(f)$, then $g$ lies in the kernel.
\end{proof}

\begin{prop}\label{prop: outer auto}
Let $X$ be a closed, negatively curved manifold.  There exists $\epsilon>0$ such that for any $r\geq 1$,  if $f\in \Diff^r(T^1X)$, and $d_{C^1}(f,\psi_{t_0})<\epsilon$, then 
the quotient $\Z^+_r(f)/\Z_r^c(f)$  is finite.
\end{prop}

\begin{proof} 

The argument splits into two cases according to the dimension of $X$.  In the first case, $\dim(X)\geq 3$, the outer automorphism group of $\pi_1(X)$ is finite, which immediately gives the conclusion.  In the second case, $\dim(X)= 2$,  the outer automorphism group is infinite, isomorphic to  the extended mapping class group $\hbox{Mod}^{\pm}(X)$, which contains the mapping class group  $\hbox{Mod}(X)$  as an index $2$ subgroup.   A further analysis of the dynamics of centralizer is required.

\medskip

\noindent
{\bf The case $\dim(X)\geq 3$.} Work of  Paulin  and Sela  \cite{Paulin, ZS} shows that if $X$ is closed and negatively curved, of dimension at least $3$, then $\mathrm{Out}( \pi_1(X, p(v)))$ is finite: the fundamental group of $X$ is a torsion-free hyperbolic group that does not admit an essential small action on a real tree (see \cite[Corollary 0.2]{ZS} and the discussion that follows).
Thus Proposition~\ref{prop: outer auto} follows immediately from Proposition~\ref{prop: outerauto}.

\medskip

\noindent
{\bf The case $\dim(X)= 2$.} If $X$ is a closed, negatively curved surface, then $\mathrm{Out}(\pi_1(X, p(v)))$ is isomorphic to the extended mapping class group, which since $X$ is a surface, is the group of diffeomorphisms of $X$ modulo homotopy equivalence. The following lemmas are well-known; we sketch their proofs for completeness. 

\begin{lemma}\label{lem: Mod Finite Order}  Let $X$ be a closed, negatively curved surface.  Suppose that $h\in \mathrm{Mod}^\pm(X) \cong \mathrm{Out}(\pi_1(X, p(v)))$ has the property that for every conjugacy class $[\gamma]$ of $\gamma\in \pi_1(X, p(v))$, there exists $k\geq 1$ such that
\[h^k[\gamma] = [\gamma].\]
Then $h$ has finite order.
\end{lemma}

\begin{proof}  Represent $h$ by a diffeomorphism $\hat h\colon X\to X$, and take a system of filling curves $\gamma_1, \ldots, \gamma_n$ in $X$.  (These are closed curves with minimal intersection that separate $X$ into a union of disks).  Then some power of $\hat h$ fixes these curves (up to homotopy).  Iterating further, some power ${\hat h}^L$ leaves invariant the disks bounded by the curves (up to homotopy).  But then by coning off ${\hat h}^L$ in each disk, we get that ${\hat h}^L$ is homotopic to the identity in each disk, and so ${\hat h}^L$ is homotopic to the identity.  Thus $h^L$ is trivial.
\end{proof}

\begin{lemma}\label{lem: Mod Burnside} Let $X$ be a closed, negatively curved surface, and let $G$ be a subgroup of $\hbox{Mod}^{\pm}(X)$ with the property that every $h\in G$ has finite order.  Then $G$ is finite.
\end{lemma}

We remark that there is no assumption that $G$ be finitely generated in Lemma~\ref{lem: Mod Burnside}.
\begin{proof}[Proof of Lemma~\ref{lem: Mod Burnside}]  Since $\hbox{Mod}(X)$ has index $2$ in $\hbox{Mod}^{\pm}(X)$, it suffices to prove the statement for $G < \hbox{Mod}(X)$. 
The Torelli group $\mathrm{Tor}(X)$ is the set of $g\in \mathrm{Mod}(X)$ that induce a trivial action on first homology $H^1(X,\ZZ)$.  We have the short exact sequence 
\[1 \rightarrow \operatorname{Tor}(X) \rightarrow \operatorname{Mod}(X) \rightarrow \operatorname{Sp}\left(H^{1}(X,\ZZ)\right) \cong \operatorname{Sp}_{2 g}(\ZZ) \rightarrow 1,\]
where $g$ is the genus of $X$, and $ \operatorname{Sp}_{2 g}(\mathbb{Z})$ is the integer symplectic group.

It is well-known that $\mathrm{Tor}(X)$ is torsion-free.  Thus if $G<\mathrm{Mod}(X)$ is a torsion group, it is isomorphic to a subgroup of $\operatorname{Sp}_{2 g}(\ZZ)$.  But $ \operatorname{Sp}_{2 g}(\ZZ)$ is arithmetic and thus contains a finite index torsion free normal subgroup $H$ (for example, $H=\Gamma(3) = \{A\in \operatorname{Sp}_{2 g}(\ZZ): A\equiv I\mod 3\} $).  But this implies that $G$ injects into  $ \operatorname{Sp}_{2 g}(\ZZ)/H$,  which is finite. Hence $G$ is finite.
\end{proof}

We return to the proof of Proposition~\ref{prop: outer auto} in the case $\hbox{dim}(X)=2$.  Suppose that $g\in \Diff^1(T^1 X)$ commutes with $f$, a perturbation of the discretized geodesic flow $\psi_{t_0}$.  Lemma~\ref{lem: g fixes periodic} implies that every closed leaf of $\W^c_f$ is periodic under $g$.  Thus $h=[\bar g_\ast] \in \mathrm{Out}(\pi_1(X))$ satisfies the hypotheses of Lemma~\ref{lem: Mod Finite Order} and hence has finite order.  The image of  the quotient $\Z^+_r(f)/\Z_r^c(f)$ in $\mathrm{Out}(\pi_1(X))$ is thus a torsion group, and so by Lemma~\ref{lem: Mod Burnside} is finite. \end{proof}

We remark that Proposition~\ref{prop: outerauto} and the discussion above also imply that for $X$ negatively curved and locally symmetric, of dimension at least $3$,   \[\Z^+_r(\psi_{t_0})/\Z_r^c(\psi_{t_0}) \cong \mathrm{Out}( \pi_1(X))/<\pm \id>,\]  since by Mostow rigidity, every outer automorphism is represented by a unique isometry. With a little more work (see, e.g., \cite{JH}), one can show that for any $t_0\neq 0$ the centralizer of $\psi_{t_0}$  in $\Diff^1(T^1X)$ is precisely the group generated by the flow itself and the isometry group of $X$.  The same holds for hyperbolic surfaces. Details are left to the reader.

\begin{proof}[Proof of Theorem  \ref{main: geod fl}]Let $f$ be a diffeomorphism satisfying all the hypotheses of Theorem \ref{main: geod fl}.  By  \cite{KK} and \cite{RR}, $\psi_{t_0}$ in Theorem \ref{main: geod fl} is stably accessible and hence stably ergodic (by, e.g. \cite{BW}), and so we may assume that $f$ is accessible and ergodic.    Lemma~\ref{lemma: g pr vol} then  implies that  $\Z_1(f)\subset \Diff_{\mathrm{vol}}(M)$.
Proposition~\ref{prop: outer auto}  implies that $\Z^+_1(f)$, and hence $\Z_1(f)$, is virtually $\Z^c_1(f)$.

Assume that the disintegration of $\mathrm{vol}$ along $\W^c_f$ leaves is not Lebesgue; we show that $\Z^c_1(f)$ is virtually $<f^n>$, which will complete the proof of Theorem \ref{main: geod fl}. First, since $(f,\W^c_f)$ is leaf  conjugate to $\psi_{t_0}$, all but countably many $\W^c_f-$leaves are noncompact. For any noncompact $\W^c_f-$leaf, we consider the total order $``<"$ induced by the canonical orientation on $\W^c_f$. The action of $f$ on every non-compact $\W^c_f-$leaf is uniformly close to a translation by $t_0$ on $\RR$, and therefore is topologically conjugate to a translation.

Theorem F in \cite{AVW2} implies that the disintegration of $\mathrm{vol}$ along $\W^c_f$ leaves is atomic: there is a full volume set $S\subset T^1X$ and $k\in \ZZ^+$ such that for almost every $v\in T^1X$, $\W^c_f(v)$ is non-compact, 
\begin{equation}\label{eqn: S cap Wc}
S\cap \W^c_f(v)=\{x_{i,j}(v), i\in \ZZ, 1\leq j\leq k\},
\end{equation}
and 
\begin{equation}\label{eqn: x ij v def}
f^i(v)\leq x_{i,1}(v)<x_{i,2}(v)<\cdots <x_{i,k}(v)<f^{i+1}(v),~~f(x_{i,j}(v))=x_{i+1,j}(v).
\end{equation}

Fix an arbitrary $g\in \Z^c_1(f)$.   Lemma \ref{lemma: g pr vol} implies that $g$ is volume preserving,  which implies that, modulo a zero set, $gS=S$. As a consequence, there is an $f-$invariant full volume set $\Omega\subset T^1X$ such that for any $v\in \Omega$,
\begin{itemize}
\item $\W^c_f(v)$ is noncompact;
\item$S$ meets $\W^c_f(v)$ in exactly $k$ orbits and \eqref{eqn: S cap Wc}, \eqref{eqn: x ij v def} hold, i.e. we can define $x_{i,j}(v)$ associated to $v$;
\item $g(S\cap \W^c_f(v))=S\cap \W^c_f(v)$; and
\item $f(S\cap \W^c_f(v))=S\cap \W^c_f(v)$.
\end{itemize}

Since $g$ preserves the orientation on $\W^c_f-$leaves,  for any $v\in \Omega$, the restriction of $g$ to $\W^c_f(v)\cap S(=\{x_{i,j}(v), i\in \ZZ, 1\leq j\leq k\})$ is an order preserving transformation. By \eqref{eqn: x ij v def}, for any $v\in \Omega$, both $g|_{\W^c_f(v)\cap S}, f|_{\W^c_f(v)\cap S}$ are conjugate to a translation on $\ZZ$. 

In particular, for any $v\in \Omega$, there exists $k'(g,v)\in \ZZ$ such that on $W^c_f(v)\cap S$, we have $g^k=f^{k'(g,v)}$. Moreover by the construction of $x_{i,j}$, the fact that  $fg =gf$ implies $k'(g,v)$ is an $f-$invariant function on $v$.  Ergodicity of $f$ then implies that $k'(g,v)$ is almost everywhere a constant $k'(g)$, and on a full measure subset of $S$, $g^k=f^{k'(g)}$. But any full measure subset of $S$ is dense in $T^1X$, and hence $g^k=f^{k'(g)}$ on all of $T^1X$. In addition, any $g_1,g_2\in \Z^c_1(f)$ satisfying $k'(g_1)=k'(g_2)$  must induce the same transformation on $S\cap \W^c_f(v)$ for almost every $v\in T^1X$, which implies that $g_1=g_2$. Therefore $k'$ induces a group embedding $$k':\Z^c_1(f)\to \ZZ,$$  and $k'(<f^n>)=k\ZZ$. Then $\Z^c_1(f)$ is virtually $<f^n>$, proving Theorem~\ref{main: geod fl}\end{proof}

\begin{proof}[Proof of Theorem \ref{main: cent dic geod fl}]

Returning to the proof of Theorem \ref{main: cent dic geod fl},
if the volume has singular disintegration along $\W^c_f$, then Theorem \ref{main: cent dic geod fl} is just a corollary of Theorem \ref{main: geod fl}.

Suppose now the volume has Lebesgue disintegration along $\W^c_f$. Theorem F in \cite{AVW2} implies that there is a continuous vector field $Y$   tangent to $\W^c_f$ such that the continuous flow (a priori it might not be smooth) $\varphi_t$ generated by $Y$ satisfies the following:
\begin{itemize}
\item $\varphi_1=f$, and
\item $Y$, and hence $\varphi_t$, is uniformly smooth along the leaves of $\W^c_f$.
\end{itemize}

By assumption, $\psi_{t_0}$ has either $2$-bunched or narrow band spectrum. Let $r_0 = r(\psi_{t_0})\geq 1$.
Fix $r>r_0 $; we may assume, by Lemma~\ref{lem: NBS Stable}, that $f$ has either $2$-bunched or narrow band spectrum, and $r(f) <r$.  Consider $h\in \Z^c_r(f)$.

 By ergodicity of $f$, $h$ preserves the disintegration of  volume along $\W^c_f$. Therefore $h=\varphi_t$ for some $t\in \RR$.
If follows that 
\begin{equation}\label{eqn: Zc phi t}
\Z^c_r(f) =  \{\varphi_t, t\in D\}, \text{ where }D:=\{t\in \RR: \varphi_t\in \Diff^r(T^1X)\}.
\end{equation}

Since $f = \varphi_1$ is $C^\infty$, it follows that $D$ is a non-empty subgroup of $\RR$, and by Proposition~\ref{prop: outer auto}, $\Z_r(f)$  contains  $\{\varphi_t:  t\in D\}$ as a finite index subgroup.

\smallskip

\noindent\textbf{Case 1:} $D$ is discrete. Then, since $f=\varphi_1$, it follows that $<f>$ has finite index in $\{\varphi_t: t\in D\}$,  and hence in $\Z_r(f)$.  Thus $f$  has virtually trivial centralizer in $\Diff^r(T^1X)$.

\smallskip

\noindent\textbf{Case 2:} $D$ is dense in $\RR$.  
We use the normal form theory from Section~\ref{sec: normal form} to show that the $C^\infty$ smoothness of the $\varphi_t$ with $t\in D$ extends to all $t\in\RR$, as follows.
Applying  Proposition \ref{prop: applic normal form} to the triple $(f, \varphi_t, Y)$,  we obtain that $D=\RR$, $Y$ is a $C^\infty$ vector field and $\varphi_t$ is a $C^\infty$ flow.
As a consequence, by \eqref{eqn: Zc phi t} for any $s\geq r$ we have $$\Z^c_r(f)=\{\varphi_t : t\in \RR\}\subset\Z^c_s(f)\subset  \Z^c_r(f),$$ 
which implies $\Z^c_s(f)=\{\varphi_t : t\in \RR\}$. 
 Thus by Proposition~\ref{prop: outer auto} for any  $s\geq 1$,  $\Z^+_s(f)$ hence $\Z_s(f)=\Z_{\Diff^s_{\mathrm{vol}}(T^1X)}(f)$ is virtually $\{\varphi_t: t\in \RR\}\cong \RR$.
\end{proof}


\section{Proof of Theorem \ref{main: thm pr}}\label{sec: proof Cartan}

As mentioned in the introduction, the key idea in the proof of Theorem~\ref{main: thm pr} is to   show existence of many   partially hyperbolic elements commuting with $f$, an argument that we now detail.

\subsection{The groups $G$ and $G_0$}\label{sec: start ass} Two central players in the proof of 
Theorem~\ref{main: thm pr} are groups $G$ and $G_0$, which we define in this subsection.
We start with an easy observation. 

For $f_0$ as in Theorem \ref{main: thm pr}, we denote by $\lambda^1(f_0)>\cdots>\lambda^i(f_0)>\cdots$ the \emph{distinct} Lyapunov exponents of $f_0$ and the corresponding $Df_0-$invariant Lyapunov splitting by
\begin{equation}\label{eqn: dom spl non C f0}
T\TT^d=\oplus E^i_{f_0}\oplus E^c_{f_0}.
\end{equation}
Since $f$ is $C^1-$close to $f_0$, it follows that there is a corresponding $Df-$invariant dominated splitting 
\begin{equation}\label{eqn: dom spl non C}
T\TT^d=\oplus E^i_{f} \oplus E^c_{f}
\end{equation}
and $f$-invariant foliations $\W^s, \W^u, \W^{cs}, \W^{cu}$, and $\W^c$.
 
Consider an arbitrary element $g\in \Z_2(f)$.  Proposition~\ref{prop: gwc=wc}  implies that $g\W^c=\W^c$.  Thus $f,g$ induce homeomorphisms $\bar{f},\bar{g}$ on the topological manifold $\TT^d/\W^c$ such that $\bar{f}\bar{g}=\bar{g}\bar{f}$.  Moreover $\bar{f}$ is H\"older conjugate to the hyperbolic automorphism  $T_{A}$   on $\TT^{d-1}$. By Lemma \ref{lemma: cent linear Ansv},  $\bar{g}$ is conjugate to an affine map by the same conjugacy.  For  $g\in \Z_2(f)$, we denote the linear part of this affine map by $T_{A_g}$ ,  where $A_g\in\GL(d-1,\ZZ)$.  In particular we have $A_f=A$.

Let $\pi\colon \TT^d\to \TT^{d-1}$ be the fibration given by Proposition~\ref{lemma: dich atom rot}, which satisfies $\pi\circ f = T_{A_f}\circ \pi$. Then the center  leaf $\pi^{-1}(0)$ is invariant under $f$; denote it by $\W^c_f(x_0)$.  We use this leaf to define $G$ and $G_0$.

\begin{definition}\label{def: G G0} Let  $G_0$ be the group of all the elements $g\in\Z_2(f)$ such that $g$ fixes $\W^c_f(x_0)$ and preserves the orientation of $\W^c$ and $\TT^d/\W^c$.
Let $G<\SL(d-1,\ZZ)$ be the group generated by $\{A_g : g\in G_0\}$.
\end{definition}

The next proposition lays out the properties of $G$ and $G_0$ that we will use here.
\begin{prop}\label{lemma: prpty G0 G}Suppose $f, \ell_0$ satisfy the hypotheses of Theorem \ref{main: thm pr}.  Then
\begin{enumerate}
\item $\Z_2(f)$ is virtually $G_0$.  
\item $G_0, G$ are abelian groups. If the disintegration of $\mathrm{vol}$ along $\W^c$ is not Lebesgue, then $G_0$ is finitely generated.
\item One or both  of the following cases holds:
\begin{itemize}\item[I.] $G$ is virtually $\ZZ^\ell$ for some $
\ell\leq \ell_0$, where $\ell<\ell_0$ if $\ell_0>1$.
\item[II.] $G$ is a finite index subgroup of $\Z_{\SL(d-1,\ZZ)}(A_f)$. In particular, $G$ induces a maximal Anosov action on $\TT^{d-1}$ if $\ell_0>1$.
\end{itemize}
\end{enumerate}
\end{prop}
\begin{proof}[Proof of Proposition~\ref{lemma: prpty G0 G}]
(1)  
Let $\Z^+$ be the group of all the elements $g\in\Z_2(f)$ such that $g$  preserves the orientation of $\W^c$ and $\TT^d/\W^c$. Clearly $\Z^+$ has finite index in
$\Z_2(f)$. Denote by $\Z^c$ the set of center-fixing elements of $\Z^+$.

Consider the map from $\Z^+$ to  $\Z_{\Ho^+(\TT^{d-1})}(T_{A_f})$, sending $g$ to the map induced by $g$.       The kernel is $\Z^c$, and so  $\Z^+/\Z^c$ is isomorphic to a subgroup of 
$\Z_{\Ho^+(\TT^{d-1})}(T_{A_f})$.  By Lemmas~\ref{lemma: cent linear Ansv} and \ref{lemma: rank cent}, the group $\Z_{\Ho^+(\TT^{d-1})}(T_{A_f})$ is virtually   $\ZZ^m$, for some $m$,
and hence $\Z^+/\Z^c$ is virtually $\ZZ^{m'}$, for some ${m'}$.  

 Note that since there are finitely many center leaves fixed by $f$, and each element of  $\Z^+$ permutes the fixed center leaves,  there exists $k\geq 1$ such that for every element $g\in \Z^+/\Z^c$, we have
$g^k\in G_0/\Z^c$.  Thus the   finitely generated, abelian   quotient
\[\frac{\Z^+/\Z^c}{G_0/\Z^c} \cong \Z^+/{G_0} \]  
has the property that every element has order at most $k$, and is therefore finite.
This proves that $G_0$ has finite index in $\Z^+$, as claimed.

(2) Since $A_f$ is irreducible, Lemma \ref{lemma: rank cent} implies that $\Z_{\SL(d-1,\ZZ)}(A_f)$ (hence $G$) is a finitely generated abelian group.

To study the group $G_0$, first we consider the group $\Z^c$ defined as in the proof of (1). For any $h\in \Z^c$,  the rotation number $\rho(h,x)\in \TT$ is well-defined for $h|_{\W^c(x)}$ for any $x\in \TT^d$. By commutativity, it is not hard to get $\rho(h,x)=\rho(h,f(x))$. Since the rotation number is a continuous funtion on diffeomorphisms,  the ergodicity of $f$ implies  $\rho(h)=\rho(h,x)$ is independent of $x$. Moreover we have
\begin{lemma}\label{lemma: Zc circle grp}The map $\rho:\Z^c\to \TT,~~h\mapsto \rho(h)$ is a group embedding. In particular for any $h\in \Z^c$, if there exists $x\in \TT^d$ such that $\rho(h,x)=0$ then $h=\id$.
\end{lemma}
 
\begin{proof}[Proof of Lemma~\ref{lemma: Zc circle grp}]  By Proposition~\ref{lemma: dich atom rot} we have three possibilities:

\smallskip

\noindent{\bf Case 1}: The volume $\mathrm{vol}_{\TT^d}$ has atomic disintegration along $\W^c$.  Lemma \ref{lemma:fix center} implies that $\Z^c < \mathcal G_{\hbox{\tiny{$\mathrm{fix}$}}}(\W^c_f)$ is an abelian group, and therefore $\rho:\Z^c\to \TT$ is a group homomorphism. Moreover for $h\in \Z^c$, if  $\rho(h)=0$, then by the proof of Lemma \ref{lemma:fix center}, $h$ fixes all the atoms, which are dense in $\TT^d$. Thus $h=\id$.

\smallskip

\noindent{\bf Case 2}: $f$ is topologically conjugate to 
$T_{A_f}\times R_\theta$ for some $\theta\notin \QQ/\ZZ$. Let $\zeta$ be the conjugacy, so that  $\zeta^{-1}\circ f\circ\zeta(x,y) =(T_{A_f}(x), y+\theta)$.  Fix $h\in \Z^c$.  Since $h$ is center-fixing, there exists a continuous function $R(x,y)=R(h,x,y)$ such that $\zeta^{-1}\circ h\circ \zeta(x,y)=(x,y+R(x,y))$.

Since $h$ commutes with $f$,  $R(x,y)$  is $T_{A_f}\times R_\theta$-invariant. Transitivity of $T_{A_f}\times R_\theta$ implies that $R(x,y)$ is a constant function. Therefore for any $h\in \Z^c$, we have $\zeta^{-1}\circ h\circ \zeta=\id\times R_{\rho(h)}$, which implies Lemma \ref{lemma: Zc circle grp}.

\smallskip

\noindent{\bf Case 3}: $f$ is accessible, and the disintegration of $\mathrm{vol}_{\TT^d}$ has a continuous density function on the leaves of $\W^c$. Then \cite[Theorem C]{AVW2} implies that $f$ is topologically conjugate to a rotation extension over $(T_{A_f})_r$, i.e. there exist a continuous function $r(x)=r(x,y)$ and a homeomorphism $\zeta:\TT^d\to \TT^d$ such that 
 $\zeta^{-1}\circ f\circ\zeta(x,y)=(T_{A_f}(x), y+r(x)).$

For any $h\in \Z^c$,  as in Case 2. we can assume that there exists a $(T_{A_f})_r-$invariant, continuous function $R(x,y)=R_h(x,y)$ such that $\zeta^{-1}\circ h\circ \zeta(x,y)=(x,y+R(x,y))$. Then by transitivity of $(T_{A_f})_r$ (which follows from transitivity of $f$),  we have $\zeta^{-1}\circ h\circ \zeta=\id\times R_{\rho(h)}$, for any $h\in \Z^c$, which implies Lemma \ref{lemma: Zc circle grp}.

\end{proof}

Returning to the proof of item (2) of Proposition~\ref{lemma: prpty G0 G}, we 
obtain from   Lemma \ref{lemma: Zc circle grp} that $\Z^c$ is an abelian group. Observe that the map $h\mapsto A_h$ is a  surjective homomorphism from $G_0$ to $G$, with kernel $\Z^c$, and therefore $G_0$ is a group extension of $G$ by $\Z^c$. By commutativity of $\Z^c$ and $G$, we have that $G_0$ is a solvable group and $[G_0,G_0]\subset \Z^c$.

Now we claim that $[G_0,G_0]$ is trivial, and so $G_0$ is abelian. Suppose there exists $h\in [G_0,G_0]\subset \Z^c$, $h\neq \id $.  Lemma \ref{lemma: Zc circle grp} implies that  $\left.h\right|_{\W^c_f(x_0)}$ has non-zero rotation number, where $\W^c_f(x_0)$ is the $G_0-$fixed center leaf we defined in Section \ref{sec: start ass}. On the other hand, $G_0|_{\W^c_f(x_0)}$ is a solvable, orientation-preserving action on a circle.  It is known (cf. \cite{Na}) that rotation number induces a group homomorphism from any solvable subgroup of $\mathrm{Homeo}^+(S^1)$ to $\TT^1$, and so the kernel contains $[G_0,G_0]$. Thus $h|_{\W^c_f(x_0)}$ must have rotation number $0$, which is a contradiction. 

To show that $G_0$ is finitely generated if the disintegration of $\mathrm{vol}_{\TT^d}$ along  $\W^c_f$ is not Lebesgue, we only need to show the following lemma, since $G_0$ is a group extension of $G$ by $\Z^c$ and $G$ is finitely generated.

\begin{lemma}\label{lemma: Zc fin not Leb}The group $\Z^c$ is finite if the disintegration of $\mathrm{vol}_{\TT^d}$ along  $\W^c_f$ is not Lebesgue.
\end{lemma}

\begin{proof}Proposition~\ref{lemma: dich atom rot} gives two possibilities.

If  conclusion  \ref{case:atomic} holds,  i.e. the volume $\mathrm{vol}_{\TT^d}$ has atomic disintegration along $\W^c$,  then the finiteness  follows directly from Lemma \ref{lemma:fix center}.

If \ref{case: rigid} holds, then  $f$ is conjugate to $T_{A_f}\times R_\theta$ for some $\theta\notin \QQ/\ZZ$. In this case $E^u_f$ and $E^s_f$ are jointly integrable. Let $\W^H$ be the compact foliation tangent to the distribution $E^u_f\oplus E^s_f$. 

Let $\zeta$ be the conjugacy satisfying $\zeta^{-1}\circ f\circ \zeta=T_{A_f}\times R_\theta$. Then as in the proof of Lemma \ref{lemma: Zc circle grp}, for any $h\in \Z^c$, we have $\zeta^{-1}\circ h\circ \zeta=\id\times R_{\rho(h)}$. Let \[D:=\{\rho\in \TT: \zeta\circ (\id \times R_\rho)\circ \zeta^{-1}\in \Z^c\}.\]

If $D$ is discrete, then $\Z^c$ is finite.  If $D$ is dense, we will prove that in this case $\mathrm{vol}_{\TT^d}$ has Lebesgue disintegration along $\W^c$, which contradicts our assumption above. By density of $D$, any measure on $\TT$ invariant under $\{R_\rho: \rho\in D\}$ is the Lebesgue measure $\mathrm{vol}_\TT$. Recall that $\mathrm{vol}_{\TT^d}$ is $\Z^c-$invariant, therefore $\mathrm{vol}_{\TT^d}$ has the form $\zeta_\ast(\nu\times \mathrm{vol}_{\TT})$,  where $\nu$ is some probability measure on $\TT^{d-1}$.

In particular, if we denote by $\mathrm{Pr}^c$ the projection from $\TT^d$ to $\W^c(x_0)$ along $\W^H$ and $\mathrm{Pr}^H$ the canonical projection from $\TT^d$ to $\TT^d/\W^c$, we have that any $\Z^c-$invariant measure $\mu$ is the product of  $\mathrm{Pr}^c_\ast(\mu)$ with  $\mathrm{Pr}^H_\ast (\mu)$. In particular, for almost every $x$, the conditional measure $m^c_x$ on $\W^c(x)$ of $\mathrm{vol}_{\TT^d}$ has the following form 
\[m^c_x =\mathrm{Pr}^c|_{\W^c(x)}^\ast(\mathrm{Pr}^c_\ast(\mathrm{vol}_{\TT^d}));\]
that is, $m^c_x$ is the pullback of $\mathrm{Pr}^c_\ast(\mathrm{vol}_{\TT^d})$ on $\W^c(x)$ by $\mathrm{Pr}^c|_{\W^c(x)}$.

By Lemma \ref{lemma: WH C2}, we have the key fact that $\W^H$ is a $C^1$ foliation. It follows that $\mathrm{Pr}^c$ is $C^1$, and so $\mathrm{Pr}^c|_{\W^c(x)}^\ast(\mathrm{Pr}^c_\ast(\mathrm{vol}_{\TT^d}))$ has continuous density function for any $x$. This implies that $\mathrm{vol}_{\TT^d}$ has Lebesgue disintegration along $\W^c$ leaves. 
\end{proof}
This completes the proof of item (2).
 Item (3) is a corollary of Lemma \ref{lemma: rank cent}: for more details, see \cite{RH07}.
The proof of Proposition~\ref{lemma: prpty G0 G} is complete.\end{proof}

Having defined the groups $G$ and $G_0$ and established their essential properties, we return to the proof of
 Theorem~\ref{main: thm pr}.  We are given $f$ sufficiently $C^1$ close to $f_0$ and aim to prove that  either the $\mathrm{vol}$ Lebesgue disintegration along $\W^c_f$, or $\Z_2(f)$ is virtually $\ZZ^\ell$ for some $\ell\leq \ell_0$, with  $\ell<\ell_0$ if $\ell_0>1$.

By Proposition~\ref{lemma: prpty G0 G}, item (3), there are two possibilities: 
\begin{itemize}\item[I.] $G$ is virtually $\ZZ^\ell$ for some $
\ell\leq \ell_0$, where $\ell<\ell_0$ if $\ell_0>1$.
\item[II.] $G$ is a finite index subgroup of $\Z_{\SL(d-1,\ZZ)}(A_f)$. In particular, $G$ induces a maximal Anosov action on $\TT^{d-1}$ if $\ell_0>1$.
\end{itemize}

Suppose that conclusion I. holds.  Lemma \ref{lemma: Zc fin not Leb}  implies that either $\mathrm{vol}_{\TT^d}$ has Lebesgue disintegration along $\W^c$; or $\mathrm{vol}_{\TT^d}$ has singular disintegration along $\W^c$ and $\Z^c$ is finite. In the former case, we are finished.
In the latter case,  item (2) of Proposition~\ref{lemma: prpty G0 G}  implies that $G_0$ is a finitely generated abelian group and also  a group extension of $G$ by $\Z^c$. As we are assuming that $G$ is virtually $\ZZ^{\ell}$ for some $\ell\leq \ell_0=\ell_0(A_f)$ ($\ell<\ell_0$ if $\ell_0>1$), tt is not hard to check that there is a subgroup $G_1$ of $G_0$ isomorphic to the torsion free part of $G$, which is $\ZZ^\ell$. Therefore by finiteness of $\Z^c$, $G_0$ is virtually $\ZZ^\ell$.  Thus Theorem~\ref{main: thm pr} follows from conclusion (1) of Proposition~\ref{lemma: prpty G0 G}.

Suppose on the other hand that conclusion II of Proposition~\ref{lemma: prpty G0 G}, item (3), holds.  The case $\ell_0=1$ is contained in conclusion I, so we may assume that $\ell_0>1$.  We have the following key proposition.

\begin{prop}\label{lemma: G high rank ac}Suppose $f$ is as in Theorem \ref{main: thm pr}, and $G_0, G$ are as in Definition \ref{def: G G0}. If $G$ induces a maximal Anosov action on $\TT^{d-1}$, then $\mathrm{vol}_{\TT^d}$ has Lebesgue disintegration along $\W^c_f$.
\end{prop}

Assuming this proposition, the proof of Theorem~\ref{main: thm pr} is complete.  The proof of Proposition~\ref{lemma: G high rank ac} is lengthy and occupies the next section.

\section{Proof of Proposition~\ref{lemma: G high rank ac}}
This section is devoted to the proof of  Proposition~\ref{lemma: G high rank ac}.
We continue to assume that  $f_0:\TT^d\to \TT^d$ and  $\ell_0$ are as in   Theorem \ref{main: thm pr},   and   
that $f\in \Diff^2_{\mathrm{vol}}(\TT^d)$ is a $C^1-$small, ergodic perturbation of $f_0$.  In  addition,
we assume the hypothesis of Proposition~\ref{lemma: G high rank ac}, that $G$ induces a maximal Anosov action on $\TT^{d-1}$.  Our goal is to show that  $\mathrm{vol}_{\TT^d}$ has Lebesgue disintegration along $\W^c_f$.

 Without loss of generality we may assume that $G$ and $G_0$ are  finitely generated abelian groups (otherwise Proposition \ref{lemma: G high rank ac} follows from Proposition~\ref{lemma: prpty G0 G}). Then  there is a subgroup $G_1$ of $G_0$ isomorphic to $G$, through the map $g\mapsto A_g$. Replacing $G_0$ with $G_1$, we may thus assume that $G_0$ is isomorphic to $G$ through the map $g\mapsto A_g$. Moreover we may assume $G, G_0$ are torsion free (otherwise we consider their free parts instead).

The following proposition is the key step in the proof of Proposition~\ref{lemma: G high rank ac}. Recall that by linearity of the action of $G$, we can define the Lyapunov functionals and associated (hyperbolic) Weyl chamber picture as in Section \ref{sec: HR act}, independently of the invariant measure. Consider the  $G_0-$invariant dominated splitting $\oplus_i E^i\oplus E^c$ given by \eqref{eqn: dom spl non C}, ordered in $i$ by decreasing values of the Lyapunov exponents.

\begin{prop}\label{prop: key pr}  Assume that $G$ induces a maximal Anosov action on $\TT^{d-1}$.  
For every $i$, we have the following.
\begin{enumerate}
\item The bundle $E^i_f$ is uniquely integrable, tangent to an absolutely continuous foliation $\W^i_f$ with $C^2$ leaves.  
\item The restriction of $\pi\colon \TT^d\to \TT^{d-1}$
to each leaf of  $\W^i_f$ is a bi-H\"older  homeomorphism onto the leaf of the affine foliation tangent to $E^i_{A_f}$, with exponent $\delta$, where $\delta\to 1$ as $d_{C^1}(f,f_0)\to 0$.
Consequently $\pi$ itself is $\delta$-H\"older continuous as well.
\item For any $h\in G_0$ such that $A_h$ is not in any Weyl chamber wall of the action of $G$, $Dh$ uniformly contracts or expands $E^i$.
\end{enumerate}
\end{prop}
  
The rest of Section \ref{sec: proof Cartan} is dedicated to the proofs of Propositions \ref{lemma: G high rank ac} and \ref{prop: key pr}. The plan of the proofs is as follows: in Section \ref{sec: same picture} we prove the fundamental property of $G_0$, namely that $G$ and $G_0$ share the same Weyl chamber picture. Then in Sections \ref{sec: same W C} 
and \ref{sec: key sec} we prove Proposition \ref{prop: key pr}. In Section  \ref{sec: integra and top rig} we derive from Proposition \ref{prop: key pr} an important corollary: the joint integrability of $E^s_f$ and $E^u_f$. In particular, this   implies that $f$ is H\"older conjugate to $T_{A_f}\times R_\theta$ for some $\theta\notin \QQ/\ZZ$; see Proposition~\ref{coro: hold rig}. 

In Section \ref{sec: e stat pres} we consider a partially hyperbolic generalization of a classical result of thermodynamic formalism in the Anosov setting. Combining a cocycle rigidity result (see Section \ref{sec: cocyc rig}) over a partially hyperbolic abelian action, using Proposition~\ref{coro: hold rig} we complete the proof of Proposition \ref{lemma: G high rank ac} in Section \ref{sec: unique mme}.

\subsection{$G, G_0$ have the same hyperbolic Weyl chamber picture}\label{sec: same picture}
\begin{lemma}\label{lemma: share same wall}Any $G_0-$invariant ergodic measure $\nu$ has the same hyperbolic Weyl chamber  picture  as $G$.
\end{lemma}
\begin{proof}[Proof of Lemma~\ref{lemma: share same wall}]  First we prove that the action of $(G_0, \nu)$ has the same Weyl chamber walls as $G$.  Proposition~\ref{prop: gwc=wc}  implies that the foliations $\W^{u}_f$  and $\W^{s}_f$ are $G_0-$invariant. Moreover $\pi(\W^{*}_f)=\W^{*}_{A_f}$, for $\ast\in\{u,s\}$. Therefore to analyze the hyperbolic Weyl chamber walls of $G_0$ we need only consider the action of $G_0$ on   $\W^u_f$, $\W^s_f$   separately.  We show this for $\W^u$; the proof for $\W^s$ is analogous.

 Recall that by Lemma \ref{lemma: h dim reg chart}, to prove Lemma \ref{lemma: share same wall}, we only need to establish the following claim: for any $A_h\in G$ that is not in any Weyl chamber wall, if $A_h$ has $d^{u}_-, d^{u}_+$--dimensional stable and unstable distributions respectively within $\W^{u}_{A_f}$, then $h$ has $d^{u}_-, d^{u}_+$--dimensional stable and unstable \emph{topological foliations}  (with exponential contracting or expanding speed) respectively within $\W^{u}_f$.

In fact if the claim holds, then Lemma \ref{lemma: h dim reg chart} implies that for typical $h\in G$, the map $h$ and the matrix $A_h$ have the same number of positive (resp. negative) Lyapunov exponents with respect to any $G_0-$invariant ergodic  measure $\nu$.  From this it follows that $G$ and $(G_0,\nu)$  have the same Weyl chamber walls.

Notice that for any $x\in M$,  the restriction $\pi: \W^{u(s)}_f(x)\to \W^{u(s)}_{A_f}(\pi(x))$ is a homeomorphism.
To finish the proof of the claim, we use the following classical bi-H\"older estimate on the projection $\pi$, which we will use repeatedly to lift hyperbolicity  of elements acting in $\TT^{d-1}$ to hyperbolicity in $\TT^d$.  
\begin{lemma}\label{lemma: Hold exp}There exist $C, \delta>0$ such that for any $x\in \TT^d$ and $y\in \W^{u(s)}_{f}(x,loc)$,
\begin{equation*}
d_{\TT^{d-1}}(\pi(x), \pi(y))\leq C\cdot d_{\TT^d}(x,y)^{\delta},\quad\hbox{ and }~~
d_{\TT^d}(x,y)\leq C\cdot d_{\TT^{d-1}}(\pi(x), \pi(y))^{\delta}.
\end{equation*}
\end{lemma}
Lemma~\ref{lemma: Hold exp} is a simple consequence of  Theorem~\ref{main: HPS}.
This bi-H\"olderness of $\pi$ implies that the hyperbolicity of $T_{A_h}|_{\W^{u}_{T_{A_f}}}$ lifts under $\pi$ to uniform hyperbolicity of $h\vert_{\W^u_f}$.  Consequently, $G$ and $(G_0,\nu)$ have the same Weyl chamber walls.

Next consider the Lyapunov functionals $\{\lambda^{u,i}_{G_0}(\cdot, \nu), i=1,\dots, \dim E^u_f\}$ associated to the action of $G_0$ on $\W^u_f$ with respect to an ergodic measure $\nu$, and the Lyapunov functionals $\{\lambda^{u,i}_{G}(\cdot), i=1,\dots, \dim E^u_{T_{A_f}}\}$ associated to the action of $G$ on $\W^u_{T_{A_f}}$. By our discussion above, without loss of generality we may assume that Weyl chamber wall $\ker \lambda^{u,i}_{G_0}(\cdot, \nu)$ coincides with that of $\lambda^{u,i}_{G}(\cdot)$. Moreover $$\lambda^{u,i}_{G_0}(f, \nu)>0,\,\hbox{ and}~~ \lambda^{u,i}_{G}(T_{A_f})>0.$$
then by Lemma \ref{lemma: crtrn same weyl} (identifying $G, G_0$ with $\ZZ^k$ in the obvious way), the Weyl chamber picture of the action of $G_0$ on $\W^u_f$ with respect to  $\nu$ is the same as that of $G$ on $\W^u_{T_{A_f}}$. The same argument applied to the action of $G_0$ on $\W^s_f$ gives that  the Weyl chamber picture of the action of $G_0$ on $\W^s_f$ with respect to  $\mu$ is the same as that of $G$ on $\W^s_{T_{A_f}}$. In conclusion, $(G_0,\nu)$ has the same hyperbolic Weyl chamber picture as $G$.  This completes the proof of Lemma~\ref{lemma: share same wall}.
\end{proof}

\subsection{Estimates for elements in the same Weyl chamber}\label{sec: same W C}
 
 We continue our analysis of the dynamics of $G$ relative to the Weyl chamber picture. 
\begin{lemma}\label{lemma: same W cber PH} Suppose $h\in G_0$ has the property that  $A_h$ and $A_f$ lie in the same Weyl chamber. Then  
\begin{enumerate}
\item there exists $c>0$ such that for any $h-$invariant ergodic measure $\nu$,  
\begin{equation*}
\lambda^{\hbox{\tiny{$\mathrm{max}$}}} (Dh|_{E^s_f}, \nu) < -c,\;\hbox{and }\,
\lambda^{\hbox{\tiny{$\mathrm{min}$}}}(Dh|_{E^u_f}, \nu) > c;
\end{equation*} 
\item for every $i$, $Dh$ either uniformly contracts or uniformly expands $E^i$.
\end{enumerate}
\end{lemma}  
\begin{proof}(1): If $A_h$ is in the same Weyl chamber as $A_f$, then as in Lemma \ref{lemma: share same wall} we have that $T_{A_h}$ uniformly contracts $\pi (\W^s_f)$ and uniformly expands $\pi (\W^u_f)$. By Lemma \ref{lemma: Hold exp}, for any $x\in \TT^{d}$ and $y\in \W^s_f(x, loc)$,  
\begin{equation}\label{eqn: bd LE sm chamber}
\limsup_{n\to \infty} \frac{1}{n}\log d_{\W^s_f}(h^n(x), h^n( y))\leq \delta\cdot\lambda^{\hbox{\tiny{$\mathrm{max}$}}}(T_{A_h}|_{\pi(\W^s_f)})<0,
\end{equation}
 
where $\delta$ is the H\"older exponent of $\pi$ in Lemma \ref{lemma: Hold exp}. Then for any $h-$invariant ergodic measure $\nu$,  Lemma~\ref{lemma: h dim reg chart}
and  \eqref{eqn: bd LE sm chamber} together  imply that for $\nu-$almost every $x\in \TT^{n+1}$,  the Pesin stable manifold  passing through $x$ is $\W^s_f(x,loc)$,   and  
\[\lambda^{\hbox{\tiny{$\mathrm{max}$}}} (Dh|_{E^s_f},\nu)\leq \delta\lambda^{\hbox{\tiny{$\mathrm{max}$}}} (T_{A_h}|_{\pi(\W^s_f)})<0.\]
 
Similarly, for any $h-$invariant ergodic measure $\nu$, we have   
\[\lambda^{\hbox{\tiny{$\mathrm{min}$}}}(Dh|_{E^u_f},\nu)\geq \delta \lambda^{\hbox{\tiny{$\mathrm{min}$}}}(T_{A_h}|_{\pi(\W^u_f)})>0.\] 
Setting   $c:=\min(| \delta\lambda^{\hbox{\tiny{$\mathrm{max}$}}} (T_{A_h}|_{\pi(\W^s_f)})|, \delta\lambda^{\hbox{\tiny{$\mathrm{min}$}}}(T_{A_h}|_{\pi(\W^u_f)}) )$   completes the proof of (1).

(2): Since $Dh|_{E^s_f}, Dh|_{E^u_f}$ are continuous,  item (1) of Lemma \ref{lemma: same W cber PH} implies that $Dh|_{E^s_f}$ and $Dh|_{E^u_f}$ satisfy the conditions of Lemma \ref{lemma: est cocyc}. Thus  $Dh|_{E^u_f}$ and $Dh^{-1}|_{E^s_f}$  have uniform exponential growth, which implies (2).  
\end{proof}
Thus we have shown that Proposition \ref{prop: key pr} holds in one case:
\begin{eqnarray*} &&\hbox{\em $h$  and $f$ lie in the same hyperbolic Weyl chamber}\\
&\iff&
 \hbox{\em $A_h$ and $A_f$ lie in the same Weyl chamber}.
\end{eqnarray*}
\subsection{Proof of Proposition \ref{prop: key pr}}\label{sec: key sec}
In this section we will prove Proposition \ref{prop: key pr} for  those $h$ for which $A_h$ and $A_f$ lie in different Weyl chambers.

 Recall that we have ordered the bundles $E^i_f$ and $E^i_{f_0}$ in $i$ by decreasing size of Lyapunov exponents. Write
\[E^u_{f} = E^1_f\oplus E^2_f\oplus\cdots \oplus E^k_f,\quad\hbox{and } E^s_{f} = E^{k+1}_f\oplus\cdots \oplus E^\ell_f,\]
and for $i\leq j$, let
\[E^{[i,j]}_f : = E^{i}_f\oplus E^2_f\oplus\cdots \oplus E^j_f\]

An immediate application of the normally hyperbolic theory in \cite{HPS} implies that for every $i\in [1,k]$,
$E^{[i,k]}\oplus E^c$ is integrable, tangent to an $f$-invariant foliation $\W^{[i,k]c}_f$ that projects under $\pi$ to the affine foliation $\W_{A_f}^{[i,k]}$ tangent to $E^i_{A_f}\oplus\cdots\oplus E^k_{A_f}$. 
Further application of \cite{HPS} gives the following.
\begin{lemma} For every $i \subset [1,k]$, there is an $f$-invariant foliation $\W^{[i,k]}_f$ with the following properties
\begin{enumerate}
\item $E^{[i,k]}_f$ is uniquely integrable, tangent to  $\W^{[i,k]}_f$, and
\item  $\W^{[i,k]}_f$ and $\W^c$  are jointly integrable, tangent to the foliation $\W^{[i,k]c}_f$; the restriction of $\pi$ to $\W^{[i,k]}_f$ is a bi-H\"older homeomorphism onto  $\W_{A_f}^{[i,k]}(\pi(x))$.
\end{enumerate}
\end{lemma}
\begin{proof} 
For fixed $i\in [1,k]$ we apply the graph transform argument for $f$ in restriction to the disjoint union of the leaves of  $\W^{[i,k]c}_f$: as the splitting $E^{[i,k]}_f\oplus E^c_f$ is dominated, and since $E^{[i,k]}_f$ is uniformly expanded, it is uniquely integrable, tangent to a foliation $\W^{[i,j]}_f$.  By construction $\W^{[i,j]}_f$ and $\W^c_f$ are jointly integrable.
\end{proof}

For $j\in [1,k]$, the bundle $E^{[1,j]}_f$ is a strong unstable bundle for $f$ and therefore uniquely integrable, tangent to a foliation $\W^{[1,j]}_f$. By intersecting foliations $\W^{[1,j]}_f$ and $\W^{[i,k]}_f$
we obtain $f$-invariant foliations $\W^{[i,j]}_f$ tangent to the uniquely integrable bundle $E^{[i,j]}_f$, for any interval $[i,j]\subset [1,k]$.  We denote by $\W^m_f$ the foliation $\W^{[m,m]}_f$, which is tangent to $E^m_f$.  Lemma~\ref{lemma: PHcent} implies that the foliations 
$\W^{[i,j]}_f$ are $G_0$-invariant.

Note  that except in the case $j=k$, we do
not know that  $\W^{[i,j]}_f$ projects to  under $\pi$ to $\W_{A_f}^{[i,j]}$ (and {\em a priori} for a single $f$ this will not be the case).  We will need to use the maximality of the $G_0$-action to establish this.  

Set $h_0 = f$. 
Since $G$ induces a maximal Anosov action on $\TT^{d-1}$,  Lemma~\ref{lemma: za max} implies that there exists a Weyl chamber adjacent to that of $A_f$ such that for any element $h_i\in G_0$ with $A_{h_i}$ in this chamber, the signs of all the exponents of $A_f$ and $A_{h_i}$ are the same except one exponent corresponding to $E^{i}_{f_0}$.  By this process we produce elements $h_1,\ldots,h_k\in G_0$.  We prove the following statement inductively, for $i=1,\ldots, k$.

\noindent{\bf Inductive hypotheses ($i$):}  
\begin{itemize}
\item[$(A_i)$]   $\W_f^{[i,k]}$ is absolutely continuous, with $C^2$ leaves, 
\item[$(B_i)$]   $\pi(\W_f^{[i,k]}) = \W^{[i,k]}_{A_f}$, and  $\pi(\W^i_f)=\W^i_{A_f}$.  The restriction of $\pi$ to  $\W^i_f$ leaves is a bi-Holder homeomorphism onto $\W^i_{A_f}-$ leaves with H\"older exponent $\delta\to 1$ as $d_{C^1}(f_0, f)\to 0$.
\item[$(C_i)$]  $Dh_i$ uniformly contracts $E^i_f$ and expands $E^{[i+1,k]}_f$.
\end{itemize}

The inductive hypothesis holds vacuously for  $i=0$.  Assume then that the hypothesis holds for $i-1$, for some $i\in\{1,\ldots, k\}$.  We establish the hypothesis for $i$ in several steps.

\begin{enumerate}

\item[{\bf Step 1:}] Show that  $\W_f^{[i,k]}$ is absolutely continuous, with $C^2$ leaves.
\item[{\bf Step 2:}] Define a $G_0$-invariant topological foliation $\W^\#$,  subfoliating $\W^{[i,k]}_f$, such that  $\pi (\W^\#)=\W^i_{A_f}$. 
\item[{\bf Step 3:}] Show that for any $h_i-$invariant ergodic measure $\nu$, $\W^\#$ coincides with the Pesin stable manifold of $h_i\vert {\W^{[i,k]}_f}$, $\nu-$almost everywhere. In particular $T\W^\#$ is well-defined $\nu-$almost everywhere, and $\W^\#$ is absolutely continuous in $\W^{[m,k]}_f$.
\item[{\bf Step 4:}] Show that $T\W^\#=E^i_f, \mathrm{vol}-a.e.$. Consequently, there is a full volume set $K$ such that for any $x\in K$, $\W^\#(x)$ is a $C^1$ manifold tangent to $E^i_f$ everywhere.
\item[{\bf Step 5:}] Using an approximation argument, show that  $\W^\#=\W^i_f$ and $\pi(\W^i_f)=\W^i_{A_f}$.  Obtain integrability of $E^i_f\oplus E^c$ and the fact that $\pi(\W^{ic}) = \W_{A_f}^{ic}$.  Conclude that the restriction of $\pi$ to $\W^i_f$ leaves is $\delta$ bi-H\"older, with $\delta$
close to $1$.
\item[{\bf Step 6:}] Using Theorem~\ref{lemma: h dim reg chart} , show that $Dh_i$ uniformly contracts $E^i_f$ and uniformly expands $E^{[i+1,k]}_f$.
\end{enumerate}

\textbf{Step 1.}   By the inductive hypothesis $A_{i-1}$, the foliation  $\W_f^{[i-1,k]}$ is absolutely continuous, with $C^2$ leaves. Hypothesis $C_{i-1}$ implies that the derivative $Dh_{i-1}$ uniformly contracts $E^{i-1}_f$ and expands $E^{[i,k]}_f$. It follows that $\W^{[i,k]}_f$ is an absolutely continuous subfoliation of $\W_f^{[i-1,k]}$ with $C^2$ leaves, and thus $\W^{[i,k]}_f$ is absolutely continuous. (Note that for $i=1$, the statement holds, because $\W_f^{[1,k]}=\W^u_f$).

\textbf{Step 2.} We define the $G_0-$invariant topological foliation ${\W^\#}$ to be the lift of $\W^i_{{A_f}}=\W^s_{{A_{h_i}}}\cap \W^u_{{A_f}}$ by $\pi^{-1}$ on  $\W_f^{[i,k]}-$leaves.

{\em Henceforth to simplify notation, we write $h=h_i$, $\W = \W^{[i,k]}_f$,  $E = E^i_f$, $E' = E^{[i+1,k]}_f$, and $F=  E^{[i,k]}_f = E\oplus E' = T\W$.} By construction, the topological foliation  ${\W^\#}$  subfoliates the absolutely continuous foliation $\W$, whose leaves are $C^2$.

\textbf{Step 3.} Since $\pi$ is bi-H\"older when restricted to $\W^u_f$ leaves, the map $h$ contracts distance in the leaves of ${\W^\#}$ exponentially fast; i.e. for any $x\in \TT^d$ and $y\in {\W^\#}(x,loc)$, 
\begin{equation}\label{eqn: exp ctrt hat W}
\limsup_{n\to \infty}\frac{1}{n}\log d_{\W^u_f}(h^n(x), h^n(y))\leq \lambda_0<0,
\end{equation} 
where ${\W^\#}(x,loc)$ is defined to be the lift of $\W^s_{{A_{h}}}\cap \W^{[i,k]}_{{A_f}}(\pi(x),loc)$ to ${\W^\#}(x)$.

Proposition~\ref{prop: stb mfd} implies that for any $G_0-$invariant ergodic measure $\nu$, the leaf ${\W^\#}(x)$ coincides with the (global) Pesin stable manifold $\W^{Pe}_{h\vert\W}(x, gl)$ of $x$ (for the restricted dynamics $h\vert\W$ ), for $\nu-$almost every $x$, since globally $h$ contracts ${\W^\#}$ exponentially fast. Therefore ${\W^\#}(x, loc)$ is tangent to $E^s_{h,\nu}\cap F(x)$ at $x$ for $\nu-$almost every $x$, where $E^s_{h,\nu}$ is the Oseledec stable space of $(Dh, \nu)$.  

\textbf{Step 4.}  Restricting to the case  $\nu=\mathrm{vol}_{\TT^d}$ in Step 3, we then have 
\begin{prop}\label{lemma: key lemma non C} The measurable distribution $E^s_{h, \mathrm{vol}}\cap F$ coincides with $E$, $\mathrm{vol}-$a.e.
\end{prop}
\begin{proof}[Proof of Proposition~\ref{lemma: key lemma non C}] We split the proof into two cases.

\noindent
\textbf{Case 1:} $\dim\left(E^s_{h, \mathrm{vol}}\cap F\right) (=\dim E)=1$. The following lemma is easy to show. 

\begin{lemma}\label{lemma: dom ose spl}There exists $E^j_f$ such that $E^s_{h, \mathrm{vol}}\cap F\subset E^j_f$, $\mathrm{vol}-$almost everywhere. 
\end{lemma}
\begin{proof}[Proof of Lemma~\ref{lemma: dom ose spl}] Evidently $E^s_{h, \mathrm{vol}}\cap F$ is a $\mathrm{vol}-$a.e. defined, one dimensional, $Df-$invariant distribution within $F$, in particular we have Lyapunov exponents defined for $(v,x)$ for $\mathrm{vol}-$a.e. $x$ and $v\in E^s_{h, \mathrm{vol}}\cap F$. 

For any regular point $x$ and any $v\in T_x\TT^d-\{0\}$, the forward and the backward Lyapunov exponents for $v$
exist and coincide. Then $v$ is contained in some Oseledec subspace, and hence is contained in some $E^j_x$. Therefore  $\TT^d$  decomposes as a finite union of measurable sets $\cup_j X_j$ such that
for all $m$,  $f(X_j)=X_j$, and
\[E^s_{h, \mathrm{vol}}\cap F\subset E^j(x),\] 
 for all  $x\in X_j$. By ergodicity of $f$ we have that one of the $X_j$ has full volume.\end{proof}

Let $V:=E^j_f$, where $E^j_f$ is the subbundle obtained from Lemma \ref{lemma: dom ose spl}.
\begin{lemma}\label{lemma: step 2} For a full volume set of  $x\in \TT^d$, the leaf ${\W^\#}(x)$ is a $C^1$ submanifold tangent to $V$ everywhere.  
\end{lemma}
\begin{proof}[Proof of Lemma~\ref{lemma: step 2}]
The absolute continuity of the foliation  $\W$ from Step 1 implies that any set of full volume meets almost every leaf of  $\W$ in a set of full leaf volume.  Hence there is a full volume, $f$-invariant set $P\subset M$ of Pesin regular points for $(f,\mathrm{vol})$ in $M$ such that for every $p\in P$, the leaf  $\W$ meets $P$ in a set of full leafwise volume.

Let $N := \bigsqcup_{p\in M} \W (p)$ be the disjoint union of unstable manifolds: it is a non-compact $C^2$ Riemannian manifold.  The maps induced by $f$ and $h$ on $N$ are $C^2$, with uniform bounds on the derivatives.  Applying the arguments in \cite{PSh} to the (Pesin regular) points in 
$P_N := \bigsqcup_{p\in P} P\cap \W(p)$, we obtain that the Pesin local stable manifolds  
\[\P_{loc}:=\{ \P_{loc}(x):=\W^{Pe}_{h| \W}(x, loc):x\in P\}\]  
of $h|{N}$ form an absolutely continuous family of disks.
 In particular,  for every $p\in P$, a set $B\subset \W(p)$ has  $\mathrm{vol}_{\W}$-measure $0$ in  $\W(p)$  if and only if it has    $\mathrm{vol}_{\P_{loc}(p)}$-measure   $0$ in $\P_{loc}(z)$, for almost every $z\in\W$.

This implies in particular that for $\mathrm{vol}-$a.e.  $x\in \TT^d$, there is a dense subset of   $y\in\P_{loc}(x)$   such that $y$ belongs to $P$. For such $y$,  the smooth disk   $\P_{loc}(x)$   is tangent to $E(y)$. Thus  for $\mathrm{vol}-$a.e.  $x\in \TT^d$, the submanifold   $\P_{loc}(x)$    is tangent to $E$ on a dense subset, and hence by continuity of $E$,   $\P_{loc}(x)$   is tangent to $V$ everywhere and is therefore a $C^1$ submanifold.    

Fix a positive volume, compact Pesin block $\Lambda$ for $h$.  By Pesin theory,   for $y\in \Lambda$,  the size of $\P_{loc}(y)$ is at least $r_0>0$.
Let $x$ be an $f$-regular point in $\Lambda$.  Then there exist infinitely many $n$ such that $f^{-n}(x), n\geq 0$ intersects $\Lambda$ infinitely many times (this property holds for $\mathrm{vol}-$almost every $x$, by Poincar\'e recurrence).  Then the submanifold $f^n(\P_{loc}(f^{-n}(x)))$ 
\begin{itemize}
\item is contained in ${\W^\#}(x)$
\item is tangent to $V$ everywhere.
\item has length  $\geq r_0 e^{\lambda n}$ for some $\lambda>0$, for all $n$ with  $f^{-n}(x)\in \Lambda$.
\end{itemize}
As $n$ tends to infinity, we obtain that ${\W^\#}(x)=\cup_{n\geq 0} f^n(\P_{loc}(f^{-n}(x)))$ is a $C^1$ submanifold  tangent to $V$ everywhere. Since the $h$-Pesin blocks exhaust the volume, we conclude that for $\mathrm{vol}-$a.e. $x$, ${\W^\#}(x)$ is a $C^1$ submanifold tangent to $V$.  Let $K$ be the set of such $x$. Then $K$ is dense since it has full volume, completing the proof of Lemma~\ref{lemma: step 2}.
\end{proof}

Now we claim that $j=i$, and  thus $V=E$. 
Suppose $j\neq i$. By H\"older continuity of $\pi$, there exist positive constants $\epsilon_1, C_1$ such that for any $x$ and any $y\in {\W^\#}(x)$ with $d_{{\W}}(x,y)\leq \epsilon_1$,  we have
\begin{equation}\label{eqn: pi quasi iso glob}
d_{\TT^{d-1}}(\pi(x),\pi(y))\leq C_1\, \epsilon_1.
\end{equation}

Now we pick an arbitrary $x\in K$ and consider the $C^1$ submanifold ${\W^\#}(x)$. Since, by Lemma~\ref{lemma: step 2},  ${\W^\#}(x)$ is everywhere tangent to $V=E^j_f$,  $f$ expands ${\W^\#}$ at a rate slower than $e^{\lambda^j(f_0)+\eta}$ for some $\eta\ll \lambda^{j-1}(f_0)-\lambda^j(f_0)$ (by smallness of  $d_{C^1}(f,f_0)$). Choose $y\in {\W^\#}(x,loc)$ such that $d_{{\W^\#}}(x,y)\leq \epsilon_1$. 
For $n$ large, $\pi(f^n(x)),\pi(f^n(y))$ can be connected by a $\W^i_{{A_f}}-$path with length less than $O(C_1\,\epsilon_1\, e^{n(\lambda^j+\eta)})$. But 
since $T_{A_f}$ expands $\W^i_{{A_f}}$ leaves at a constant rate  $e^{\lambda^i(f_0)}$, the points $\pi(f^n(x)),\pi(f^n(y))$ cannot be linked by a $\W^i_{{A_f}}$ path with length  $o(e^{n\lambda^j(f_0)})$,  a contradiction. Therefore $j$ must be $i$. This completes the proof of 
Proposition~\ref{lemma: key lemma non C} in Case 1.
\smallskip

\noindent
\textbf{Case 2:} $\dim\left(E^s_{h, \mathrm{vol}}\cap F\right)  \,(=\dim E)>1$. Suppose that $E^s_{h, \mathrm{vol}}\cap F$ does not coincide $\mathrm{vol}-$a.e.~with $E$.
Then we have 
\begin{lemma}\label{lemma: non triv inters with rest} The measurable distribution $E^s_{h, \mathrm{vol}}\cap F$ has  non-trivial intersection (over a positive volume set) with $E'=E^{[i+1,k]}_f$.
\end{lemma}
\begin{proof}[Proof of Lemma~\ref{lemma: non triv inters with rest}] Suppose that $E^s_{h, \mathrm{vol}}\cap F $ has  trivial intersection with $E'$, $\mathrm{vol}$-almost everywhere. Since $E^s_{h, \mathrm{vol}}\cap F$ does not coincide $\mathrm{vol}-$a.e.~with $E$,  Lusin's theorem implies that there is a compact set $K_2$ with positive volume and a positive constant $\delta_2$ such that for any $x\in K_2$, 
\begin{equation}\label{eqn: bd dist d2}
\angle(E^s_{h, \mathrm{vol}}\cap F (x), E(x))>\delta_2.
\end{equation}
 Therefore for any $n\geq 1$ such that $f^{n}(x)\in K_2$,
\begin{equation}\label{eqn: bd dist iteration}
\angle(E^s_{h, \mathrm{vol}}\cap F (f^n(x)), E(f^n(x)))>\delta_2.
\end{equation}
 On the other hand since $F=E\oplus E'$ is a dominated splitting, and $E^s_{h, \mathrm{vol}}\cap F (x)$ is $Df-$invariant and has trivial intersection with $E'$, we have that 
\[\lim_{n\to \infty}\angle(E^s_{h, \mathrm{vol}}\cap F (f^n(x)), E(f^n(x)))= \angle(Df^n( E^s_{h, \mathrm{vol}}\cap F(x)), Df^n(E(x)) )=0.\]
If $x\in K_2$ is recurrent, then this contradicts  (\ref{eqn: bd dist iteration}).  Since almost every $x\in K_2$ is recurrent, this gives a contradiction, completing the proof of  Lemma~\ref{lemma: non triv inters with rest}.
\end{proof}

As in the proof of Lemma \ref{lemma: step 2}, we  thus obtain that ${\W^\#}$ is absolutely continuous and that there is a full volume set $K\subset \TT^d$ such that for any $x\in K$, ${\W^\#}(x)$ is a $C^1$ submanifold, and $T{\W^\#}(x)$ has non-trivial intersection with $E'$ everywhere. This uses Lemma~\ref{lemma: non triv inters with rest}, and the continuity of $E'$ and $T{\W^\#}$ along the leaves of $\W^\#$.

Moreover, by the Cauchy-Peano existence theorem, for $x\in K$, there exists a $C^1$ path $\gamma:I\to {\W^\#}(x)$ such that for any $t\in I$, $\gamma'(t)\in E' \cap T{\W^\#}$. Let  $z_0=\gamma(0)$ and  $z_1=\gamma(1)$. As in Case 1, $f^{n}(z_0)$ and  $f^n(z_1)$ can be linked by a $C^1$ path $f^n(\gamma)$ in ${\W^\#}(x)$ of length $O(e^{n(\lambda^{i+1}(f_0)+\eta')})$ for some $\eta'\ll \lambda^{i}(f_0)-\lambda^{i+1}(f_0)$; this implies that $\pi(f^n(z_0))$ and $\pi(f^n(z_1))$ can  be linked by a $\W^i_{{A_f}}-$path of length $O(e^{n(\lambda^{i+1}(f_0)+\eta')})$ (since \eqref{eqn: pi quasi iso glob} holds here). On the other hand, since $\pi(f^n(z_i))={T^n_{A_f}}(\pi (z_i)), i=0,1$, it follows that $\pi(f^n(z_0))$ and $\pi(f^n(z_1))$ cannot be connected by a $\W^i_{{A_f}}-$path of length $o(e^{n\lambda^1(f_0)})$, which is a contradiction.   This completes the proof of Proposition~\ref{lemma: key lemma non C}.
\end{proof}

Combining Proposition~\ref{lemma: key lemma non C} from Step 3 with Lemma~\ref{lemma: step 2}, we obtain a full volume subset  $K\subset \TT^d$  such that for $x\in K$, $\W^\#(x)$ is a $C^1$ manifold tangent to $E$ everywhere and coinciding with a global Pesin stable manifold.
 By the absolute continuity of the family Pesin disks tangent to $E$  and Fubini's theorem, it  follows that $\W^\#$ is absolutely continuous.

\textbf{Step 5.} If a topological foliation has almost every leaf coinciding with another topological foliations, the two foliations must coincide.
It follows that $\W^\# = \W^i_f$. and in particular $\pi(\W^i_f) = \W^i_{{A_f}}$.   Observe that the leaves of $\pi^{-1}\left( \W^i_{{A_f}}\right)$ are jointly subfoliated by $\W^c_f$ and  $\W^i_f$, both of which have $C^1$ leaves. Therefore $E\oplus E^c$ is integrable and tangent to $\W^{ic}$. Integrability of $E' \oplus E^c$ follows from normal hyperbolicity.

Note that any leaf conjugacy from $(f, \W^c_f)$  to  $(f_0, \W^c_{f_0})$ to close to the identity must map $\W^{ic}_{f_0}$ to $\W^{ic}_f$.
Lemma~\ref{lemma: multi distri PSW} implies there is a leaf conjugacy $h^c$ to  $(f_0, \W^c_{f_0})$ that is $\delta$-bi-H\"older along $\W^{i}_{f_0}$ leaves,  where
$\delta\to 1$ as $d_{C^1}(f,f_0)\to 0$.     Then $\pi = P\circ h^c$, where $P\colon \TT^d\to \TT^{d-1}$ is the coordinate projection.  It follows that that the restriction of $\pi$ to $\W^i_f$ leaves has bi-H\"older exponent $\delta$ as well.  This completes Step 5.

\textbf{Step 6.} Since $T_{A_h}^n$ exponentially contracts distances in $\W^i_{A_f}$ leaves and $\pi$ is bi-H\"older between $\W^i_f$ and $\W^i_{A_f}$ leaves, we have that $h^n$ exponentially contracts distances in $\W^i_f$ leaves.  Lemma~\ref{lemma: h dim reg chart} implies that $Dh$ uniformly contracts $T\W^i_f =E^i_f = E$.

Similarly, since $T_{A_h}^{-n}$ exponentially contracts distances in $\W^{[i+1,k]}_{A_f}$ leaves, $h^{-n}$ exponentially contracts distances in   $\W^{[i+1,k]}_{f}$ leaves.  Again, Lemma~\ref{lemma: h dim reg chart} implies that $Dh$ uniformly expands $T\W^{[i+1,k]}_{f} = E^{[i+1,k]}_f = E'$.  Note that he same argument holds for any  $h\in G_0$ such that $A_h$ is not in a Weyl chamber wall.

This completes the induction, and so Proposition \ref{prop: key pr} holds for the elements $h_1,\ldots h_k$ and the bundles $E^1_f,\ldots, E^k_f$. Working in the stable bundle $\W^s_f$ with $f^{-1}$ we obtain elements $h_{k+1}\,\ldots, h_\ell$ satisfying the conclusions of Proposition \ref{prop: key pr} for $f^{-1}$ and the bundles $E^{k+1}_f,\ldots, E^\ell_f$.  The remark above shows that that  the conclusions hold for any $h\in G_0$ such that $A_h$ is not in a Weyl chamber wall.   This completes the proof of Proposition \ref{prop: key pr}.

\subsection{Existence of partially hyperbolic elements and topological rigidity}\label{sec: integra and top rig}

We now return to the proof of    Proposition~\ref{lemma: G high rank ac}.  The next step is to show that there is a partially hyperbolic element in every hyperbolic Weyl chamber.

\begin{prop}\label{lemma: exst PH}Suppose $f$ satisfies the hypotheses of Proposition~\ref{lemma: G high rank ac}.  Then in each hyperbolic Weyl chamber of $G_0$, there exists a partially hyperbolic element $h$.
\end{prop}

\begin{proof}[Proof of Proposition~\ref{lemma: exst PH}]
Fix an $h\in G_0$ such that $A_h$ is not in any Weyl chamber wall.
\begin{lemma}\label{lemma: bd LE Dh E}For any $\epsilon>0$, there exists $n\in \ZZ^+$ such that for any $j$, $\log \mathrm{Jac}|Dh^n|
_{E^j_f}|$ lies in the interval
\begin{equation}\label{eqn: bd LE Dh E}\begin{cases}
(\dim(E^j_f)\cdot(\delta^{-1}\lambda^j(A_{h})-\epsilon)\cdot n),~~\dim(E^j_f) \cdot(\delta \lambda^j(A_{h})
+\epsilon)\cdot n), \text{ if }\lambda^j(A_{h})<0;\\
(\dim(E^j_f) \cdot(\delta\lambda^j(A_{h})-\epsilon)\cdot n),~~\dim(E^j_f) \cdot (\delta^{-1} \lambda^j(A_{h})
+\epsilon)\cdot n), \text{ if }\lambda^j(A_{h})>0,
\end{cases}
\end{equation} 
where $\lambda^j(A_{h})$ is the Lyapunov exponent  of ${A_{h}}|_{\W^j_{{A_f}}}$, $\delta\approx 1$ is the H\"older exponent given by   Proposition~\ref{lemma: G high rank ac}, and $\mathrm{Jac}(\cdot|_{E^j_f})$ is the leafwise Jacobian for the map restricted on $\W^j_f$, .
\end{lemma}
\begin{proof}Without loss of generality, assume that $\lambda^j(A_{h})<0$. Lifting the action of $T_{A_h}$ to $h$ and using the $\delta$-H\"older continuity of $\pi$ restricted to $\W^j_f$, we obtain that that for each $\epsilon>0$, there exist
$\eta>0$ and $N\in \NN$ such that for all $x\in M$, and $y\in \W^j_f(x)$:
\[d(x,y)<\eta \implies  e^{(\delta^{-1}\lambda^j(A_{h})-\epsilon)n} \leq d(f^n(x), f^n(y)) \leq  e^{(\delta^{-1}\lambda^j(A_{h})+\epsilon)n},
\]
for all $n\geq N$.
The conclusion follows easily from Lemma~\ref{lemma: h dim reg chart}, completing the proof of Lemma~\ref{lemma: bd LE Dh E}.
\end{proof}


Lemma \ref{lemma: g pr vol} implies that $h$ is volume preserving. Since $E^i_f, E^c_f$ are all  continuous distributions in $T\TT^{d}$, there exists $C_0\geq 1$, depending only on the angles between $E^j_f, E^c_f$, such that for any $k\in \ZZ$, 
\begin{equation}\label{eqn: est ec h}
C_0^{-1}\leq (\prod_j\mathrm{Jac}(Dh^k|_{E^j_f}))\cdot \|Dh^k|_{E^c_f}\|\leq C_0;
\end{equation}
since $A_h$ has determinant $1$, we also have
\begin{equation}\label{eqn: lambd sum 0}
\sum_j\dim(E^j)\lambda^j(A_{h})=0.
\end{equation}
Therefore by  \eqref{eqn: est ec h}, \eqref{eqn: lambd sum 0}  and Lemma~\ref{lemma: bd LE Dh E}, we have that for $n$ large enough,
\begin{equation}\label{eqn: est fin ec h}
\|D{h}^n|_{E^c_f}\|\in [e^{-\gamma n}, e^{\gamma n}],
\end{equation}
where $\gamma$ is small if $\delta$ is sufficiently close to $1$ and $\epsilon$ in \eqref{eqn: bd LE Dh E} is small.

 Comparing \eqref{eqn: est fin ec h} with  \eqref{eqn: bd LE Dh E}, for $|\lambda^j(A_{h})|\gg \gamma$ (which holds for any $f$ which is sufficiently $C^1$ close to $f_0$ and any $h$ that is not close to the Weyl chamber wall), we get $h$ is in fact a partially hyperbolic diffeomorphism, with  $E^s_{h} \oplus E^u_{h} =\oplus_j E^j_f,$ and  $E^c_{h}=E^c_f$, completing the proof of Proposition~\ref{lemma: exst PH}.\end{proof}



From the existence of partially hyperbolic elements in every chamber, we obtain  topological rigidity of the action.
\begin{prop}\label{coro: hold rig} If $G$ induces a maximal Anosov action on $\TT^{d-1}$, then
\begin{enumerate} 
\item there exists a $G_0-$invariant continuous metric on $E^c_f$; and  
\item $f$ is H\"older conjugate to $T_{A_f}\times R_\theta $  for some $\theta\notin \QQ/\ZZ$.   Similarly,  any $h\in G_0$ is H\"older conjugate (by the \textbf{same} H\"older conjugacy) to a product of $T_{A_h}$ with a circle rotation.
\end{enumerate}
\end{prop}
\begin{proof}  Proposition~\ref{lemma: exst PH}  implies that the $G_0$ action is partially hyperbolic, with hyperbolic subbundle  $E^H:=\oplus_j E^j_f$.  The rigidity of such actions is studied in  \cite{DX, DX0}; in particular, the proofs of Proposition 8.1 in \cite{DX0} and Proposition 5.1 in \cite{DX} imply that
 $E^H$ is tangent to a $C^1$ foliation $\W^H$.

Denote by $\W^c_f(x_0)$ a  $G_0-$fixed center leaf. Since $E^H$ is integrable, there is no open accessibility class for $f$.  Proposition~\ref{lemma: dich atom rot} implies that   $\TT^d$ has a \emph{product structure}, i.e. $\TT^d$ is topologically the product of  $\W^c_f(x_0)$ and $\TT^{d-1}/\W^c$. By H\"older continuity of $\pi$ and $\W^c_f$ this product structure is H\"older continuous as well.

Consider  the projection $\mathrm{Pr}^c$ from $\TT^d$ to $\W^c_f(x_0)$ along $\W^H$. Since $\W^H$ is a $C^1$ foliation, $\mathrm{Pr}^c$ is $C^1$ as well. Therefore  $\mathrm{Pr}^c_\ast (\mathrm{vol}_{\TT^d})$ is an $f-$invariant volume on $\W^c_f(x_0)$ with continuous density function, and $f|_{\W^c_f(x_0)}$ is $C^1$ conjugate to a circle rotation $R_\theta$.
By ergodicity of $f$, the rotation number $\theta$ must be irrational. 

The continuous density function mentioned above gives an $f-$invariant continuous metric on $T\W^c_f({x_0})$, and this pulls back via $D\mathrm{Pr}^c\vert_{E^c}$ to an $f$-invariant metric on $E^c$. Since the construction of this $f-$invariant continuous metric on $E^c$ only depends on the product structure and the volume form on $\TT^d$, it must be $G_0-$invariant. This proves (1).

For (2),  we know that the action induced by $f$ 
\begin{itemize}
\item on $\TT^d/\W^c$ is H\"older conjugate to $T_{A_f}$ on $\TT^{d-1}$; and
\item on $\TT^d/\W^H$ is $C^1-$conjugate to  $R_\theta$.
\end{itemize}
Using the product structure of $f$, we obtain that $f$ is H\"older conjugate to the product of $T_{A_f}$ on $\TT^{d-1}$ with an irrational rotation $R_\theta$.  The same proof also works for any $h\in G_0$ (although if $h$ is not ergodic, the rotation number might not be irrational). 
Therefore, by the same conjugacy, $h$ is H\"older conjugate to the product of  $T_{A_h}$ with a circle rotation.
\end{proof}

\subsection{  Absolute   continuity of $\W^c_f$: volume and  equilibrium states.}\label{sec: e stat pres}
The following proposition is a partially hyperbolic version of Theorem 20.4.1. in \cite{KH}.
\begin{prop}\label{prop: eq state vol}
Let $f\colon M\to M$ be a $C^{1+}$, volume preserving partially hyperbolic diffeomorphism. Suppose that for any $f-$invariant ergodic measure $\nu$, the central Lyapunov exponents of $f$ with respect to $\nu$ are all zero.
Then the volume $\mathrm{vol}_M$ is an equilibrium state of the potential $\varphi:=-\log J^u(f):=-\log|\det Df|_{E^u}|$.
\end{prop}
\begin{proof}The proof is basically contained in \cite{HWZ}.  By the  Pesin entropy formula \cite{P77} and the vanishing of the central Lyapunov exponents, we have 
 \begin{equation*}\label{eqn: mea entpy}
h_{\mathrm{vol}}(f)=\int_M \log J^u(f)(x)d \mathrm{vol}(x).
\end{equation*}
Therefore $P_{\mathrm{vol}}(\varphi)=0$, where 
\[P_{\mathrm{vol}}(\varphi) = h_{\mathrm{vol}}(f) + \int \varphi\, d\mathrm{vol}. \]
is the free energy of $\varphi$ with respect to $\mathrm{vol}$. We need only show that that the pressure of $\varphi$ vanishes:
\[P(\varphi) := \sup_{\mu: f_\ast\mu=\mu}  \left(h_{\mu}(f) + \int \varphi\, d\mu\right) = 0.
\]

In \cite{HWZ}, the authors introduce the concept of \emph{unstable pressure} $P^u(f,\psi)=P^u(\psi)$ for any continuous $\psi$ and $C^1$ partially hyperbolic diffeomorphism $f$.   Corollary A.2 and the paragraph right after the statement of Corollary A.2 in \cite{HWZ} implies that  $P^u(\psi)\leq P(\psi)$ for any continuous $\psi$. Moreover if $f$ is $C^{1+}$ and there is no positive Lyapunov exponent in the center direction with respect to any $f-$invariant ergodic measure $\nu$, then  equality holds.  
Corollary C.1 in \cite{HWZ} implies that $P^u(\varphi)=0$. for the potential $\varphi=-\log J^u(f)$.

The assumptions of Proposition \ref{prop: eq state vol} imply that  for any $\psi\in C(M,\RR)$, 
$P^u(\psi)= P(\psi)$, and it follows that $P(\varphi)=0$. This completes the proof of Proposition \ref{prop: eq state vol}.\end{proof}

\subsection{Absolute continuity of $\W^c_f$: $\W^H$-leafwise cocycle rigidity of higher rank partially hyperbolic actions}\label{sec: cocyc rig}
 
Proposition~\ref{coro: hold rig}, implies that the action of $G_0$ on $\TT^d$ is H\"older conjugate to an \emph{irrational rotation extension} $\al$ over $\bal$, where $\bal$ is the  maximal linear Anosov $\ZZ^{d-2}$ action on  $\TT^d$. Here $\al$ is a  $\ZZ^{d-2}-$action on $\TT^d$ 
defined  by $\al(\mathbf{a}) = \bal(\mathbf{a})\times R_{\theta(\mathbf{a})}$
for $\mathbf{a}\in \ZZ^{d-2}$, where $\mathbf{a}\mapsto R_{\theta(\mathbf{a})}$ is an action by circle rotations with at least one $\theta(\mathbf{a})$  irrational. 



Recall that a continuous function $\beta: \ZZ^{d-2}\times \TT^{d}\to \RR $ is an \emph{(additive) cocycle} over $\al$ if 
$\beta(a+b,x)=\beta(a,\al(b)\cdot x)+\beta(b,x)$ holds
for all  $a,b\in \ZZ^{d-2}$ and $x\in \ZZ^{d}$.
A cocycle $\beta_1$ is  \emph{cohomologous to} another cocycle $\beta_2$ if there exists a continuous function (called the \emph{transfer function}) $\Psi: \TT^{d}\to \RR$ such that $\beta_1(a,x)=\beta_2(a, x)+\Psi(\al(a)\cdot x)-\Psi(x)$.


It is well known that for a maximal $\ZZ^{d-2}-$Anosov action $\bal$ on $\TT^{d-1}$ (see Lemma \ref{lemma: R-val cocy rig tori}),   any H\"older continuous cocycle over $\bal$ is cohomologous to a constant cocycle. 

We obtain here a corresponding result for the irrational rotation extension $\al$ over $\bal$.  A cocycle $\beta$ on $\TT^{d-1}\times\TT$ is \emph{constant on $\TT^{d-1}$} if $\beta(a,x)=\beta(a,y)$ whenever $x,y$ have  the same $\TT$-component, i.e. they lie on the same leaf of the {\em  horizontal $\TT^{d-1}$-foliation} $\{\TT^{d-1}\times\{t\}: t\in \TT\}$.

\begin{prop}\label{prop: coh rig ph act}  Let $\al$ be an  irrational rotation extension over a maximal, linear Anosov  $\ZZ^{d-2}$-action $\bal$ on $\TT^{d-1}$.
Then any H\"older continuous cocycle over $\al$ is cohomologous to a cocycle that  is constant on $\TT^{d-1}$.
\end{prop}

This proposition is a direct corollary of the following more general result on partially hyperbolic actions: 
\begin{prop} \label{prop: cyc rig WH} Let $\al$ be a partially hyperbolic $\mathbb Z^k$ action with coarse Lyapunov distributions $E^j$ and corresponding coarse Lyapunov foliations $\mathcal F^j$, $j=1,\dots,  r$. Assume that:
\begin{enumerate}
\item $\bigoplus_{j=1}^r E^j$ integrates to a H\"older foliation $\W^H$ with compact smooth leaves. 
\item For any two $i, j\in \{1, \dots, r\}$ there exists a Weyl chamber $\mathcal C$ and an action element $a\in \mathcal C$ such that $\al(a)$ is partially hyperbolic and uniformly contracts both   $E^i$ and $E^j$. 
\end{enumerate}
Then any H\"older continuous cocycle over $\al$ is cohomologous to a cocycle that  is constant along the leaves of $\W^H$.
\end{prop}
\begin{proof}

The proof is an application of the periodic cycle functionals argument for higher rank actions developed in \cite{DK2, KK} (cf.  \cite{W} for the rank-$1$ case). The main idea is that within each accessibility class of the action one can define a transfer map for the cocycle along \emph{Lyapunov paths} (these are broken paths with pieces completely contained in leaves of foliations $\mathcal F^1, \dots, \mathcal F^1$), see  \cite[Definition 4]{DK2}. Such a transfer map gives rise to a well-defined global H\"older map provided that its values along any two broken paths with same endpoints are the same \cite[Definition 5]{DK2}. In other words, the value of the linear functional thus defined (called the periodic cycle functional \cite[Proposition 2]{DK2}) should be trivial on a closed Lyapunov path. This holds as in \cite[Section 3.3]{DK2} if the system of foliations   $\mathcal F^1, \dots, \mathcal F^1$ satisfies the condition (2),  which is also known as the \emph{totally non-symplectic} (TNS) condition.  The actions considered in \cite{DK2} are assumed to be accessible, so the whole manifold is one accessibility class and the periodic cycle functionals argument implies in the case of actions in \cite{DK2} that any H\"older cocycle is cohomologous to an everywhere constant cocycle. 

In the situation we have here the exact same argument applies along leaves of the $\W^H$ foliation, since within each leaf we have accessibility of the coarse Lyapunov foliations and property (2).  By the same argument as in \cite[Section 3.3]{DK2}, this implies that  any H\"older cocycle over $\al$ is cohomologous to a cocycle which is constant along the leaves of the foliation $\W^H$.

\end{proof}

\subsection{Absolute continuity of $\W^c_f$: uniqueness of the measure of maximal entropy}\label{sec: unique mme}
 Consider the diffeomorphism $T_A\times R_\theta: \TT^{d-1}\times\TT\to \TT^{d-1}\times \TT$ where $R_\theta$ is an irrational rotation on circle and $A\in\SL(d-1,\ZZ)$ is hyperbolic. 
\begin{lemma}\label{lemma: uniq max ent}The volume $\mathrm{vol}_{\TT^{d}}$ on $\TT^{d-1}\times \TT$ is the unique measure of maximal entropy of $T_A\times R_\theta$.
\end{lemma}
\begin{proof} This lemma is probably well-known; we sketch the proof.  The projection of the measure of maximal entropy $\nu$ for $T_A\times R_\theta$ to $\TT^{d-1}$ is the measure of maximal entropy for $T_A$, which is volume.   On the other hand, the projection of $\nu$ to the circle  $\TT$ is $R_\theta$-invariant, and hence is Lebesgue measure.  Therefore $\nu$ must be $\mathrm{vol}_{\TT^{d}}$, since zero entropy systems are disjoint from  Bernoulli systems (cf. \cite{F67}).\end{proof}



The following  proposition is a corollary of Proposition \ref{prop: coh rig ph act}. 
\begin{prop}\label{prop: uque mx etrpy msur}  Let $f$ satisfy the hypotheses of Theorem~\ref{main: thm pr}, and let  $G_0, G\subset \Diff(\TT^{d})$ be the finitely generated abelian groups defined  in Section \ref{sec: start ass}. If $G$ defines a maximal linear Anosov action, then the  volume $\mathrm{vol}_{\TT^d}$ is the unique measure of maximal entropy of $f$.
\end{prop}
\begin{proof}By the discussion in Section \ref{sec: e stat pres} and Proposition~\ref{coro: hold rig} we know that $\mathrm{vol}_{\TT^d}$ is an equilibrium state of the potential $\varphi:=-\log J^u(f)$ for $f$. We define the cocycle $\beta:=-\log J^u$ over the action of  $G_0$  as follows. For $f_1\in G_0,x\in \TT^{d}$, we set
$$\beta(f_1,x):=-\log|\det Df_1|_{E^u_f(x)}|.$$
Clearly $\beta$ is a cocycle over the action of $G_0$, and $\beta(f,x)=\varphi(x)$, for all $x$.

The action of $G_0$ is H\"older conjugate to the algebraic action $\al$ defined in Section \ref{sec: cocyc rig}. By Proposition \ref{prop: coh rig ph act} we know any H\"older continuous cocycle over $\al$ is cohomologous to a cocycle that is constant on $\TT^{d-1}$. Therefore $\beta$ must be cohomologous to a cocycle that is constant on each horizontal $\W^H-$leaf. In particular, there exist continuous functions $\psi, \Psi: \TT^{d}\to \RR, $ such that
\begin{equation}\label{eqn: coh eqn varphi}
\varphi=\psi + \Psi\circ f -\Psi,
\end{equation}and $\psi(x)=\psi(y)$ whenever $x,y$ lie in the same $\W^H-$leaf. 

As in the proof of Proposition~\ref{coro: hold rig}, we denote by $\W^c_f(x_0)$ a $G_0-$fixed center leaf, and let $\mathrm{Pr}^c: \TT^{d}\to \W^c_f({x_0}) $ be the projection along the horizontal foliation $\W^H$. Then $\psi$ defined in \eqref{eqn: coh eqn varphi} induces a well-defined continuous function $\psi^c$ on $\W^c_f(x_0)$ such that 
\begin{equation}\label{eqn: def psi c}
\psi=\psi^c\circ \mathrm{Pr}^c.
\end{equation}

Now we claim that for any $f-$invariant measure $\mu$, $\int_{\TT^d} \varphi\, d\mu$ is independent of $\mu$.  Indeed
 \begin{eqnarray*}
\int_{\TT^{d}}\varphi \, d\mu&=&\int_{\TT^{d}} (\psi+\Psi\circ f-\Psi) \,d\mu~~\text{(by \eqref{eqn: coh eqn varphi})}= \int_{\TT^{d}}\psi\, d\mu ~~\text{(since $\mu$ is $f-$invariant)}
\\
&=&\int_{\W^c_f(\bar{x_0})}\psi^c d \mathrm{Pr}^c_\ast(\mu)~~\text{(since $\psi$ is constant along each horizontal leaf)}.
\end{eqnarray*} 
But $f|_{\W^c_f(x_0)}$ is uniquely ergodic, and $\mathrm{Pr}^c_\ast(\mu)$ is $f-$invariant on $\W^c_f(x_0)$. Then the integral $\int_{\W^c_f(x_0)}\psi^c \,d \mathrm{Pr}^c_\ast(\mu)$  (and hence $\int_{\TT^{d}}\varphi \,d\mu$) is independent of $\mu$.  Write $s(\varphi)$ for the value $\int_{\TT^{d}}\varphi\, d\mu$ of this integral.

Since $\mathrm{vol}_{\TT^{d}}$ is an equilibrium state of the potential $\varphi$, we have that 
\begin{eqnarray*}
P_{\mathrm{vol}}(\varphi)
&=&\sup_{\mu~\text{is }f-inv}h_\mu(f)+\int_\mu\varphi\\
&=&\sup_{\mu~\text{is }f-inv}h_\mu(f)+s(\varphi)~~\text{(since $\int_\mu\varphi=s(\varphi)$, which is independent of $\mu$).}
\end{eqnarray*}
But $P_{\mathrm{vol}}(\varphi)=h_{\mathrm{vol}}(f)+\int_{\mathrm{vol}} \varphi=h_{\mathrm{vol}}(f)+s(\varphi).$ Therefore $h_{\mathrm{vol}}(f)=\sup_{\mu~\text{is }f-inv}h_\mu(f)$, which implies $\mathrm{vol}_{\TT^{d}}$ is a measure of maximal entropy of $f$. But by Proposition~\ref{coro: hold rig} we know $f$ is conjugate to  $T_{A_f}\times R_\theta$, for some $\theta\notin \QQ$, therefore by Lemma \ref{lemma: uniq max ent}, $\mathrm{vol}_{\TT^d}$ is the unique measure of maximal entropy of $f$.\end{proof}
As a corollary, the conjugacy between $f$ and $T_{A_f}\times R_\theta$ identifies the measure of maximal entropy $\mathrm{vol}_{\TT^{d}}$ of $T_{A_f}\times R_\theta$  with the measure of  maximal entropy $\mathrm{vol}_{\TT^{d}}$ of $f$.  Recall that $\mathrm{vol}_{\TT^{d}}$, the measure of maximal entropy of $T_{A_f}\times R_\theta$ is the product of $\mathrm{Pr}^{\TT}_\ast(\mathrm{vol}_{\TT^{d}})$ and $\mathrm{Pr}^{\TT^{d-1}}_\ast(\mathrm{vol}_{\TT^{d}})$. Therefore $\mathrm{vol}_{\TT^d}$, the measure of maximal entropy of $f$, is the product of $\mathrm{Pr}^{c}_\ast(\mathrm{vol}_{\TT^d})$ and $\mathrm{Pr}^{H}_\ast(\mathrm{vol}_{\TT^d})$, where $\mathrm{Pr}^{H}$ is the projection from $\TT^d$ to $\TT^d/\W^c_f$ along $\W^c_f$. 

In particular, since  $\mathrm{Pr}^{c}_\ast(\mathrm{vol}_{\TT^d})$ is absolutely continuous with respect to the Lebesgue measure on $\W^c_f(\bar{x_0})$ (since $\mathrm{Pr}^c$ is $C^1$!), it follows that $\mathrm{vol}_{\TT^d}$ has Lebesgue disintegration along $\W^c_f$. This  completes the proof of Proposition~\ref{lemma: G high rank ac}, which implies  Theorem~\ref{main: thm pr}.

 \section{Proof of Theorem \ref{main: dich}}\label{sec: pf Thm dich}
  
 Let $f_0$ be as in Theorem \ref{main: dich}. Let $f\in \Diff^\infty_{\mathrm{vol}}({\TT^d})$ be a $C^1-$small ergodic perturbation of $f_0$.   
Denote by $\lambda^i(f_0)$ the distinct Lyapunov exponents of $f_0$ (ordered in $i$ by decreasing size) and   by  $T\TT^d=\oplus E^i_{f_0}\oplus E^c_{f_0}$ the corresponding $Df_0-$invariant Lyapunov splitting. Let $T\TT^d=\oplus E^i_f\oplus E^c_f$ be the corresponding $Df-$invariant dominated splitting.  
\begin{lemma}\label{lemma: f narr band} If $d_{C^1}(f,f_0)$ is sufficiently small then the cocycles $Df^{-1}|_{E^u_f}, Df|_{E^s_f}$ satisfy the narrow band condition defined in Section \ref{sec: normal form}.
\end{lemma}
\begin{proof}
It is clear that 
the cocycles $Df_0^{-1}|_{E^u_{f_0}}, Df_0|_{E^s_{f_0}}$ have point Mather spectrums.
If $d_{C^1}(f,f_0)$ is small then $E^i_f$ is close to $E^i_{f_0}$
and therefore the Mather spectrum of $Df|_{E^i_f}$ for each $i$ is contained in an arbitrarily small narrow band, which implies Lemma \ref{lemma: f narr band}.
\end{proof}

Since $f$ is leaf conjugate to $f_0$, there is an $f-$fixed center leaf $W^c_f(x_0)$. As in the proof of Proposition~\ref{lemma: prpty G0 G},   for any $s\geq 1$, $\Z_{s}(f)$ is virtually $G_0$, where $$G_0:=\{h\in \Z_{s}(f): h\text{ preserves the orientation of }\W^c_f,\text{ and } h(\W^c_f(x_0))=\W^c_f(x_0)\}.$$   By Proposition~\ref{lemma: dich atom rot} there is a H\"older continuous fiber bundle $\pi: \TT^d\to \TT^{d-1}$  such that $\pi\circ f=T_{A_f}\circ \pi$, and the fibers of $\pi$ are leaves of $\W^c_f$. For any $h\in G_0$, $h$ preserves the fiber bundle structure, and there is an automorphism $T_{A_h}\colon\TT^{d-1}\to \TT^{d-1}$ such that $\pi\circ h=T_{A_h}\circ \pi$. As in the proof of Proposition~\ref{lemma: prpty G0 G}, we consider the group $\Z^c$  of center-fixing elements in $G_0$ and we let $G=\{A_h:  h\in G_0 \}$. Then $G_0$ is a group extension of $G$ by $\Z^c$.

 In the case that $\W^c_f$ is a smooth foliation, the volume has a smooth disintegration along $\W^c_f$, and  $f$ is smoothly conjugate to an ergodic smooth isometric extension of $g$ such that $\rho$ is homotopic to identity. We have the following lemma for $g$ and $\rho$:
 \begin{lemma}\label{lemma: iso ext classsify} Let $r$ be as in Theorem \ref{main: dich}. If $d_{C^1}(f,f_0)$ is sufficiently small and $\W^c_f$ is a smooth foliation, then one of the following holds.
\begin{enumerate}
 \item  $\Z_{s}(f)$ is virtually $\ZZ\times \TT$ for every $s\geq r$. In this case, either $g$ is not $C^\infty$ conjugate to $T_{A_f}$, or $\rho$ is not $C^\infty$ cohomologous to a constant.
\item $\Z_{s}(f)$ is virtually $\ZZ^{\ell_0(A_f)}\times \TT$ for every $s\geq 1$, and $g_\rho$ is $C^\infty$ conjugate to $T_{A_f}\times R_\theta$.
\end{enumerate} \end{lemma}
\begin{proof}  Fix $s\geq r$.  Let $G(g_\rho), \Z^c(g_\rho)$,  and $G_0(g_\rho)$ be the groups defined in Section~\ref{sec: start ass} for $g_\rho$.

The proof of Lemma \ref{lemma: Zc circle grp} and ergodicity of $g_\rho$ imply that any $h$ commuting with $g_\rho$ is an isometric extension, $\Z_{s}(g_\rho)$ is virtually $G_0(g_\rho)$,  and $G_0(g_\rho)$ is a group extension of $G(g_\rho)$ by $\Z^c(g_\rho)$. Moreover, by the proof of Proposition~\ref{lemma: prpty G0 G},  $\Z^c(g_\rho)=\{\id\times R_\theta, \theta\in \TT\}$, and $G_0(g_\rho)$ is virtually $\ZZ^\ell\times\Z^c$, where $\ell$ is the rank of the finitely generated abelian group $G(g_\rho)$;  to see this, note that in any short exact sequence of  abelian groups: $0\to H\to G_0\to G\to 0$, with $G$ finitely generated, the group $G_0$ is virtually the product of $H$ with the torsion free part of $G$.

 It is not hard too see that $D{g_\rho}|_{E^s_{g_\rho}}$, $D{g_\rho}^{-1}|_{E^u_{g_\rho}}$ have narrow band spectrum if and only if $Dg|_{E^s_g}$ and $Dg^{-1}|_{E^u_g}$ do. 
Since $f$ has narrow band spectrum, and $f$ is smoothly conjugate to $g_\rho$, both $g$ and $g_\rho$ have narrow band spectrum.  
By Corollary \ref{coro: gl rig toral}, $\Z_{s}(g)$ is either virtually trivial or $g$ is smoothly conjugate to $T_A$. The former implies that $G(g_\rho)$ is virtually trivial, hence $\ell=1$ and item (1) of Lemma \ref{lemma: iso ext classsify} holds.

If $g$ is smoothly conjugate to $T_{A_f}$, then without loss of generality we may assume that $g=T_{A_f}$. If $\rho$ is smoothly cohomologous to a constant $\theta$, then by ergodicity $\theta\notin \QQ/\ZZ$, and item (2) of Lemma \ref{lemma: iso ext classsify} holds.

We claim now that if $g=T_{A_f}$, and $\rho$  is not $C^\infty$ cohomologous to a constant then $\ell=1$ for any $s\geq 1$. 
Suppose $\ell>1$ for some $s\geq 1$.
By taking a finite iterate if necessary, we can assume that  there is an isometric extension $(T_{B})_{\rho_B}$ (a priori $C^s$) such that $(T_{B})_{\rho_B}$ commutes with $g_\rho=(T_{A_f})_\rho$, and the group generated by ${A_f},B$ is not virtually trivial. Using commutativity, by considering the induced action of $(T_{B})_{\rho_B}, (T_{A_f})_\rho$ on $\pi_1(\TT^d)$, we get that $\rho_B$ is cohomologous to a constant, which can be viewed as a function on $\TT^{d-1}$.  By Lemma \ref{lemma: lnr high ran}, the group generated by $T_{A_f}, T_{B}$ on $\TT^n$ is a higher rank action, therefore by Lemma \ref{lemma: R-val cocy rig tori}, $\rho, \rho_B$ are (simultaneously) cohomologous to constants.  By Liv\v sic's theorem the conjugacy is smooth, i.e. $\rho$ is $C^\infty$ cohomologous to a constant, which is a contradiction.
\end{proof}

  
\begin{proof}[Proof of Theorem~\ref{main: dich}]If the disintegration of volume along $\W^c_f$ leaves is not Lebesgue, then Theorem \ref{main: dich} is a corollary of Theorem \ref{main: thm pr}.  Assume that $\mathrm{vol}_{\TT^d}$ has Lebesgue disintegration along $\W^c_f$.  Proposition~\ref{lemma: dich atom rot} implies that one of the following cases holds:

\noindent
\textbf{Case 1:} $f$ is accessible, and the disintegration of $\mathrm{vol}_{\TT^d}$  has a continuous density function on the leaves of $W^c_f$. 
By \cite[Theorem E]{AVW}, there is a volume-preserving flow $\varphi_t$ tangent to and $C^\infty$ along the leaves of $\W^c_f$, commuting with $f$ and satisfying $\varphi_1 = \id$.  Lemma \ref{lemma: Zc circle grp} implies that  $h=\varphi_{\rho(h)}$  for any $h\in \Z^c$, i.e. $\Z^c\subset\{\varphi_t\}_{t\in \TT}$. 
 
Let $D:=\{t\in \TT: \varphi_t\in \Z^c\}$.  There are two possibilities:
\begin{enumerate}
\item{\em $D< \TT$ is discrete.} Then $\Z^c$ is finite.  By Lemma \ref{lemma: rank cent}, the group $G$ is abelian with rank $\ell\leq \ell_0$. 
\begin{enumerate}
\item $\ell<\ell_0$ or $\ell=\ell_0=1$. Since $G_0$ is abelian group extension of $G$ by $\Z^c$, by finiteness of $\Z^c$ we can construct a finite index subgroup of $G_0$ isomorphic to the torsion free part of $G$, which is $\ZZ^\ell$. By the same proof as in Proposition~\ref{lemma: prpty G0 G}, we have that  $\Z_{s}(f)$ is virtually $G_0$, therefore Theorem \ref{main: dich} holds in this case.

\item $\ell=\ell_0>1$.  As in the proof of Theorem \ref{main: thm pr}, we can construct partially hyperbolic elements in all the Weyl chambers of the action of $\Z_0$, which implies that $E^u_f\oplus E^s_f$ is jointly integrable, contradicting the accessiblity of $f$.
\end{enumerate}
 
\item{\em $D< \TT$ is dense.} Lemma \ref{lemma: f narr band} implies that the triple $(f,\varphi_t, X)$ satisfies the hypotheses of Proposition \ref{prop: applic normal form}; applying this result, we obtain that $D=\RR$, $X$ is a $C^\infty$ vector field and so $\varphi_t$ is a $C^\infty$ flow. Therefore $\W^c_f$ is a smooth foliation, and $f$ is smoothly conjugate to   an isometric extension $g_\rho$. Then by Lemma \ref{lemma: iso ext classsify} and accessibility of $f$, item (2) of Theorem \ref{main: dich} holds for $f$.
\end{enumerate}

\noindent
\textbf{Case 2:} $f$ is topologically conjugate to $T_{A_f}\times R_\theta$ for some $\theta\notin \QQ/\ZZ$. Then $E^u_f\oplus E^s_f$ is integrable and tangent to the horizontal foliation $\W^H$. By Lemma \ref{lemma: WH C2}, $\W^H$ is a $C^1$ foliation. 
 
For any $x\in \TT^d$, we denote by $\mathrm{Pr}^c_x$ the projection from $\TT^d$ to $\W^c_f(x)$ along $\W^H$ and let $\mu_x:={\mathrm{Pr}^c_x}_\ast(\mathrm{vol}_{\TT^d})$. Then the family $\{\mu_x, x\in \TT^d\}$ is $f-$invariant, i.e. 
\begin{equation}\label{eqn: finv mux}
(f|_{W^c_f(x)})_\ast \mu_x=\mu_{f(x)}.
\end{equation}
The $C^1-$ness of $\W^H$ implies that  the family of measures $\{\mu_x, x\in \TT^d\}$ along $\W^c_f-$leaves have continuous density functions. 
Therefore $f$ is center $r-$bunched, for all $r>0$, which implies  $\W^c$ has $C^\infty$ leaves, and the stable and unstable holonomies between center leaves are uniformly smooth.   Since $\W^u, \W^s$ have uniformly smooth leaves,  Journ\'e's lemma implies that $\W^H$ has uniformly smooth leaves as well. In summary, $\W^H$ is a smooth foliation.

 Since $\W^c_f$ is absolutely continuous, \cite[Theorem C (1)]{AVW2} implies that there exists a continuous, volume-preserving
flow $\varphi_t$ on $\TT^d$ commuting with $f$ whose generating vector field is tangent to the leaves of $\W^c _f$.  Moreover,  $\varphi_1=\id$ and  $\Z^c\subset\{\varphi_t\}_{t\in \TT}$.

The rest of the proof for Case 2 is similar to that of Case 1. Again we take the set $D:=\{t\in \TT, \varphi_t\in \Z^c\}$, and consider the following cases.
 
\begin{enumerate}
\item{\em $D< \TT$ is discrete.}  Then $\Z^c$ is finite. As in Case 1, we consider the abelian group $G$ which is virtually $\ZZ^\ell, \ell\leq \ell_0$.
\begin{enumerate}
\item $\ell<\ell_0$, or $\ell=\ell_0=1$. then by exactly the same proof as in Case 1 we can prove the conclusion of Theorem \ref{main: dich}.
\item $\ell=\ell_0>1$. First we claim that the action of $\Z_{s}(f)$ on $\TT^d$ is $C^\infty$ (a priori it is only    $C^s$).    For any $g\in \Z_{s}(f)$, $g$ preserves the smooth density on $\W^c$ (induced by $\{\mu_x, x\in  \TT^d\}$). 
Since   $s\geq r >r_0(A)=\max(\frac{\lambda^s}{\mu^s},\frac{\lambda^u}{\mu^u})$,   Lemma~\ref{lem: NBS Stable} implies that  if $f$ is $C^1-$close to $f_0$,  then  $f$  preserves a $C^\infty$ normal form, and  $r(f)<r\leq s$.
Theorem \ref{thm: normal form} then implies that $g$ also preserves the smooth normal form on $\W^u_f$ and $\W^s_f$, which implies that $g$ is uniformly smooth along $\W^s_f$ and $\W^u_f$. Therefore by Journ\'e's lemma, $g$ is uniformly smooth. So the action by $G_0$ is smooth and volume preserving on $\TT^d$. Since $G$ has rank $\ell_0>1$, following the proof of Theorem \ref{main: thm pr}, we can construct partially hyperbolic elements in all the Weyl chambers of the action of $G_0$.  Then the global rigidity result in \cite{DWX} implies that the action of $G_0$ is rigid (see Proposition \ref{glob_rig} in the the Appendix). Thus $f$ is smoothly conjugate to $T_{A_f}\times R_\theta$ for some $\theta\notin\QQ$. 
\end{enumerate}
 
\item{\em $D< \TT$ is dense.}  By the same proof as in Case 1, we obtain that $f$ is smoothly conjugate to an isometric extension.
Then one the two alternatives in Lemma \ref{lemma: iso ext classsify} imply the alternatives (2) and (3) in Theorem \ref{main: dich}. 
\end{enumerate}
 This completes the proof of Theorem~\ref{main: dich}.
 \end{proof}

\appendix
\section{Global rigidity of conservative partially hyperbolic abelian actions on the torus}\label{appendix} We state here the main result in \cite{DWX}, which plays a crucial role in the proof of Theorem \ref{main: dich}. The setting is as follows. Suppose $\al:\ZZ^k\to \Diff_{\mathrm{vol}}^\infty(\TT^d)$ is a smooth, volume preserving ergodic abelian action. We assume that there exists at least one $a\in \ZZ^k$ such that $\al(a)$ is a fibered partially hyperbolic diffeomorphism and all the partially hyperbolic elements of $\al$ preserve a common circle center foliation $\W^c$.

As explained in  Section \ref{sec: HR act}, the distribution $E^H:=E^u_a\oplus E^s_a$ for a partially hyperbolic element $\al(a)$ is $\alpha-$invariant, and we consider the Lyapunov functionals $\chi_i $ and the hyperbolic Weyl chamber picture induced by the cocycle $D\al|_{E^H}$ with respect to $\mathrm{vol}_{\TT^d}$.

\begin{maintheorem}\label{main: gl 1d cent}\cite{DWX} Assume that each hyperbolic Weyl chamber for $\alpha$ contains a partially hyperbolic element. Suppose that  there is no pair of Lyapunov functionals $\chi_i, \chi_j$ and  $c\in (-\infty, \frac{1}{2}]\cup [2,\infty)$ such that $\chi_i=c\chi_j$.  Then $\al$ is smoothly conjugate to the product of an affine Anosov action on $\TT^{d-1}$ with an action by rotations on $\TT^1$.
\end{maintheorem}
We  now verify that Theorem \ref{main: gl 1d cent} applies to the action in the proof of Theorem \ref{main: dich}.

\begin{prop}\label{glob_rig}The action of $G_0$ in (1)(b) of  Case 2 in Section \ref{sec: pf Thm dich} satisfies all the conditions in Theorem \ref{main: gl 1d cent}. Therefore $G_0$ is smoothly conjugate to a product of a linear Anosov action on $\TT^{d-1}$ and a rotation action on $\TT^1$. In particular, in (1)(b) of Case 2  in Section \ref{sec: pf Thm dich}, $f$  is smoothly conjugate to $A_f\times R_\theta$ for some $\theta\notin \QQ/\ZZ$.
\end{prop}
\begin{proof} Recall that in (1)(b) of Case 2 in Section \ref{sec: pf Thm dich}, we obtain that the action of $G_0$ on $\TT^d$ is abelian, $C^\infty$ and volume preserving. Every element $h$ in $G_0$ preserves the common center foliation $\W^c$ and there is a H\"older continuous fiber bundle $\pi:\TT^d\to \TT^{d-1}$ such that for any $h\in G_0$, there is a linear automorphism $T_{A_h}:\TT^{d-1}\to \TT^{d-1}$ satisfying $\pi\circ h=T_{A_h}\circ \pi$.

Since  $G=\{A_h, h\in G_0\}$ has rank $\ell_0>1$, by Lemma \ref{lemma: za max} $G$ induces a maximal Anosov affine action on $\TT^{n-1}$. Therefore, the action of $G$ is TNS (i.e., there are no negatively proportional Lyapunov functionals) and conformal on each coarse Lyapunov foliation. By the discussion in Section \ref{sec: same picture}, the Lyapunov functionals of the action of $G_0$ are close to that of $G$ (see also Step 7. of Section \ref{sec: key sec}); therefore the action of $G_0$ is TNS, and the Lyapunov functionals satisfy the $\frac{1}{2}-$pinching condition. Moreover, by following the proof of Theorem \ref{main: thm pr}, we can construct partially hyperbolic elements in all the Weyl chambers of the action of $G_0$. Thus the action of $G_0$ in (2)(b) of Case 2 of Section \ref{sec: pf Thm dich} satisfies all the conditions in Theorem \ref{main: gl 1d cent}, and by Theorem \ref{main: gl 1d cent} it is globally rigid. 
\end{proof}

\end{document}